\documentclass[11pt]{amsart}
\usepackage{amsmath,amsthm,amssymb,mathtools,xcolor,mathrsfs}
\usepackage[final]{microtype}
\usepackage{hyperref}
\hypersetup{pdfpagemode={UseOutlines},	bookmarksopen, pdfstartview={FitH},	colorlinks,	linkcolor={black},	citecolor={black},	urlcolor={black}}
\usepackage[top=3.5cm, bottom=3.5cm, left=3cm, right=3cm]{geometry}

\newtheorem{thm}{Theorem}[section]
\newtheorem{corollary}[thm]{Corollary}
\newtheorem{proposition}[thm]{Proposition}
\newtheorem{theorem}[thm]{Theorem}
\newtheorem{lemma}[thm]{Lemma}

\theoremstyle{definition}
\newtheorem{definition}[thm]{Definition}
\newtheorem{example}[thm]{Example}
\theoremstyle{remark}
\newtheorem{remark}[thm]{Remark}
\numberwithin{equation}{section}
\newcommand{\RR}{\mathbb{{R}}}
\newcommand{\ZZ}{\mathbb{{Z}}}
\newcommand{\NN}{\mathbb{{N}}}
\newcommand{\TT}{\mathbb{{T}}}
\newcommand{\CC}{\mathbb{{C}}}

\newcommand{\ran}{\operatorname{ran}}

\newcommand{\linearOp}{\mathcal{L}}

\newcommand{\SC}{\mathcal{SC}}
\newcommand{\SQC}{\mathcal{SQC}}
\newcommand{\e}{e}
\newcommand{\maxone}[1]{\max\left\{1,#1\right\}}
\newcommand{\minone}[1]{\min\left\{1,#1\right\}}
\newcommand{\finiteN}{\mathcal N}

\DeclareMathOperator*{\motimes}{\text{\raisebox{0.25ex}{\scalebox{0.7}{$\bigotimes$}}}}

\newcommand\Real{{\mathfrak R}{\mathfrak e}\,} %

\title[Similarity to contraction semigroups on infinite tensor products]{On similarity to contraction semigroups and tensor products, II: Infinite tensor products}
\author{J.\ Oliva-Maza}
\thanks{The first author has been supported by XXXV Scholarships for Postdoctoral Studies, Ram\'on Areces Foundation, by Project PID2022-137294NB-I00, DGI-FEDER, of the MCYTS and Project E48-23R, D.G. Arag\'on, Universidad de Zaragoza, Spain.}
\author{Y.\ Tomilov}
\thanks{The second author was supported by the NCN Opus grant UMO-2023/49/B/ST1/01961, by the NAWA/NSF grant BPN/NSF/2023/1/00001 and the NCN Weave-Unisono grant 2024/06/Y/ST1/00044.}
\subjclass[2020]{Primary 47D03, 47A65; Secondary 47A20, 47A80}
\keywords{infinite tensor products, semigroups, similarity, contractions, Hilbert spaces}
\date{}

\begin{document}

\begin{abstract}
We develop a framework for infinite tensor products of Hilbert spaces, operators, and semigroups tailored to questions of similarity to contraction semigroups. On the operator-theoretic side, we give a systematic treatment of incomplete infinite tensor products, including criteria for existence, non-vanishing, and continuity properties of the associated tensor product semigroups. On the semigroup-theoretic side, we prove a low-regularity similarity theorem showing that global quasi-contractive control together with local contractive information at one positive time implies similarity to a contraction semigroup, with explicit bounds on the similarity constant. These ingredients are then combined to obtain infinite analogues of the finite tensor-product splitting principle for similarity to contraction and quasi-contraction semigroups.  We also clarify the role of complete tensor products and show, in particular, that whenever a complete infinite tensor product of semigroups is a \(C_0\)-semigroup, it decomposes along incomplete tensor-product components.
\end{abstract}

\maketitle
\tableofcontents

\section{Introduction}

Contractions on Hilbert spaces form one of the central and best understood classes of operators
in analysis. Their structure is intertwined with dilation theory, functional models, and a broad
family of norm estimates. For this reason, the problem of understanding \emph{similarity} to
contractions has occupied operator theory for a long time: one asks when an operator, or a
semigroup of operators, can be reduced by an invertible change of Hilbertian norm to a
contractive one. Such a reduction preserves the algebraic and spectral dynamics while moving
the problem into a substantially more tractable analytic framework. In the semigroup setting
this is particularly attractive, since contraction semigroups are one of the basic outputs of the
Lumer--Phillips theorem and therefore form a natural class in the abstract treatment of evolution
equations.

If $H$ is a Hilbert space and $\mathcal T=(T(t))_{t\ge 0}\subset \linearOp(H)$ is a semigroup, we say that
$\mathcal T\in \SC(H)$ if there exist an invertible operator $R\in \linearOp(H)$ and a contraction semigroup
$\mathcal{S}=(S(t))_{t\ge 0}$ on $H$ such that
\[
T(t)=R^{-1}S(t)R,\qquad t\ge 0.
\]
Equivalently, there exists an equivalent Hilbertian norm $\|\cdot\|_{\rm eq}$ on $H$ such that
\[
\|T(t)h\|_{\rm eq}\le \|h\|_{\rm eq},\qquad h\in H,\ t\ge 0.
\]
In other words, similarity of a semigroup to a contraction semigroup is equivalent to the
existence of an equivalent inner product on $H$ whose associated norm makes the whole semigroup
contractive. This reformulation is fundamental for the paper and will be used repeatedly,
often implicitly, in both directions. The equivalence is basic, but it already underlines the
central difficulty: similarity for a semigroup is a \emph{joint} property of the whole family
$(T(t))_{t\ge 0}$. In contrast to the
discrete case, similarity of each individual operator $T(t)$ to a contraction does not, in
general, imply that the semigroup itself belongs to $\SC(H)$. The primary quantitative object in
this paper is therefore the similarity constant
\begin{align*}
	\mathcal C(\mathcal T) := \inf\Bigl\{\|R\|\,\|R^{-1}\| \, :\, R\in \linearOp(H)\ \text{invertible}, \,
	\|RT(t)R^{-1}\|\le 1\ \text{for all } t\ge 0 \Bigr\}.
\end{align*}
We set $\mathcal C(\mathcal T)=\infty$ if $\mathcal T\notin \SC(H)$. If $\mathcal T \in \mathcal{SC}(H)$, then the infimum above is attained; see e.g.
\cite[Section~2]{OlivaMazaTomilovSimilarity}. Thus the problem is not only
whether a contractive renorming exists, but whether it can be chosen uniformly for all times.

If $\|\cdot\|_\sharp$ is an equivalent Hilbertian norm on $H$, we denote by
\[
\mathcal C_H(\|\cdot\|_\sharp)
=
\inf\Bigl\{MN:\ \|h\|\le M\|h\|_\sharp,\ \|h\|_\sharp\le N\|h\|\ \text{for all }h\in H\Bigr\}.
\]
its distortion constant with respect to the original Hilbert norm. Then, if $R$ is any bounded and invertible operator such that $\|h\|_\sharp=\|Rh\|$ for all $h\in H$, we have that
\[
\mathcal C_H(\|\cdot\|_\sharp)
= \|R\|\,\|R^{-1}\|.
\]
Hence,
\[
\mathcal C(\mathcal T)
=
\inf\Bigl\{
\mathcal C_H(\|\cdot\|_\sharp):
\|T(t)h\|_\sharp\le \|h\|_\sharp\ \text{for all }h\in H,\ t\ge 0
\Bigr\}.
\]
Thus the similarity constant of a semigroup is obtained by minimizing the distortion of all
equivalent Hilbertian norms that make the semigroup contractive.

A natural enlargement of $\SC(H)$ is provided by quasi-contraction semigroups. Recall that a
semigroup $\mathcal{T}=(T(t))_{t\ge 0}$ is called a quasi-contraction semigroup if there exists
$\lambda\in\mathbb R$ such that
\[
\|T(t)\|\le e^{\lambda t},\qquad t\ge 0,
\]
or equivalently, if the rescaled semigroup $(e^{-\lambda t}T(t))_{t\ge 0}$ is contractive. We
denote by $\SQC(H)$ the class of semigroups similar to quasi-contraction semigroups. For
$d\in\mathbb R$ and a semigroup $\mathcal T=(T(t))_{t\ge 0}$, we write
\[
e_d\mathcal T := (e^{dt}T(t))_{t\ge 0}.
\]
This class arises naturally throughout semigroup theory and PDE, but in the present context it plays
a more structural role: it separates the global growth of the semigroup from the local
contractive information carried by one operator $T(\tau)$, $\tau>0$.

The present paper continues two earlier lines of work. In
\cite{OlivaMazaTomilovSimilarity}, semigroups similar to contraction semigroups were studied
from a structural point of view. One of the main outcomes was that, in the $C_0$-case,
similarity to a contraction semigroup can be recovered by combining two pieces of
information: a quasi-contractive renorming of the whole semigroup and a contractive
renorming of one operator $T(\tau)$ at a single positive time. 
In \cite{OlivaMazaTomilovTensorI}, finite tensor products of
semigroups were considered, and a
splitting phenomenon was established: after a suitable balancing rescaling, similarity of a
tensor product semigroup to a contraction semigroup is equivalent to the corresponding
similarity of its factors. For a more exhaustive discussion of the problem of similarity to contraction semigroups,
including the current state of the art and its motivations, we refer to
\cite{OlivaMazaTomilovSimilarity,OlivaMazaTomilovTensorI}. For broader operator-theoretic background on similarity problems and
completely bounded methods, we refer to the monographs
\cite{paulsen2002completely,pisier2001similarity}. For classical dilation-theoretic and model-theoretic
background in the single-operator setting, we refer to
\cite{nikolski2002operators2,sznagy2010harmonic}. On the semigroup side, antecedents
relevant to the present circle of ideas include the dilation and similarity results for
Hilbert-space semigroups in \cite{arendt2001functional,arhancet2017isometric,chernoff1976two,lemerdy1996dilation,lemerdy1998similarity} and in \cite[Chapter~7]{haase2007functional}.

The finite tensor-product results of \cite{OlivaMazaTomilovTensorI} may be summarized
schematically as follows.  If \(\mathcal T_n=(T_n(t))_{t\geq 0}\) are semigroups on
Hilbert spaces \(H_n\), then,
after a suitable balancing of scalar
exponential factors, similarity of the tensor product semigroup $\otimes_{n=1}^{N} \mathcal T_n$
to a
contraction semigroup on \( \motimes_{n=1}^{N} H_n \) splits into the corresponding similarity properties of
the factors:
\[
\motimes_{n=1}^{N} \mathcal T_n \in \mathcal{SC}\left(\motimes_{n=1}^N H_n\right)
\]
if and only if
\[
\exists d_1,\ldots,d_N\in\mathbb R,\quad
\sum_{n=1}^{N}d_n=0,
\quad e_{d_n}\mathcal T_n\in \mathcal{SC}(H_n)
\quad (1\leq n\leq N).
\]
The condition \(\sum_{n =1}^N d_n=0\) expresses the fact that the scalar rescalings cancel on
the tensor product.  Thus the finite theorem is not merely a permanence
result.  It says that the tensor-product structure is rigid enough to force
similarity to split.

This rigidity is special to tensor products. It fails for general products
of commuting semigroups, and this makes the tensor-product case a useful
testing ground for similarity phenomena.
Moreover, the finite theorem already points towards the infinite problem in
two ways.  First, its proof is based on changing the Hilbertian geometry on
each factor and then tensoring the resulting norms.  For finitely many factors
this creates no additional restriction, but for countably many factors the
renormings themselves must be tensorable.  Thus the similarity constants can
no longer be treated as harmless finite constants; their infinite product
becomes part of the structure.

Second, the balancing parameters which are finite-dimensional in \cite{OlivaMazaTomilovTensorI}
become an infinite sequence.  The finite condition
\(
\sum_{n=1}^{N} d_n=0
\)
has to be replaced by an absolutely summable cancellation condition
\[
(d_n)_{n\in\mathbb N}\in\ell^1,
\qquad
\sum_{n\in\mathbb N}d_n=0.
\]
Thus the expected infinite analogue cannot merely require
\[
e_{d_n} \mathcal T_n\in \mathcal{SC}(H_n)
\qquad (n\in\mathbb N).
\]
It must also contain quantitative convergence conditions, of the form
\[
(d_n)_{n\in\mathbb N}\in\ell^1,\qquad
\sum_{n\in\mathbb N} d_n=0,\qquad
\prod_{n\in\mathbb N} \mathcal C(e_{d_n} \mathcal T_n)<\infty .
\]



There is, however, a second difficulty which has no finite counterpart.
Before one can ask whether the infinite tensor product of semigroups
\[
\motimes_{n\in\mathbb N} \mathcal T_n
\]
is similar to a contraction semigroup on $\motimes_{n\in\mathbb N} H_n$,
one must know whether this infinite tensor product is well defined and has
the required continuity properties. The same problem arises for incomplete
infinite tensor products $\motimes_{n\in\mathbb N}^{\boldsymbol{y}} H_n$, defined in terms
of the reference vectors \(\boldsymbol{y}=(y_n)_{n\in\mathbb N}\), and for the corresponding
semigroups
\[
\motimes_{n\in\mathbb N}^{\boldsymbol{y}} \mathcal T_n.
\]

At the operator level, as we show below, the basic conditions involve both
norm control and compatibility with \(\boldsymbol{y}\):
\[
\prod_{n\in\mathbb N}\maxone{\|T_n(t)\|}<\infty,
\qquad
\sum_{n\in\mathbb N}
\bigl|1-\langle T_n(t)y_n,y_n\rangle\bigr|<\infty,
\qquad t>0.
\]
Thus incomplete tensor products are governed by two independent mechanisms:
growth of the factors, and preservation, up to a summable error, of the
chosen reference vector.

The classical theory of infinite tensor products provides the Hilbert-space
background.  The basic definitions of complete and incomplete infinite tensor
products go back to von Neumann~\cite{von1939infinite}, while the incomplete
setting was further elaborated in
\cite{guichardet1966produits,guichardet1969tensor}.
Related constructions of incomplete tensor products can be found, for example,
in \cite{berezansky1995spectral,rudol1992conditionally}, and
\cite{ng2013genuine} offers, in particular, an alternative approach to the
classical theory.  For tensor products of operators,
\cite{nakagami1970infinite,nakagami1974infinite} provide the classical
complete-space framework, while \cite{gill1978infinite} discusses related
Banach-space aspects, and \cite{krajczok2024examples} addresses products of
unbounded operators.  Thus the subject is classical and has appeared in
several different forms.  However, the available results do not provide a
systematic framework for the questions considered here, especially in the
incomplete infinite tensor-product setting.  They either cover rather specific
situations or are not in a form suited to the operator- and semigroup-theoretic
purposes of the present paper.

With few exceptions, notably the Banach-space approach of
\cite{gill1978infinite}, the literature on infinite products of operators and
semigroups remains largely confined to very restrictive settings.  In
particular, infinite products of unitary groups, and more generally infinite
tensor products of unitary representations, have received considerable
attention; see, for instance,
\cite{aita1997amenable,badea2017kazhdan,bedos2004infinite,
	bergmann2003asymptotic, bergmann2003induced, Kraus, reents1974infinite}.
However, the unitary case does not involve the norm-growth phenomena which
are central for general bounded operators and semigroups.  We return to these
points in Sections~\ref{sec:tensor-def} and~\ref{sec:incomplete-operators-semigroups}.

For the purposes of this paper, the missing theory has three parts: existence
and non-vanishing for incomplete tensor products of operators, regularity for
incomplete tensor product semigroups, and the low-regularity phenomena which
arise for semigroups on complete tensor products.
In particular, pathologies occur already at a basic level:
an incomplete tensor product $\otimes_{n\in\NN}^{\boldsymbol{y}}A_n$ of bounded operators $(A_n)_{n\in\NN}$ may collapse to the zero
operator, and a complete infinite tensor product $\otimes_{n\in\NN} \mathcal T_n$  of semigroups $(\mathcal T_n)_{n\in\NN}$ may fail to be
strongly measurable even when $\mathcal T_n$ extend to uniformly continuous unitary
groups, as illustrated in Section~\ref{sec:incomplete-operators-semigroups}.  For this reason, the infinite theory is
not simply the finite theory with more indices.  It requires a framework in
which incomplete tensor products of Hilbert spaces, operators, and semigroups
are treated simultaneously, and in which similarity estimates are
quantitative enough to survive tensoring.

Moreover, the relevance of infinite tensor products goes beyond the present
semigroup setting.  They are classical tools in operator algebra theory,
especially in the study of von Neumann algebras and their representations;
see, for instance, Takesaki~\cite{Takesaki} or
Kadison--Ringrose~\cite{Kadison}.  They also arise naturally in questions of
self-adjointness and spectral theory
\cite{berezansky1995spectral, Reed1970SelfAdjointness}, in the theory of
reproducing kernels \cite{gnewuch2022countable, gnewuch2024infinite}, in the representation
theory of infinite quantum systems \cite{Bratteli}, in \(E_0\)-semigroup and
product system theory \cite{arveson2001infinite,Floricel}, and 
in constructions from mathematical physics, including quantum general
relativity, see, for instance, \cite{sahlmann, ThiemannWinkler2001}. The references cited here are meant
as representative examples rather than as an exhaustive account.

The novelty of the paper lies in combining a low-regularity similarity
theorem for a single semigroup with a systematic operator/semigroup theory
of incomplete infinite tensor products, and in using this combination to
obtain infinite splitting theorems within and  beyond the classical \(C_0\)-framework.

%

\subsection{Main results}

The logic of the paper is as follows. Although the infinite tensor-product theorems are the main
target, the first main conceptual result has to be an abstract low-regularity similarity theorem.
Only after such a theorem is available can one treat infinite tensor products in a satisfactory
generality. In particular, the similarity theorem is not parallel to the tensor-product part,
nor merely complementary to it: it is one of the basic tools that make the latter possible.

Our first main theorem is therefore a similarity theorem for a single semigroup under low
regularity assumptions. Let $\mathcal{T}=(T(t))_{t\ge 0}\subset \linearOp(H)$ be a semigroup on a Hilbert space
$H$ (not assumed measurable). Suppose that there exist equivalent Hilbertian norms
$\|\cdot\|_1$, $\|\cdot\|_2$ on $H$, constants $C_1,C_2\ge 1$, and parameters
$\lambda,\tau>0$ such that
\[
\|h\|\le \|h\|_j\le C_j\|h\|,\qquad h\in H,\ j=1,2,
\]
\[
\|T(t)\|_1\le e^{\lambda t},\qquad t\ge 0,
\qquad\text{and}\qquad
\|T(\tau)\|_2\le 1.
\]
Thus the whole semigroup is quasi-contractive in one Hilbertian geometry, while one time slice
is contractive in another. Theorem~\ref{Th:eqNorm} shows that these two pieces of information can be
combined into joint similarity to contractions:
\[
\mathcal T\in \SC(H).
\]
Moreover, the resulting similarity constant is controlled explicitly. More precisely,
Proposition~\ref{Prop:simConstant} yields the estimate
\[
\mathcal C(\mathcal T) \le C_1\,\frac{\sinh (2\lambda)}{\lambda}
+ 2\sqrt{2}\,C_2\maxone{\sqrt{\tau}}
\sup_{t\in[0,\maxone{\tau}]}
\|T(t)\|^2,
\]
with a corresponding asymptotic simplification. In this form, the result may be viewed as a
low-regularity version of the main criterion from
\cite{OlivaMazaTomilovSimilarity}: the conclusion is the same, but the hypotheses are adapted
to the irregular setting that naturally appears for infinite tensor products.


The second group of results concerns the operator- and semigroup-theoretic
structure of infinite tensor products themselves.  These results are independent
of the similarity theorem in spirit, but indispensable for the infinite splitting
arguments.
We first recall the complete and
incomplete infinite tensor products of Hilbert spaces in a form adapted to
countable tensor products, and we record how finite tensor products embed
as special incomplete infinite tensor products.  Thus the finite tensor-product
splitting theorem from \cite{OlivaMazaTomilovTensorI} is not an external object but sits naturally
inside the framework developed here.

At the Hilbert-space level, we prove a tensorization result for equivalent
Hilbertian norms.  If each \(H_n\) is renormed by an equivalent Hilbertian
norm with distortion \(C_n\), then the corresponding complete and incomplete
tensor norms are equivalent precisely under the expected product condition
\[
\prod_{n\in\mathbb N} C_n <\infty,
\]
and the resulting tensor distortion is equal to this product.  This
elementary-looking fact is one of the quantitative mechanisms behind the
later similarity arguments: it explains why convergence of products of
similarity constants appears naturally in the infinite theory.

We then develop the operator theory of incomplete infinite tensor products.
For a sequence of bounded operators \(A_n\in\mathcal L(H_n)\) and a reference
\(C_0\)-sequence \(\boldsymbol{y}=(y_n)_{n\in\mathbb N}\), the central question is whether
\[
\motimes_{n\in\mathbb N}^{\boldsymbol{y}} A_n
\]
defines a bounded nonzero operator on
\[
\motimes_{n\in\mathbb N}^{\boldsymbol{y}} H_n .
\]
The answer involves two independent requirements: control of the upper norm
growth of the factors and compatibility with the reference component.  In
its simplest form these conditions are reflected by
\[
\prod_{n\in\mathbb N}\max\{1,\|A_n\|\}<\infty,
\qquad
\sum_{n\in\mathbb N}
\bigl|1-\langle A_n y_n,y_n\rangle\bigr|<\infty .
\]
The results in Section~\ref{sec:incomplete-operators-semigroups} make this precise, distinguish the factorwise and
strong-limit constructions, and show in particular that these two natural
ways of forming incomplete tensor products of operators are not equivalent
in general.  They also relax earlier sufficient conditions from the existing
literature and clarify the possibility that an incomplete product may exist
but collapse to the zero operator.

Passing from operators to semigroups gives another part of the tensor-product
theory.  Using the operator criteria, we characterize when
\[
\motimes_{n\in\mathbb N}^{\boldsymbol{y}} \mathcal T_n
\]
is a well-defined nonzero semigroup on
\[
\motimes_{n\in\mathbb N}^{\boldsymbol{y}} H_n .
\]
We then show that the usual regularity properties behave well in the
incomplete setting: weak measurability, strong measurability, and
\(C_0\)-continuity are inherited from the factors under the natural
nondegeneracy and local boundedness assumptions.  In particular, in the
Hilbert-space setting the strong continuity of the incomplete tensor product
does not require an additional summability condition of the form
\[
\sum_{n\in\mathbb N}
\sup_{t\in[0,\tau]}\|T_n(t)y_n-y_n\|<\infty,
\]
which is often imposed in earlier treatments.  Thus incomplete tensor
products provide a robust environment for semigroup constructions beyond
the unitary case.

Finally, we analyze complete infinite tensor products of semigroups.  This
part has a rather different flavour.  Weak measurability is preserved by the
complete tensor product construction, but strong measurability and
\(C_0\)-regularity are much more rigid.  We show that if a complete tensor
product semigroup is strongly continuous, then its action is forced to
respect the decomposition of the complete tensor product into incomplete
tensor-product components.  More precisely, each incomplete component is
invariant under the complete \(C_0\)-semigroup, and the complete product
decomposes as the orthogonal direct sum of the corresponding incomplete
tensor product semigroups.  Hence, in the \(C_0\)-case, the complete theory
does not replace the incomplete theory; rather, it reduces structurally to it.

With this tensor-product framework in place, the first application to
similarity is the infinite analogue of the splitting theorem from the finite
tensor paper.  Let \(\mathcal T_n=(T_n(t))_{t\geq0}\) be semigroups on Hilbert spaces
\(H_n\), let \(\boldsymbol{y}=(y_n)_{n\in\mathbb N}\) be a \(C_0\)-sequence, and assume that
the incomplete tensor product semigroup
\[
\motimes_{n\in\mathbb N}^{\boldsymbol{y}} \mathcal T_n
\]
is well-defined and nonzero.  Theorem~\ref{thm:main_incomplete} gives a two-sided splitting result,
but with asymmetric quantitative conclusions.

In the sufficient direction, if each \(\mathcal T_n\) is locally bounded on
\((0,\infty)\), and if there exists a real sequence \((d_n)_{n\in\mathbb N}\)
such that
\[
\sum_{n\in\mathbb N}|d_n|<\infty,
\qquad
\sum_{n\in\mathbb N}d_n=0,
\]
\[
e_{d_n} \mathcal T_n\in \mathcal{SC}(H_n)
\quad (n\in\mathbb N),
\qquad
\prod_{n\in\mathbb N} \mathcal C(e_{d_n}\mathcal T_n)<\infty,
\]
then
\[
\motimes_{n\in\mathbb N}^{\boldsymbol{y}} \mathcal T_n
\in
\mathcal{SC}\!\left(
\motimes_{n\in\mathbb N}^{\boldsymbol{y}} H_n
\right).
\]
Here the condition \(\sum_{n\in\NN} d_n=0\) is the infinite balancing condition:
the scalar rescalings cancel on the tensor product, while
\(\sum_n|d_n|<\infty\) ensures that the balancing itself is compatible with
the infinite product.

Conversely, if
\[
\motimes_{n\in\mathbb N}^{\boldsymbol{y}} \mathcal T_n
\in
\mathcal{SC}\!\left(
\motimes_{n\in\mathbb N}^{\boldsymbol{y}}H_n
\right),
\]
then there exists an absolutely summable balancing sequence
\((d_n)_{n\in\mathbb N}\) with
\[
\sum_{n\in\mathbb N}d_n=0
\]
such that
\[
e_{d_n} \mathcal T_n\in \mathcal{SC}(H_n)
\quad (n\in\mathbb N),
\qquad
\sup_{n\in\mathbb N} \mathcal C(e_{d_n} \mathcal T_n)<\infty .
\]
Thus similarity of the infinite tensor product forces similarity of all
suitably balanced factors, with uniform control of the corresponding
similarity constants.  The stronger product condition in the sufficient
direction is the natural condition for tensoring the individual equivalent
Hilbertian norms; it is not asserted to be necessary in this generality.

Thus the splitting phenomenon from the finite theory survives in the
infinite setting, but only after one takes into account both the summability
of the balancing rescaling and the convergence of the renormings. 
A particularly striking feature is that no measurability or strong continuity assumption on the
individual semigroups is imposed.
Apart from the nondegeneracy of the
tensor product and the stated local boundedness assumption, the theorem is
formulated at a level of generality tailored to the intrinsic low-regularity
nature of infinite tensor products.

The quasi-contractive counterpart is described in Theorem~\ref{continuousInfiniteCor11}.  In that
case the zero-sum balancing condition is no longer part of the statement,
because an overall exponential factor is allowed.  More precisely, if there is
an absolutely summable sequence \((d_n)_{n\in\mathbb N}\) such that
\[
e_{d_n} \mathcal T_n\in \mathcal{SC}(H_n)
\quad (n\in\mathbb N),
\qquad
\prod_{n\in\mathbb N} \mathcal C(e_{d_n} \mathcal T_n)<\infty,
\]
then
\[
\motimes_{n\in\mathbb N}^{\boldsymbol{y}} \mathcal T_n
\in
\mathcal{SQC}\!\left(
\motimes_{n\in\mathbb N}^{\boldsymbol{y}}H_n
\right).
\]
Conversely, if the incomplete tensor product semigroup belongs to
\(\mathcal{SQC}\), then such an absolutely summable rescaling exists with
\[
e_{d_n} \mathcal T_n\in \mathcal{SC}(H_n)
\quad (n\in\mathbb N),
\qquad
\sup_{n\in\mathbb N} \mathcal C(e_{d_n}\mathcal T_n)<\infty .
\]
Thus Theorem~\ref{continuousInfiniteCor11} gives the quasi-contractive analogue of Theorem~\ref{thm:main_incomplete},
with the same distinction between product control in the sufficient direction
and uniform control in the converse direction.

The link between the two main results is vital. Theorem~\ref{thm:main_incomplete} is not proved by a direct tensor
argument alone. Its proof requires one to extract similarity of the individual factors from the
similarity of the whole product, and then to convert suitable quasi-contractive information on
each factor into genuine similarity to contractions with estimates that are uniform enough to be
tensored again. This is precisely where the low-regularity similarity theorem enters. Theorem~\ref{Th:eqNorm}
and Proposition~\ref{Prop:simConstant} provide the mechanism that turns local contractive information together
with global exponential control into quantitative similarity bounds. Those bounds are then fed
back into the infinite tensor-product argument through tensorization of equivalent Hilbertian norms and the operator-theoretic
structure of incomplete tensor products. In this sense, the theorem for single semigroups is not
just a preliminary result, but one of the engines of the infinite tensor-product theory.

Although the incomplete tensor-product setting is the natural one for semigroups, we also
address the complete tensor product. Here the situation is subtler because regularity may fail in
a rather dramatic way. The complete theory therefore plays a complementary and partly cautionary role: it clarifies
which parts of the tensor-product picture persist on the complete infinite tensor product, and which features
force one back to the incomplete setting in semigroup theory.

\subsection{Methods and organization}

The methods of the paper combine three ingredients that, although closely related in the final
argument, come from rather different directions.

The first ingredient is the geometry of infinite tensor products themselves. At the level of
Hilbert spaces, equivalent norms can be tensored provided the corresponding infinite product
of distortion constants is finite. Quantitatively, if $\|\cdot\|_{H_n}$ is an equivalent Hilbertian norm on
$H_n$ with
\[
\|h\|\le \|h\|_{H_n}\le C_n\|h\|,\qquad h\in H_n,
\]
and if $\prod_{n\in \NN} C_n<\infty$, then the corresponding complete and incomplete tensor norms are
again equivalent, with distortion equal to $\prod_{n\in \NN} C_n$. This simple tensorization principle lies behind all
subsequent tensorization arguments. At the same time, the operator theory developed later in the paper shows that tensor
products on incomplete spaces are governed by a
mixture of norm control and compatibility with the reference sequence $\boldsymbol{y}$. The arguments here
are elementary in spirit but rather delicate in bookkeeping, because one has to separate the
possibility of existence from the possibility of collapse to the zero operator.

This tensorization viewpoint is closely related to 
earlier renormalization
procedures for infinite products.  
In the work \cite{arendt1998infinite} on infinite products of commuting \(C_0\)-semigroups, one
introduces a change of speed
\[
T_n(t)\mapsto T_n(\lambda_n t),
\]
or, at the generator level, \(A_n\mapsto \lambda_n A_n\), in order to make
the infinite product converge.  Thus there is a time scaling of
the individual semigroups.  By contrast, 
renormalization of infinite
direct products of unitary one-parameter groups in \cite{reents1974infinite} is additive at the level of
generators:
\[
U_n(t)=e^{itH_n}
\mapsto
e^{-it\alpha_n}U_n(t)
=
e^{it(H_n-\alpha_n I)} .
\]
The rescalings used in the present paper are closer to this second mechanism:
\[
T_n(t)\mapsto e^{d_nt}T_n(t).
\]
At the generator level, this corresponds formally to replacing \(A_n\) by
\(A_n+d_nI\).  The condition
\(\sum_{n\in \NN} d_n=0\) expresses cancellation of the scalar renormalizations on
the tensor product, while \(\sum_{n\in \NN} |d_n|<\infty\) ensures that the balancing
itself is compatible with the infinite tensor product.  What is specific to
the similarity problem is the additional quantitative requirement that the
renormalized Hilbertian geometries tensor, 
reflected in conditions such as
\(
\prod_{n\in \NN} \mathcal  C(e_{d_n} \mathcal T_n)<\infty .
\)

The second ingredient is semigroup theory on incomplete tensor products. Once the operator
level is understood, one can pass to semigroups and show that the relevant local boundedness,
measurability, and continuity properties survive tensorization. This part of the paper may be
viewed as foundational. It prepares the ground for the similarity results and also clarifies why
incomplete tensor products are the right environment for semigroups: they preserve enough
regularity to support semigroup arguments, yet they are flexible enough to contain genuinely
infinite constructions that are invisible in the usual $C_0$-theory on complete tensor spaces.

The third ingredient is the new proof of the similarity theorem for single semigroups. The idea
is to combine a specific exponential upper bound with a contractive quotient obtained by
invariant averaging. The two assumptions enter the construction in genuinely different ways:
the quasi-contractive estimate controls the size of finitely supported orbit representations
directly, whereas the contractive information at one fixed time produces, after averaging over
one period, the quotient semigroup that carries the contractive part of the argument. After
dilating the quotient semigroup, one defines a Hilbert seminorm on finitely supported orbit
representations and then passes to the induced quotient norm on the original space.
Concretely, the proof works with finitely supported expressions
\[
f=\sum_{j=1}^n \delta_{t_j}\otimes h_j,
\qquad
h=\sum_{j=1}^n T(t_j)h_j,
\]
and builds the contractive norm through an infimum over such representations, combining an
$\alpha$-part controlled by the exponential estimate and a $\beta$-part coming from the
contractive quotient. This construction is flexible enough to survive under weak regularity and,
more importantly for the present paper, explicit enough to yield similarity-constant estimates
that can later be tensored. 

The organization of the paper reflects this logic. After the introduction, we first isolate the invariant-mean formalism on finite intervals that underlies the averaging arguments used later. We then develop
the low-regularity similarity theorem for a single semigroup, providing the explicit equivalent norm and an upper bound for the similarity constant. Only after this abstract machinery is in place do we return to
infinite tensor products themselves: we collect the necessary numerical preliminaries on
unconditional infinite products, recall the complete and incomplete tensor products of Hilbert
spaces, and study the corresponding tensor products of operators and semigroups, with emphasis
on the incomplete setting. On that basis we prove the main splitting theorems for similarity to
contractions and quasi-contractions of infinite tensor products. We also provide a simple
counterexample illustrating the role of the local boundedness assumptions, showing that their necessity already appears in a very elementary finite-dimensional setting. More broadly, we include examples clarifying both the scope and the limits of the developed framework. 
We also provide a counterpart for infinite tensor products of operators of one of the main similarity results proved earlier for semigroups.
\bigskip

\noindent \textbf{Notation}. We fix here some notation that will be used throughout the paper. Given a Hilbert space $H$, $\linearOp(H)$ denotes the algebra of bounded operators on $H$. $I_H$ denotes the identity operator on $H$, or simply $I$ when the choice of space is apparent.

For $\lambda \in \CC$, $e_\lambda$ denotes the exponential function $t \mapsto e^{\lambda t}$, and its domain will typically be clear from context. Consequently, given a semigroup $\mathcal T = (T(t))_{t\geq0}$ of bounded operators and $\lambda \in \CC$, $e_{\lambda} \mathcal T$ denotes the rescaled semigroup $(e^{\lambda t} T(t))_{t\geq0}$. 

As is standard in the literature, and with a slight abuse of notation, we denote by $\|T\|$ the norm of a bounded operator $T$ on a Hilbert space equipped with a fixed norm $\|\cdot\|.$ This convention avoids unnecessarily cumbersome notation, especially when the same space is endowed with several norms.

Given $N \in \NN$ and Hilbert spaces $H_1, \ldots, H_N$, we denote by $\motimes_{n=1}^N H_n$ their Hilbert tensor product. Similarly, given bounded operators $A_n \in \linearOp(H_n)$, or semigroups of operators $\mathcal T_n = (T_n(t))_{t\geq0} \subset \linearOp(H_n)$, $n=1,\ldots,N$, we denote their tensor products by $\motimes_{n=1}^N A_n$ and $\motimes_{n=1}^N \mathcal T_n$, respectively. For a sequence of Hilbert spaces $(H_n)_{n\in \NN}$, we use the notation $\times_{n\in \NN}$ for their direct product.

Given a set $X$, $\ell^\infty(X)$ denotes the Banach space of complex, bounded functions on $X$.

Finally, we write $\ZZ_+$ for the set of non-negative integers, $\RR_+$ for $[0,\infty)$, and set $\mathbb T := \RR /\ZZ$. Given a subset $\Delta \subseteq \RR$, we denote by $\chi_\Delta$ the characteristic function of $\Delta$. If $\Delta$ is measurable, then $\operatorname{meas}(\Delta)$ denotes its Lebesgue measure. The real part of a complex number $z \in \CC$ is denoted by $\Real z$.

\section{Invariant means on finite intervals}\label{subsec:periodic-means}

This section provides the averaging formalism used later in the
low-regularity similarity argument and fixes the notation for the interval
means \(\mu_a\), transported from the fixed invariant mean on \(\mathbb T\),
and for the associated set functions \(\nu_a\).
Since these objects reappear in several logically distinct
parts of the paper, it is convenient to isolate them once and for all before turning to the main
operator-theoretic constructions.

Throughout the paper, let
\[
\mathbb T:=\mathbb R/\mathbb Z,
\]
and fix a positive functional
\[
m:\ell^\infty(\mathbb T)\to \mathbb C
\]
with $m(\chi_\TT)=1$ and which is translation-invariant and inversion-invariant, that is,
\[
m(\theta_s\varphi)=m(\varphi),\qquad
m(\varphi\circ\iota)=m(\varphi),
\qquad s\in\mathbb R,\ \varphi\in \ell^\infty(\mathbb T),
\]
where
\[
(\theta_s\varphi)(x):=\varphi(x+s),\qquad x\in\mathbb T,
\]
and
\[
\iota(x):=-x,\qquad x\in\mathbb T.
\]
Such a choice can always be made by symmetrizing a translation-invariant mean $\widetilde m: \ell^\infty(\mathbb T)\to \mathbb C$, that is, defining $m$ as
$$m(\varphi):= \frac{1}{2} \left(\widetilde m (\varphi)  +  \widetilde m ( \varphi \circ \iota) \right), \qquad \varphi \in \ell^\infty(\TT).
$$

For each $a>0$, let
\[
q_a:[0,a)\to \mathbb T,\qquad q_a(t):=t/a+\mathbb Z.
\]
Given $f\in \ell^\infty([0,a))$, define $J_af\in \ell^\infty(\mathbb T)$ by
\[
(J_af)(q_a(t)):=f(t),\qquad t\in[0,a).
\]
We then set
\[
\mu_a(f):=m(J_af),\qquad f\in \ell^\infty([0,a)).
\]
If $A\subseteq [0,a)$, we write
\[
\nu_a(A):=\mu_a(\chi_A)=m(J_a\chi_A).
\]
Thus $\mu_a$ is the functional on bounded functions on $[0,a)$ induced from the
fixed mean $m$ on $\mathbb T$, while $\nu_a$ is the associated finitely additive
probability measure on subsets of $[0,a)$.

The next identities express how translations on $\mathbb T$ encode shifts on $[0,a)$.
Let $0\le s<a$ and let $f\in \ell^\infty([0,a))$. Then
\[
J_a\bigl(f(\cdot+s)\chi_{[0,a-s)}\bigr)
   = \theta_{s/a}\bigl(J_a(f\chi_{[s,a)})\bigr),
\]
and
\[
J_a\bigl(f(\cdot+s-a)\chi_{[a-s,a)}\bigr)
   = \theta_{s/a}\bigl(J_a(f\chi_{[0,s)})\bigr).
\]
Consequently,
\[
\mu_a\bigl(f(\cdot+s)\chi_{[0,a-s)}\bigr)=\mu_a(f\chi_{[s,a)}),
\]
and
\[
\mu_a\bigl(f(\cdot+s-a)\chi_{[a-s,a)}\bigr)=\mu_a(f\chi_{[0,s)}).
\]

Moreover, if $A\subseteq[0,a)$, then
\[
J_a(\chi_{a-A})=(J_a\chi_A)\circ\iota,
\]
where \(a-A\) is understood modulo \(a\), that is,
\[
a-A:=q_a^{-1}\bigl(-q_a(A)\bigr)\subseteq[0,a).
\]
Equivalently, for \(t\in(0,a)\) the point \(a-t\) is taken in the usual
sense, while \(0\) is sent to \(0\). Hence
\[
\nu_a(a-A)=\nu_a(A).
\]

Moreover, for every $s\in[0,a]$ one has
\[
\nu_a([0,s))=\frac{s}{a}.
\]
Indeed, for each $m\in\NN$ the intervals
\[
\Bigl[\frac{ka}{m},\frac{(k+1)a}{m}\Bigr),\qquad k=0,\dots,m-1,
\]
are translates of one another in $\mathbb T$, so translation invariance and finite additivity give
\[
\nu_a\Bigl(\Bigl[\frac{ka}{m},\frac{(k+1)a}{m}\Bigr)\Bigr)=\frac1m,
\qquad k=0,\dots,m-1.
\]
Hence
\[
\nu_a([0,pa/m))=\frac{p}{m}
\]
for every $p=0,\dots,m$. The general formula follows by monotonicity, approximating $s/a$ from below and above by rationals. In particular, if $c\in \mathbb C$, then
\[
\mu_a\bigl(c\,\chi_{[0,s)}\bigr)=\frac{s}{a}\,c.
\]

Finally, if $F:\mathbb R_+\to \mathbb C$ is bounded on $[0,a)$, we write
\[
\mu_a(F):=\mu_a(F|_{[0,a)}).
\]
Likewise, whenever $0\le s<a$ and the displayed expressions are bounded on the
indicated intervals, we interpret
\[
\mu_a\bigl(F(\cdot+s)\chi_{[0,a-s)}\bigr),\qquad
\mu_a\bigl(F(\cdot+s-a)\chi_{[a-s,a)}\bigr)
\]
by first viewing the corresponding bounded functions on $[0,a)$ and then applying
the transport $J_a$.

\section{An explicit equivalent norm and a low-regularity similarity theorem}\label{sec:sim-reg-th}

We now begin the abstract part of the paper by proving a low-regularity norm version of the similarity criterion
from \cite{OlivaMazaTomilovSimilarity}. Since the notions of $\SC(H)$, $\SQC(H)$, and the
constants $\mathcal C(\mathcal T)$ and $\mathcal C_H(\|\cdot\|_\sharp)$ were already introduced
in the Introduction, we do not repeat the corresponding operator-theoretic formulation here.
Instead, we work directly with equivalent Hilbertian norms and their distortion constants.

More precisely, we show that if the whole semigroup is quasi-contractive in one equivalent
Hilbertian norm, while one fixed operator $T(\tau)$ is contractive in another, then these two
pieces of information can be combined into a single equivalent Hilbertian norm for which the
whole semigroup is contractive. Quantitatively, the proof produces an explicit bound for the
distortion constant of that norm, and hence for the similarity constant $\mathcal C(\mathcal T)$.
This is the form of the result that will later be used in the tensor-product arguments.


To begin with, we need an existence result for the unitary dilation of semigroups  of contractions $\mathcal T\subset \linearOp(H)$ under essentially no regularity assumptions, which can be found in \cite[Theorem 3.8]{evans1977dilations}. Our later similarity argument depends on it in an essential way, while
the lack of regularity rules out several of the usual tools, such as semigroup generators, Cayley
transforms, and arguments based on measurability or integration.

\begin{proposition}[Unitary dilation of a contraction semigroup]\label{prop:unitary-dilation}
	Let $H$ be a Hilbert space and let $\mathcal T=(T(t))_{t\ge0}\subset\mathcal L(H)$ be a semigroup
	of contractions. 
	Then there exist a Hilbert space $K$, a unitary group $\mathcal U=(U(t))_{t\in\mathbb R}$ on $K$,
	and an isometric embedding $J:H\to K$ such that
	\begin{align*}
		J^*U(t)J = \begin{cases}
			T(t), & t\ge 0,\\[1mm]
			T(-t)^*, & t<0,
		\end{cases}
		\qquad t\in\mathbb R.
	\end{align*}
	
	Moreover, if
	\[
	K_+:=\overline{\operatorname{span}}\{U(t)Jx:\ t\ge0,\ x\in H\}\subseteq K,
	\]
	then $W(t):=U(t)|_{K_+}$, $t\ge0$, defines an isometric semigroup on $K_+$ and
	\[
	T(t)=J^*W(t)J,\qquad t\ge0.
	\]
\end{proposition}

\begin{remark}
 If $\mathcal U$ is weakly measurable (respectively, strongly continuous), then
 \[
 t\mapsto T(t)x=J^*U(t)Jx, \qquad t \ge 0,
 \]
 has the same regularity for every \(x\in H\). Thus a non-measurable semigroup cannot arise as
 the compression of a weakly measurable or strongly continuous unitary dilation.
\end{remark}

Now we return to our similarity result. Throughout the rest of this section, $H$ is a Hilbert space and
$\mathcal T=(T(t))_{t\ge0}\subset\mathcal L(H)$ is a semigroup, not assumed measurable.
We assume that there exist equivalent Hilbertian norms $\|\cdot\|_1,\|\cdot\|_2$ on $H$
and constants $C_1,C_2,\lambda,\tau>0$ such that
\begin{equation}\label{eq:Tasumptions}
\begin{aligned}
	\|h\|\le \|h\|_j\le C_j\|h\|, \qquad h\in H,\ j=1,2, \\
	\|T(t)\|_1\le e^{\lambda t}\qquad (t\ge0), \qquad
	\|T(\tau)\|_2\le1.
\end{aligned}
\end{equation}

Using Section~\ref{subsec:periodic-means}, let \(\mu_\tau\) be the
periodic mean on \(\ell^\infty([0,\tau))\) induced by the fixed mean on \(\mathbb T\). Observing that $\mathcal T$ is bounded on compact intervals since $\|T(t)\| \leq C_1 e^{\lambda t}$, $t\geq0$, define
\begin{equation}\label{eq:innerProd0Def}
\langle h,g\rangle_0
:=
\mu_\tau\bigl(\langle T(\cdot)h,T(\cdot)g\rangle_2\bigr),
\qquad h,g\in H.
\end{equation}
Let \(N_0:=\{h\in H:\|h\|_0=0\}\), let \(H_0\) be the completion of the quotient pre-Hilbert
space \(H/N_0\), and let
\[
J_0:H\to H_0
\]
be the canonical map. We also use the symbol $\|\cdot\|_{0}$ to denote the induced norm on $H_0$. For \(t\ge0\), define
\[
T_0(t)J_0h:=J_0T(t)h,\qquad h\in H.
\]

\begin{lemma}\label{lem:T0-contractive}
 The family \(\mathcal T_0=(T_0(t))_{t\ge0}\) is a contractive semigroup on \(H_0\).
\end{lemma}

\begin{proof}
 First we check that \(T_0(t)\) is well-defined, or equivalently, that $T_0 (t) J_0 h =0$ if $J_0 h=0$. For this purpose, suppose that $h\in H$ is such that \(J_0h=0\). If \(0<t\le\tau\), then
 \begin{equation}\label{eq:T0welldefined1}
 \begin{aligned}
  \|T(t)h\|_2^2
  &\le
  \frac{\tau}{t}\,
  \sup_{u\in[0,t]}\|T(u)\|_2^2\,
  \mu_\tau\bigl(\|T(\cdot)h\|_2^2\,\chi_{[0,t)}\bigr) 
  \leq  \frac{\tau}{t}\,
  \sup_{u\in[0,t]}\|T(u)\|_2^2\, \|Jh\|_0^2
  =0.
  \end{aligned}
 \end{equation}
 In particular $T(\tau/2)h=0$, hence, if \(t\ge\tau\), then 
 \begin{equation}\label{eq:T0welldefined2}
 	T(t)h = T(t-\tau/2) T(\tau/2)h = 0.
 \end{equation}
It follows from the above that \(J_0T(t)h=0\) for all $t\geq0$, so \(T_0(t)\) is well-defined.
 
 Next we prove contractivity for \(0\le t\le\tau\).
 Let \(h\in H\), and define
 \[
 \varphi_h(s):=\|T(s)h\|_2^2,
 \qquad s\in[0,\tau).
 \]
  Then
 \begin{align*}
  \|T_0(t)J_0h\|_0^2
  = \mu_\tau \bigl(\varphi_{T(t)h}\bigr)
  =
  \mu_\tau\bigl(\varphi_h(\cdot+t)\chi_{[0,\tau-t)}\bigr)
  +
  \mu_\tau\bigl( \varphi_{T(t)h} \chi_{[\tau-t,\tau)}\bigr).
 \end{align*}
 By the identities from Section~\ref{subsec:periodic-means},
 \begin{equation}\label{eq:T0a}
 \mu_\tau\bigl(\varphi_h(\cdot+t)\chi_{[0,\tau-t)}\bigr)
 =
 \mu_\tau\bigl(\varphi_h\,\chi_{[t,\tau)}\bigr).
 \end{equation}
 Also, for \(s\in[\tau-t,\tau)\),
 \[
 T(s)T(t)h=T(s+t-\tau)\,T(\tau)h,
 \]
and then
 \begin{equation}\label{eq:T0b}
 \mu_\tau\bigl(\varphi_{T(t)h} \chi_{[\tau-t,\tau)}\bigr)
  =
  \mu_\tau  \bigl(\varphi_{T(\tau)h} \chi_{[0,t)}\bigr)
  \leq \mu_\tau\bigl(\varphi_h\,\chi_{[0,t)}\bigr),
 \end{equation}
 where we used that $\varphi_{T(\tau) h} \leq \varphi_h$ since $\|T(\tau)\|_2\leq 1$.
 Combining \eqref{eq:T0a} and \eqref{eq:T0b}, we get
 \[
 \|T_0(t)J_0h\|_0^2
 \le
 \mu_\tau(\varphi_h)
 =
 \|J_0h\|_0^2.
 \]
 That is, $T_0(t)$ can be extended continuously to a contraction on $H_0$ for each $t\in [0,\tau]$.

 From the definition of $\mathcal T_0$, it is direct that $\mathcal T_0$ is a semigroup on $(0,\infty)$, thus $T_0(t)$ is a bounded operator on $H_0$ for every $t\geq0$. Moreover, given \(t\ge0\), write \(t=n\tau+r\) with \(n\in\mathbb N_0\) and \(r\in[0,\tau)\). Then
 \[
 T_0(t)=T_0(\tau)^nT_0(r),
 \]
 so \(T_0(t)\) is contractive since both $T_0(r)$ and $T_0(\tau)$ are contractive by the above.

\end{proof}

By Lemma~\ref{lem:T0-contractive}, \(\mathcal T_0\) is a contractive semigroup on \(H_0\).
Applying Proposition~\ref{prop:unitary-dilation} to \(\mathcal T_0\), we obtain a Hilbert space \(K\),
a unitary group \(U=(U(t))_{t\in\mathbb R}\) on \(K\), and an isometric embedding
\[
J_{H_0}:H_0\to K
\]
such that
\[
J_{H_0}^*U(t)J_{H_0}= \Phi_0(t),\qquad t\in\RR,
\]
where $\Phi_0$ is the operator-valued function given by
\[
\Phi_0(t):=
\begin{cases}
	T_0(t), & t\ge0,\\[1mm]
	T_0(-t)^*, & t<0,
\end{cases}
\qquad t\in\mathbb R.
\]
We keep this notation for the positive semidefinite operator-valued function obtained from the dilation.
We then write
\[
J:=J_{H_0}J_0:H\to K
\]
for the canonical map from \(H\) into the dilation space \(K\), obtained by first passing to
the averaged quotient \(H_0\) and then using the isometric embedding furnished by the dilation
of \(\mathcal T_0\). 

Let
\[
\mathcal F_+ := \operatorname{span}\{\delta_t:t\in\mathbb R_+\}
\]
be the vector space of finitely supported scalar-valued functions on \(\mathbb R_+\), where
\[
\delta_t(s):=
\begin{cases}
 1,& s=t,\\
 0,& s\neq t.
\end{cases}
\]
Set
$$F^+:=\mathcal F_+\otimes H,
$$
and note that via the canonical identification
\[
\delta_t\otimes h \longleftrightarrow [\,s\mapsto \delta_t(s)h\,], \qquad t \in \RR_+, \, h \in H,
\]
we may view \(F^+\) as the space of finitely supported \(H\)-valued functions on \(\mathbb R_+\).
Thus every \(f\in F^+\) can be written in the form
\[
f=\sum_{j=1}^n \delta_{t_j}\otimes h_j, \qquad n \in \mathbb{N}, \, t_j \in \mathbb R_+, \, h_j \in H.
\]
We now define two sesquilinear forms and their sum. 
First, for \(s,t\ge0\), define 
\[
a(s,t):=(1-|s-t|)_+=\operatorname{meas}([s,s+1]\cap[t,t+1]).
\]
Recall that $\operatorname{meas}$ denotes the Lebesgue measure on $\RR$. Then, for
\begin{equation}\label{eq:fgFunct}
f=\sum_{j=1}^n \delta_{t_j}\otimes h_j,
\qquad
g=\sum_{\ell=1}^m \delta_{s_\ell}\otimes g_\ell
\end{equation}
in \(F^+\), define
\begin{align*}
	\langle f,g\rangle_\alpha
	&:=
	\int_0^2 e^{-2\lambda r}
	\Biggl\langle
	\sum_{j=1}^n T(t_j)h_j\,\chi_{[t_j,t_j+1]}(r),\
	\sum_{\ell=1}^m T(s_\ell)g_\ell\,\chi_{[s_\ell,s_\ell+1]}(r)
	\Biggr\rangle_1\,dr,\\
	\langle f,g\rangle_\beta
	&:=
	\sum_{j=1}^n\sum_{\ell=1}^m
	\bigl(1+a(t_j,s_\ell)\bigr)
	\bigl\langle \Phi_0(t_j-s_\ell)J_0h_j,\ J_0g_\ell\bigr\rangle_0,
\end{align*}
where $\langle \cdot, \cdot\rangle_1$ is the inner product associated to the Hilbertian norm $\|\cdot\|_1$ satisfying~\eqref{eq:Tasumptions}, and $\langle \cdot, \cdot\rangle_0$ is defined in~\eqref{eq:innerProd0Def}; and set
\begin{equation}\label{eq:qDef}
\langle f,g\rangle_q:=\langle f,g\rangle_\alpha+\langle f,g\rangle_\beta.
\end{equation}

\begin{proposition}\label{prop:q-hilbert}
\(\langle\cdot,\cdot\rangle_q\) is a positive semidefinite sesquilinear form on \(F^+\).
\end{proposition}

\begin{proof}
Fix $f,g \in F^+$, and let $t_j, h_j, s_\ell, g_\ell$ be the elements for the tensorial decomposition of $f$ and $g$ as in \eqref{eq:fgFunct}. Then form \(\langle\cdot,\cdot\rangle_\alpha\) is sesquilinear and positive semidefinite, since
\[
\langle f,f\rangle_\alpha
=
\int_0^2 e^{-2\lambda r}
\Bigl\|
\sum_{j=1}^n T(t_j)h_j\,\chi_{[t_j,t_j+1]}(r)
\Bigr\|_1^2\,dr
\ge0.
\]
 
 For the \(\beta\)-part, define
 \[
 \eta_t:=1\oplus \chi_{[t,t+1]}\in \mathbb C\oplus L^2(\mathbb R_+),\qquad t\ge0.
 \]
 Then, for every $s,t\geq0,$ we have that
 \[
 \langle \eta_t,\eta_s\rangle_{\mathbb C\oplus L^2(\mathbb R_+)}
 =
 1+\operatorname{meas}([t,t+1]\cap[s,s+1])
 =
 1+a(t,s).
 \]
 Also,
 \[
 \langle \Phi_0(t_j-s_\ell)J_0h_j,\ J_0g_\ell\rangle_0
 =
 \langle U(t_j)Jh_j,\ U(s_\ell)Jg_\ell\rangle_K.
 \]
 Hence
 \begin{align*}
  \langle f,g\rangle_\beta
  &=
  \sum_{j=1}^n \sum_{\ell=1}^m
  \langle \eta_{t_j},\eta_{s_\ell}\rangle\,
  \langle U(t_j)Jh_j,\ U(s_\ell)Jg_\ell\rangle_K\\
  &=
  \Biggl\langle
  \sum_{j=1}^n \eta_{t_j}\otimes U(t_j)Jh_j,\
  \sum_{\ell=1}^m \eta_{s_\ell}\otimes U(s_\ell)Jg_\ell
  \Biggr\rangle_{(\mathbb C\oplus L^2(\mathbb R_+))\otimes K}.
 \end{align*}
 In particular,
 \begin{equation*}
 	\langle f, f \rangle_\beta = \left\| \sum_{j=1}^n \eta_{t_j}\otimes U(t_j)Jh_j \right\|_{(\mathbb C\oplus L^2(\mathbb R_+))\otimes K}^2 \geq 0,
\end{equation*}
 thus \(\langle\cdot,\cdot\rangle_\beta\) is positive semidefinite, and the claim follows.
\end{proof}

We define the map
\[
\Sigma:F^+\to H,
\qquad
\Sigma\left(\sum_{j=1}^n \delta_{t_j}\otimes h_j\right)
:=
\sum_{j=1}^n T(t_j)h_j,
\]
and, using this map, for \(h\in H\), we set
\[
F_h^+:=\{f\in F^+:\ \Sigma f=h\}.
\]

We are now ready to prove the main result of this section. 
It says that an arbitrary semigroup
\(\mathcal T=(T(t))_{t\ge0}\subset \mathcal L(H)\), not assumed measurable
or strongly continuous, is contractive with respect to a single equivalent
Hilbertian norm whenever \(\mathcal T\) is quasi-contractive with respect to
one equivalent Hilbertian norm and one operator \(T(\tau)\), \(\tau>0\), is
contractive with respect to another.

\begin{theorem}\label{thm:explicit-eq-norm}\label{Th:eqNorm}
Let $H$ be a Hilbert space and let $\mathcal T=(T(t))_{t\ge0}\subset \mathcal L(H)$ be a semigroup. Suppose that there exist equivalent Hilbertian norms
$\|\cdot\|_1$ and $\|\cdot\|_2$ on $H$, constants $C_1,C_2\ge1$, and parameters
$\lambda,\tau>0$ such that~\eqref{eq:Tasumptions} holds.
Then the mapping
\begin{equation*}
	\|h\|_{\rm{eq}} := \inf\left\{ \sqrt{q(f)} \, : \, f \in F_h^+\right\}, \qquad h \in H,
\end{equation*}
is an equivalent Hilbertian norm on $H$ such that
\[
\|T(t)h\|_{\rm{eq}} \le \|h\|_{\rm{eq}},
\qquad h\in H,\ t\ge0. 
\]
In particular, $\mathcal T\in \SC(H)$. Moreover,
\[
\frac{1}{M_2} \|h\|\le \|h\|_{\rm{eq}} \le M_1 \|h\|,
\qquad h\in H,
\]
where
\[
M_1=
\left(
C_1^2\frac{1-e^{-4\lambda}}{2\lambda}
+
2C_2^2\sup_{s\in[0,\tau)}\|T(s)\|^2
\right)^{1/2},
\]
and
\[
M_2=
\left(
\frac{e^{4\lambda}-1}{2\lambda}
+
4\maxone{\tau}\sup_{s\in[0,1]}\|T(s)\|^2
\right)^{1/2}.
\]
Consequently, $\mathcal C (\mathcal T) \leq \mathcal C_H(\|\cdot\|_{\rm{eq}})\le M_1 M_2$.
\end{theorem}


\begin{proof}
 The proof is quite long and we split it into five steps.
 
 \medskip
 \noindent
 \textbf{Step 1: factorization through \(H_0\).}
 For \(t>0\), define
 \[
 P(t)J_0h:=T(t)h,\qquad h\in H.
 \]
 We show that this extends to a bounded operator \(P(t):H_0\to H\). Note first that, if \(J_0h=0\), then $T(t)h=0$ as proven in \eqref{eq:T0welldefined1} and \eqref{eq:T0welldefined2}, thus \(P(t)\) is well-defined.

If \(0<t\le\tau\), repeating the argument in \eqref{eq:T0welldefined1} with a minor modification provides that
 \begin{align*}
  \|T(t)h\|^2
  &=
  \frac{\tau}{t}\,
  \mu_\tau\bigl(\|T(t)h\|^2\,\chi_{[0,t)}\bigr)\\
  &\le
  \frac{\tau}{t}\,
  \sup_{u\in[0,t]}\|T(u)\|^2\,
  \mu_\tau\bigl(\|T(\cdot)h\|^2\,\chi_{[0,t)}\bigr)\\
  &\le
  \frac{\tau}{t}\,
  \sup_{u\in[0,t]}\|T(u)\|^2 \,\|J_0h\|_0^2, \qquad h \in H,
 \end{align*}
 where we used that $\|\cdot\| \leq \|\cdot\|_2$. Similarly, if \(t\ge\tau\), then for every \(s\in[0,\tau)\) one has
 \[
 \|T(t)h\|^2
 =
 \|T(t-s)T(s)h\|^2
 \le
 \sup_{u\in[t-\tau,t]}\|T(u)\|^2\,\|T(s)h\|_2^2, \qquad h\in H.
 \]
 Averaging over \(s\in[0,\tau)\) with \(\mu_\tau\) yields
 \[
 \|T(t)h\|^2
 \le
 \sup_{u\in[t-\tau,t]}\|T(u)\|^2\,\|J_0h\|_0^2, \qquad h \in H.
 \]
 Therefore
 \[
 \|P(t)\|
 \le
 \maxone{\sqrt{\tau/t}}
 \sup_{u\in[0,t]}\|T(u)\|,
 \qquad t>0,
 \]
 and, in particular, $\|P(1)\|\le \maxone{\sqrt{\tau}}\sup_{u\in[0,1]}\|T(u)\|$.
 \medskip

 \noindent
 \textbf{Step 2: upper bound.}
 For $h\in H$, since \(h=T(0)h\), we have \(\delta_0\otimes h\in F_h^+\). Hence
 \[
 \|h\|_{\rm{eq}}^2\le \|\delta_0\otimes h\|_\alpha^2+\|\delta_0\otimes h\|_\beta^2.
 \]
 Now
 \[
 \|\delta_0\otimes h\|_\alpha^2
 =
 \int_0^{2} e^{-2\lambda s}\|h\|_1^2\,ds
 =
 \frac{1-e^{-4\lambda}}{2\lambda}\|h\|_1^2
 \le
 C_1^2\frac{1-e^{-4\lambda}}{2\lambda}\|h\|^2,
 \]
 while, since \(a(0,0)=1\),
 \[
 \|\delta_0\otimes h\|_\beta^2
 =
 2\|J_0h\|_0^2
 \le
 2C_2^2\sup_{s\in[0,\tau)}\|T(s)\|^2\,\|h\|^2.
 \]
 Thus $\|h\|_{\rm{eq}}\le M_1 \|h\|$ for all $h\in H$, as claimed.
\medskip

\noindent
 \textbf{Step 3: lower bound.}
 For \(t\ge0\), define
 \[
 c(t):=\operatorname{meas}([0,2]\cap[t,t+1])=
 \begin{cases}
  1,& 0\le t\le1,\\
  2-t,& 1\le t\le2,\\
  0,& t\ge2.
 \end{cases}
 \]
 Fix \(h\in H\), and let
 \[
 f=\sum_{j=1}^n \delta_{t_j}\otimes h_j\in F_h^+,
 \qquad\text{so that}\qquad
 h=\sum_{j=1}^n T(t_j)h_j.
 \]
 We write
 \begin{equation}\label{two_sum}
  h=
  \sum_{j=1}^n c(t_j)T(t_j)h_j
  +
  \sum_{j=1}^n (1-c(t_j))T(t_j)h_j,
 \end{equation}
 and  estimate each of the two summands above separately. The first summand will be controlled directly by the $\alpha$-term on the interval $[0,2]$, whereas the second summand will be rewritten through the quotient semigroup and the dilation so as to be controlled by the $\beta$-term.
 
 \smallskip
 \noindent
 To estimate the ``\(c(t_j)\)''-part, we set
 \[
 G(s):=\sum_{j=1}^n T(t_j)h_j\,\chi_{[t_j,t_j+1]}(s),\qquad s\in[0,2].
 \]
Then
 \[
 \sum_{j=1}^n c(t_j)T(t_j)h_j
 =
 \int_0^2 G(s) \,ds,
 \]
so the Cauchy--Schwarz inequality and the inequality $\|\cdot\| \leq \|\cdot\|_1$ give
\begin{align*}
 \Bigl\|
 \sum_{j=1}^n c(t_j)T(t_j)h_j
 \Bigr\|
 &\le  \int_0^2 \|G(s)\|_1\,ds\\
 &\le
   \left(\int_0^2 e^{2\lambda s}\,ds\right)^{1/2}
 \left(\int_0^2 e^{-2\lambda s}\|G(s)\|_1^2\,ds\right)^{1/2}\\
 &= \Biggl(\frac{e^{4\lambda}-1}{2\lambda}\Biggr)^{1/2}\|f\|_\alpha.
\end{align*}
 
 \smallskip
 \noindent
 The bound for the ``\((1-c(t_j))\)''-part is more involved.
 Note that
 \[
 1-c(t)=\operatorname{meas}([0,1]\cap[0,t-1]),\qquad t\ge0,
 \]
 where we use the convention $[a,b] =\emptyset$ if $b<a$. 
 Since $1-c(t)=0$ when $t\in [0,1)$, then
 \begin{align*}
 \sum_{j=1}^n (1-c(t_j))T(t_j)h_j
 &=
 T(1) \sum_{t_j\geq1} (1-c(t_j))T(t_j-1)h_j
 \\&=
 T(1)\int_0^1 \sum_{t_j\ge1} T(t_j-1)h_j\,\chi_{[0,t_j-1]}(s)\,ds.
 \end{align*}
 Define
\[
H(s):=\sum_{t_j\ge1} T_0(t_j-1)J_0h_j\,\chi_{[0,t_j-1]}(s),\qquad s\in[0,1].
\]
Therefore, using the operator \(P(1):H_0\to H\) from Step 1,
\begin{align*}
 \Bigl\|
 \sum_{j=1}^n (1-c(t_j))T(t_j)h_j
 \Bigr\|
 \le \|P(1)\|\int_0^1 \|H(s)\|_0\,ds
 \le \|P(1)\|\left(\int_0^1 \|H(s)\|_0^2\,ds\right)^{1/2}.
\end{align*}
 Now
 \[
 T_0(t_j-1)J_0h_j
 =
 J_{H_0}^*U(t_j-1)J_{H_0}J_0h_j
 =
 J_{H_0}^*U(t_j-1)Jh_j,
 \]
 so, since \(\|J_{H_0}^*\|\le1\), if we set
\begin{align*}
\widetilde H(s)&:=\sum_{t_j\ge1} U(t_j-1)Jh_j\,\chi_{[0,t_j-1]}(s)
= \sum_{j=1}^n U(t_j-1)Jh_j\,\chi_{[0,t_j-1]}(s),\qquad s\in[0,1],
\end{align*}
then
\[
\int_0^1 \|H(s)\|_0^2\,ds\le \int_0^1 \|\widetilde H(s)\|_K^2\,ds.
\]
For \(s\in[0,1]\) we have
 \[
 \chi_{[0,t_j-1]}(s)
 = 1-\chi_{[t_j-1,1]}(s).
 \]
 Hence, 
 \begin{align*}
  \int_0^1 \|\widetilde H(s)\|_K^2\,ds 
  &\le
  2\Bigl\|
  \sum_{j=1}^n U(t_j-1)Jh_j
  \Bigr\|_K^2
  +
  2\int_0^1
  \Bigl\|
  \sum_{j=1}^n U(t_j-1)Jh_j\,\chi_{[t_j-1,1]}(s)
  \Bigr\|_K^2\,ds.
 \end{align*}
 Changing variables in the second integral and using
 \[
 \chi_{[t,2]}(s)=\chi_{[t,t+1]}(s)+\chi_{[t,t+1]}(s-1),
 \qquad s\in[1,2], \quad t\geq0,
 \]
 we obtain
 \begin{align*}
  &\int_0^1
  \Bigl\|
  \sum_{j=1}^n U(t_j-1)Jh_j\,\chi_{[t_j-1,1]}(s)
  \Bigr\|_K^2\,ds
  = \int_1^2
  \Bigl\|
  \sum_{j=1}^n U(t_j)Jh_j\,\chi_{[t_j,2]}(s)
  \Bigr\|_K^2\,ds\\
  &\quad\le
  2\int_1^2
  \Bigl\|
  \sum_{j=1}^n U(t_j)Jh_j\,\chi_{[t_j,t_j+1]}(s)
  \Bigr\|_K^2\,ds
  + 2\int_1^2
  \Bigl\|
  \sum_{j=1}^n U(t_j)Jh_j\,\chi_{[t_j,t_j+1]}(s-1)
  \Bigr\|_K^2\,ds\\
  &\quad=
  2\int_0^2
  \Bigl\|
  \sum_{j=1}^n U(t_j)Jh_j\,\chi_{[t_j,t_j+1]}(s)
  \Bigr\|_K^2\,ds.
 \end{align*}
Putting these estimates together we deduce that
 \begin{align*}
	\int_0^1 \|H(s)\|_0^2\,ds
  & \le
  2\Bigl\|
  \sum_{j=1}^n U(t_j-1)Jh_j
  \Bigr\|_K^2
  +
  4\int_0^2
  \Bigl\|
  \sum_{j=1}^n U(t_j)Jh_j\,\chi_{[t_j,t_j+1]}(s)
  \Bigr\|_K^2\,ds.
 \end{align*}
 Since, 
 \[
 \Bigl\|
 \sum_{j=1}^n U(t_j-1)Jh_j
 \Bigr\|_K^2
 = \Bigl\|\sum_{j=1}^n U(t_j)Jh_j
 \Bigr\|_K^2
 =
 \sum_{i,j=1}^n\langle \Phi_0(t_i-t_j)J_0h_i,\ J_0h_j\rangle_0,
 \]
 and
 \begin{align*}
  &\int_0^2
  \Bigl\|
  \sum_{j=1}^n U(t_j)Jh_j\,\chi_{[t_j,t_j+1]}(s)
  \Bigr\|_K^2\,ds
  \\ & \quad \leq \int_0^\infty
   \Bigl\|
   \sum_{j=1}^n U(t_j)Jh_j\,\chi_{[t_j,t_j+1]}(s)
   \Bigr\|_K^2\,ds
  \\ & \quad = \sum_{i,j=1}^n  \langle \Phi_0(t_i-t_j)J_0h_i,\ J_0h_j  \rangle_0 \int_0^\infty \chi_{[t_i,t_i+1]}(s) \, \chi_{[t_j,t_j+1]}(s)\, ds
  \\ & \quad = 
  \sum_{i,j=1}^n
  a(t_i,t_j)\,
  \langle \Phi_0(t_i-t_j)J_0h_i,\, J_0h_j\rangle_0,
 \end{align*}
 we have
 \begin{align*}
  \Bigl\|
  \sum_{j=1}^n (1-c(t_j))T(t_j)h_j
  \Bigr\|^2
  & \le 
  2 \|P(1)\|^2
  \sum_{i,j=1}^n(1 + 2 a(t_i,t_j) )  \langle \Phi_0(t_i-t_j)J_0h_i,\, J_0h_j\rangle_0
 \\
  &\le
  4\|P(1)\|^2\,\|f\|_\beta^2\\
  &\le
  4\maxone{\tau}\sup_{s\in[0,1]}\|T(s)\|^2 \|f\|_\beta^2.
 \end{align*}
 
 Now combining the obtained estimates for the two summands in \eqref{two_sum}, we infer that
 \begin{equation}\label{eq:qBounded}
 \|h\|
 \le
 \left(\frac{e^{4\lambda}-1}{2\lambda}\right)^{1/2}
 \|f\|_\alpha
 +2\maxone{\sqrt{\tau}}\sup_{s\in[0,1]}\|T(s)\|\,
 \|f\|_\beta
 \le M_2 q(f)^{1/2}.
 \end{equation}
 Taking the infimum over all \(f\in F_h^+\) yields
 \[
 \|h\|\le M_2\|h\|_{\rm{eq}},
 \qquad h\in H.
 \]
 
 \medskip
 \noindent
 \textbf{Step 4: Hilbertianity.}
 By Proposition~\ref{prop:q-hilbert}, \(q\) is induced by a positive semidefinite sesquilinear form
 on \(F^+\).
 Hence, if we set 
 $$N_q := \{ f \in F^+ \, : \, q(f) = 0\},
 $$
 then \(F^+/N_q\) is a pre-Hilbert space, and its completion \(X\) is a Hilbert space.
 The map \(\Sigma:F^+\to H\), which is $q$-bounded by \eqref{eq:qBounded}, induces a surjective bounded linear map
 \[
 \widetilde\Sigma:X\to H,
 \]
 and \(\|\cdot\|_{\rm{eq}}\) is precisely the quotient norm induced by \(\widetilde\Sigma\).
 Indeed, for every \(h\in H\) one has
 \[
 \|h\|_{\rm{eq}}
 =
 \inf\{q(f)^{1/2}: f\in F^+,\ \Sigma f=h\},
 \]
 and this formula passes unchanged to the quotient by \(N_q\) and then to
 the completion \(X\).
 Thus \(\|\cdot\|_{\rm eq}\) is the quotient norm of a Hilbert space, hence a
 Hilbertian seminorm.
 Since the lower bound proved above implies that \(\|h\|_{\rm{eq}}=0\) only for \(h=0\),
 it follows that \(\|\cdot\|_{\rm{eq}}\) is an equivalent Hilbertian norm on \(H\).

 \medskip
 \noindent
 \textbf{Step 5: contractivity.}
 Fix \(h\in H\), let
 \[
 f=\sum_{j=1}^n \delta_{t_j}\otimes h_j\in F_h^+,
 \]
 and  \(t\ge0\). Consider the right shift $\sigma_t: F^+ \to F^+$, defined as
 \[
 \sigma_t f:=\sum_{j=1}^n \delta_{t+t_j}\otimes h_j, \qquad f \in F^+.
 \]
 Note that $\sigma_t (F_h^+) \subseteq F_{T(t)h}^+$. Indeed, given $f\in F_h^+$,
 \[
 \Sigma(\sigma_t f)=\sum_{j=1}^n T(t+t_j)h_j = T(t)\sum_{j=1}^n T(t_j)h_j=T(t)h.
 \]
 
 If \(t\ge2\), then clearly \(\|\sigma_t f\|_\alpha=0\le \|f\|_\alpha\).
 Assume now that \(0\le t\le2\). Using \(\|T(t)\|_1\le e^{\lambda t}\), we get
 \begin{align*}
  \|\sigma_t f\|_\alpha^2
  &=
  \int_0^2 e^{-2\lambda s}
  \Bigl\|
  \sum_{j=1}^n T(t+t_j)h_j\,\chi_{[t+t_j,t+t_j+1]}(s)
  \Bigr\|_1^2\,ds\\
  &\le
  \int_0^2 e^{-2\lambda(s-t)}
  \Bigl\|
  \sum_{j=1}^n T(t_j)h_j\,\chi_{[t_j,t_j+1]}(s-t)
  \Bigr\|_1^2\,ds\\
  &=
  \int_0^{2-t} e^{-2\lambda r}
  \Bigl\|
  \sum_{j=1}^n T(t_j)h_j\,\chi_{[t_j,t_j+1]}(r)
  \Bigr\|_1^2\,dr
  \le
  \|f\|_\alpha^2.
 \end{align*}
 
 For the \(\beta\)-part, since \(a(r+t,s+t)=a(r,s)\) for all \(r,s,t\ge0\),
 \begin{align*}
  \|\sigma_t f\|_\beta^2
  &=
  \sum_{i,j=1}^n
  \bigl(1+a(t_i+t,t_j+t)\bigr)
  \langle \Phi_0((t_i+t)-(t_j+t))J_0h_i,\ J_0h_j\rangle_0\\
  &=
  \sum_{i,j=1}^n
  \bigl(1+a(t_i,t_j)\bigr)
  \langle \Phi_0(t_i-t_j)J_0h_i,\ J_0h_j\rangle_0
  =
  \|f\|_\beta^2.
 \end{align*}
 Hence
 \[
 q(\sigma_t f)\le q(f).
 \]
 Taking infima over all \(f\in F_h^+\), we conclude that
 \[
 \|T(t)h\|_{\rm{eq}}\le \|h\|_{\rm{eq}},
 \qquad h\in H,\ t\ge0.
 \]
 Thus
 \[
 \|T(t)\|_{\rm{eq}}\le1,
 \qquad t\ge0.
 \]
This completes the proof.
\end{proof}

\begin{remark}
 The choice
 \[
 a(t,s)=(1-|t-s|)_+
 \]
 is convenient because it admits the concrete Gram representation
 \[
 1+a(t,s)=\langle \eta_t,\eta_s\rangle,
 \qquad
 \eta_t=1\oplus \chi_{[t,t+1]}
 \in \mathbb C\oplus L^2(\mathbb R_+),
 \]
 and because the corresponding \(\alpha\)-term interacts well with the exponential bound
 \[
 \|T(t)\|_1\le e^{\lambda t}.
 \]
 
 Alternatively, one may replace the scalar kernel \(1+a(t,s)\) by any positive semidefinite kernel
 \(b(t,s)\) on \(\mathbb R_+\) for which the resulting \(\beta\)-form remains positive and for which
 the lower-bound argument can still be carried out. In that case one defines
 \[
 \|f\|_{\beta,b}^2
 :=
 \sum_{i,j=1}^n b(t_i,t_j)\,
 \langle \Phi_0(t_i-t_j)J_0h_i,\ J_0h_j\rangle_0,
 \qquad
 f=\sum_{j=1}^n \delta_{t_j}\otimes h_j.
 \]
 The proof above then goes through once the analogues of the positivity and lower-bound estimates
 are available.
 
 A more homogeneous variant would be to replace \(1+a(t,s)\) by a translation-invariant positive
 definite kernel \(\varphi(t-s)\) on \(\mathbb R\).
 This would keep the \(\beta\)-part manifestly Toeplitz. However, the present kernel
 \(1+a(t,s)\) is tailored to the elementary decomposition used in Step 3, and it is not clear
 that the same estimates remain equally transparent for a general choice of \(\varphi\).
\end{remark}

To finish this section, we provide a tighter upper bound for the similarity constant $\mathcal C(\mathcal T)$, which is obtained by averaging in an optimal way the $\alpha$-part and $\beta$-part of the equivalent norm.
\begin{proposition}\label{Prop:simConstant}
 In the setting of Theorem~\ref{Th:eqNorm}, 
 \[
  \mathcal C(\mathcal T)
  \le C_1 \frac{\sinh(2\lambda)}{\lambda}
  + 2\sqrt{2}\,C_2\maxone{\sqrt{\tau}}
    \sup_{t\in[0,\maxone{\tau}]}\|T(t)\|^2.
 \]
\end{proposition}
\begin{proof}
 Set
 \begin{align*}
  a&:= C_1 \sqrt{\frac{1-e^{-4\lambda}}{2\lambda}}, \qquad  &b&:= \sqrt{2} C_2 \sup_{t\in [0,\tau)} \|T(t)\|, \\
  c&:= \sqrt{\frac{e^{4\lambda}-1}{2\lambda}}, \qquad  &d&:= 2 \maxone{\sqrt{\tau}} \sup_{t\in [0,1]} \|T(t)\|,
 \end{align*}
 and for each $r,s>0$, define
 \begin{align*}
  \|h\|_{r,s}^2 := \inf_{f \in  F_h^+} r \|f\|_\alpha^2 + s \|f\|_\beta^2, \qquad h \in H,
 \end{align*}
 Repeating the argument given in the proof of Theorem \ref{thm:explicit-eq-norm}, we obtain that, for each $r,s>0$, $\|\cdot\|_{r,s}$ is an equivalent Hilbertian norm on $H$ such that $\|T(t)\|_{r,s} \leq 1$ for all $t\geq 0$ and
 \begin{align*}
  \|h\|_{r,s} \leq \sqrt{r a^2 + s b^2} \, \|h\|, \qquad \|h\| \leq \sqrt{\frac{1}{r} c^2 + \frac{1}{s} d^2} \, \|h\|_{r,s},
 \end{align*}
 for all $h\in H$. This yields that
 \begin{align}\label{Eq:BadeaTrick}
  \mathcal C(\mathcal T) \leq \sqrt{ \left(r a^2 + s b^2 \right) \left(\frac{1}{r} c^2 + \frac{1}{s} d^2 \right)}, 
 \end{align}
 for all $r,s>0$. It is readily seen that the absolute minimum of the right-hand side in \eqref{Eq:BadeaTrick} is reached for any choice of $r,s$ satisfying $\frac{r}{s} = \frac{bc}{ad}$, and then \eqref{Eq:BadeaTrick} reads as
 \begin{align*}
  \mathcal C(\mathcal T) \leq ac + bd, 
 \end{align*}
 which yields the claim.
\end{proof}

\section{Numerical preliminaries on unconditional infinite products}
\label{sec:num-prod}

We now return to the tensor-product side of the paper. First, in Subsection~\ref{subsec:preliminaries} we collect several
basic facts on unconditional convergence of infinite products. They are elementary, but they
enter repeatedly in the operator and semigroup constructions below, so it is useful to assemble
them in one place. Secondly, we provide in Subsection~\ref{subsec:uniform-bounds-zero} the near-zero product bounds that will be needed in the converse direction of the similarity theorems in Section~\ref{sec:similarity-tensor}.


\subsection{Definitions and basic identities}\label{subsec:preliminaries}

Given a sequence $(z_n)_{n\in \NN} \subset \CC$, we say that the product $\prod_{n\in \NN} z_n$ is \emph{unconditionally convergent} to $a\in\mathbb C$ if, for every bijection $\sigma: \NN\to \NN$ we have that
	\begin{equation*}
		\lim_{N \to \infty} \prod_{n=1}^N z_{\sigma(n)} = a.
	\end{equation*}
	This is equivalent to the fact that, for every $\varepsilon>0$
	there exists a finite set $I_0\subset \NN$ such that
	\begin{equation}\label{eq:uncondProd}
		\Bigl|a-\prod_{n\in I} z_n\Bigr|\le \varepsilon
		\qquad
		\text{for every finite }I\subset \NN \text{ with } I_0\subseteq I.
	\end{equation}
	In other words, let $\finiteN$ be the net of finite subsets of $\NN$, ordered by inclusion. That is,
	$$\finiteN := \{F \subset \NN \, : \, |F| < \infty\},
	$$
	and $F\leq E$ in $\finiteN$ if and only if $F\subseteq E$. Then $\prod_{n\in \NN} z_n$ is unconditionally convergent to $a\in\mathbb C$ if and only if the net
	$$\Big( \prod_{n\in F} z_n \Big)_{F\in\finiteN}  
	$$
	converges to $a$.
	
	If $\prod_{n\in \NN}|z_n|$ is unconditionally convergent, we say that $\prod_{n\in \NN} z_n$ is
	\emph{quasi-convergent}, and we define the quasi-convergent product, denoted by the same symbol $\prod_{n\in \NN} z_n$, as 
	\begin{align*}
		\begin{cases}
			\prod_{n\in \NN} z_n, \quad &\mbox{if } \prod_{n\in \NN} z_n \mbox{ is unconditionally convergent,} \\
			0, \quad & \mbox{otherwise}.
		\end{cases}
\end{align*}


	Now we define analogous limits for infinite series. We say that the series $\sum_{n\in \NN} z_n$ is \emph{unconditionally convergent} to $a \in \CC$ if for every bijection $\sigma:\NN \to \NN$, 
	\begin{equation}\label{eq:uncondSeries}
		\lim_{N\to \infty} \sum_{n=1}^N z_{\sigma(n)} = a.
	\end{equation}
	We will also consider series of sequences $(c_n)_{n\in \NN} \subset [-\infty, \infty)$, where we consider the usual convention
	$$-\infty + (-\infty) = - \infty, \qquad -\infty + x = -\infty, \qquad x \in \RR.
	$$
	For such sequences we also allow unconditional convergence to an
	infinite-valued limit \(a\in\{-\infty,\infty\}\), defined analogously.
	Hence, the series $\sum_{n\in \NN} c_n$ is unconditionally convergent to a finite value if and only if it is absolutely convergent, and it is unconditionally convergent to $\pm \infty$ if and only if, either $c_n = -\infty$ for some $n\in \NN$, or exactly one of the series 
	$$\sum_{n\in \NN} \max\{c_n,0\},\qquad \sum_{n\in \NN} \min\{c_n,0\},
	$$
	converges, while the other one diverges.



For $x\in\mathbb R$ set 
$$x_+:=\max\{x,0\}, \qquad x_-:=\min\{x,0\}.
$$

We record several basic lemmas on convergence of products that will be used repeatedly.

	\begin{lemma}\label{lem:vn1}
		Let $(\alpha_n)_{n\in \NN}\subset[0,\infty)$. 
		\begin{itemize}
			\item [(i)] $\prod_{n\in \NN}\alpha_n$ is unconditionally convergent
			if and only if either $\alpha_n = 0$ for some $n\in \NN$, or 
			$$\sum_{n\in \NN} (\alpha_n - 1)_+ < \infty.
			$$
			\item [(ii)] $\prod_{n\in \NN}\alpha_n$ is unconditionally convergent to a value in $(0,\infty)$ if and only if $\alpha_n\neq 0$ for all $n\in \NN$ and 
			$$\sum_{n\in \NN} |\alpha_n - 1| < \infty.
			$$
		\end{itemize}
	\end{lemma}
	\begin{proof}
		The proof of this result is given in \cite[Sections 2.3 \& 2.4]{von1939infinite}.
	\end{proof}
	
	Note that if $\beta_n\ge 1$ for all $n\in \NN$, then the net of partial products $\prod_{n\in I}\beta_n$
	is monotone increasing and hence converges 
	if and only if it is bounded.
	In particular, for such products ordinary and unconditional convergence coincide. If $0<\gamma_n\le 1$
	for all $n\in \NN$, then $\prod_{n\in I}\gamma_n$ is monotone decreasing and hence always converges (possibly to $0$).
	This is the reason why controlling only the ``upper deviations'' $(\alpha_n-1)_+$ yields finite convergence of
	$\prod_{n\in \NN} \alpha_n$, but does not by itself prevent the limit from being $0$.

The next lemma isolates the nonzero case and explains the max/min splitting in logarithmic variables.


\begin{lemma}[Max/min factorization and log parts]\label{lem:maxmin}
Let $(\alpha_n)_{n\in \NN}\subset(0,\infty)$.
The following are equivalent:
\begin{enumerate}
\item [(i)] $\prod_{n\in \NN}\alpha_n$ is unconditionally convergent to a value in $(0,\infty)$;
\item [(ii)] $\sum_{n\in \NN}|\log \alpha_n|<\infty$;
\item [(iii)]$\sum_{n\in \NN}(\log\alpha_n)_+<\infty$ and $\sum_{n\in \NN}(\log\alpha_n)_->-\infty$;
\item [(iv)] $\prod_{n\in \NN} \maxone{\alpha_n}<\infty$ and $\prod_{n\in \NN}\minone{\alpha_n}>0$.
\end{enumerate}
If any of the above hold,
\[
\prod_{n\in \NN}\alpha_n
=
\Bigl(\prod_{n\in \NN} \maxone{\alpha_n}\Bigr)\Bigl(\prod_{n\in \NN} \minone{\alpha_n}\Bigr),
\qquad
\log\Bigl(\prod_{n\in \NN}\alpha_n\Bigr)=\sum_{n\in \NN}\log \alpha_n.
\]
\end{lemma}

\begin{proof}
(i)$\Rightarrow$(ii): Assume $\prod_{n\in \NN} \alpha_n$ is unconditionally convergent to a value in $(0,\infty)$.
Then for every finite $I\subset \NN$,
\[
\log\left( \prod_{n\in I}\alpha_n\right) = \sum_{n\in I}\log\alpha_n.
\]
Since $\log:(0,\infty)\to\mathbb R$ is continuous, the net $I\mapsto\sum_{n\in I}\log\alpha_n$ converges in
$\mathbb R$. Hence $\sum_{n\in \NN}\log\alpha_n$ converges unconditionally and therefore absolutely, so
$\sum_{n\in \NN}|\log\alpha_n|<\infty$.

(ii)$\Rightarrow$(i): If $\sum_{n\in \NN}|\log\alpha_n|<\infty$, then for finite $I\subset \NN$,
$\prod_{n\in I}\alpha_n=\exp(\sum_{n\in I}\log\alpha_n)$, and the unconditional convergence of
$\sum_{n\in \NN}\log\alpha_n$ implies unconditional convergence of the product to a nonzero limit.

(ii)$\Leftrightarrow$(iii) is immediate from $|\log\alpha|=(\log\alpha)_+-(\log\alpha)_-$.

(iii)$\Leftrightarrow$(iv) follows from what we already proved and the equalities 
$$\log(\maxone{\alpha})=(\log\alpha)_+, \qquad \log(\minone{\alpha})=(\log\alpha)_-.
$$
The displayed identities then follow from the same logarithmic identities.
\end{proof}
\noindent

%
%
%

Let $(\alpha_n)_{n\in \NN} \subset [0,\infty)$, so that the sequence is allowed to take the value $0$. With the convention \(\log0=-\infty\), the product \(\prod_{n\in\NN} \alpha_n\) is
unconditionally convergent if and only if \(\sum_{n\in\NN} \log\alpha_n\) is
unconditionally convergent in the extended sense, and then
$$
\log \left(\prod_{n\in \NN} \alpha_n \right) 
=
\sum_{n\in \NN} \log \alpha_n.
$$



\subsection{Uniform product bounds near zero}\label{subsec:uniform-bounds-zero}


The converse similarity arguments in Section~\ref{sec:similarity-tensor} require quantitative control
of the complementary tensor factors near the origin. The purpose of this
subsection is to isolate the near-zero product estimates that can be recovered
from boundedness of the whole tensor product. 
More precisely, under local boundedness of the
factors and boundedness of the infinite product near zero, we show that
each individual factor is bounded near zero and then obtain a positive-measure
set on which the complementary products remain uniformly controlled. These
estimates will be the main bridge from tensor-product similarity to the
one-factor similarity arguments from Section~\ref{sec:sim-reg-th}. 

We first prove the following result which does not impose any measurability on each of the factors. 

\begin{lemma}\label{lem:unif_bound}
	For each $n\in \NN$, let $f_n:(0,\infty) \to [0,\infty)$ be a submultiplicative function that is locally bounded in $(0,\infty)$. 
	Assume that, for every \(t>0\), the product
	\[
	f(t):=\prod_{n\in\mathbb N} f_n(t)
	\]
	is unconditionally convergent to a finite value, that \(f\) is not
	identically zero on \((0,\infty)\), and that
	\[
	\limsup_{t\to0+} f(t)<\infty .
	\]
	Then the following holds:
	\begin{itemize}
		\item [(i)] For each finite subset $I \subset \NN$, we have that
		$$\limsup_{t \to 0^+} \left( \prod_{n \in I}  \maxone{f_n(t)} \right) \leq \limsup_{t\to 0^+} f(t).
		$$ 
		In particular, $f_n$ is bounded near $0$ for every $n\in \NN$. \medskip
		
		\item [(ii)] Let $\delta>0$ such that $\prod_{n\in \NN} f_n(\delta) \neq 0$, and let $\mu_\delta$ and $\nu_\delta$ be as in Section~\ref{subsec:periodic-means}. Then there exists $M>0$ such that
		$$\nu_\delta \bigg( \bigg\{ t \in (0,\delta)  \, : \, \prod_{\ell \neq n} f_\ell(t)\leq M \bigg\} \bigg) \geq \frac{1}{2}, \qquad n \in \NN.
		$$
	\end{itemize} 
\end{lemma}
\begin{proof}
	To prove~(i), let $\varepsilon>0$ and fix $b>0$ such that $f(2b) \neq 0$ and
	$$ f(t)  \leq (1 +\varepsilon) K, \qquad t \in (0,2b),
	$$
	where $K := \limsup_{t\to 0^+} f(t)$. Set also
	$$M_n := \sup_{t \in [b, 2b]} f_n(t), \qquad n \in \NN,
	$$
	and note that $M_n<\infty$ for all $n\in \NN$, since we assumed that $f_n$ is locally bounded on $(0,\infty)$. For every $t \in (0,2b)$,
	\begin{align*}
		0 \neq  f(2b)
		&\leq  f(t) f(2b - t)
		\leq (1+\varepsilon) K   f(t).
	\end{align*}
	Consequently,
	\[
	D:= \inf_{t \in (0,2b)} f(t) > 0.
	\]
	Next, let $I\subset \NN$ be finite. Then
	\begin{align*}
		(1+\varepsilon) K 
		&\geq f(t) 
		= \left( \prod_{n\in \NN} \minone{f_n(t)} \right)
		\left( \prod_{n\in \NN} \maxone{f_n(t)} \right)
	\end{align*}
	for all $t \in (0,2b)$. Consequently,
	\begin{align}\label{eq:minmax1}
		\left( \prod_{n\in \NN} \minone{f_n(t)} \right)
		&\left( \prod_{n\in \NN \setminus I}  \maxone{f_n(t)} \right) 
		\leq \frac{(1+\varepsilon) K}{\prod_{n\in I}  \maxone{f_n(t)}}, \qquad t \in (0,2b).
	\end{align}
	
	For each $t\in (0,2b)$, let $m_t$ be the least positive integer satisfying $m_t t\in [b, 2b)$, so $m_t \geq b / t$. Then
	\begin{align*}
		D
		\leq f (m_t t)  
		&= \left(\prod_{n\in \NN} \minone{f_n(m_t t)}\right)
		\left(\prod_{n\in \NN} \maxone{f_n(m_t t)} \right) \\
		&\leq \left(\prod_{n\in \NN} \minone{f_n(t)}\right)^{m_t}  \left(\prod_{n\in \NN \setminus I} \maxone{f_n(t)}\right)^{m_t}
		\left(\prod_{n\in I} \maxone{M_n} \right)
	\end{align*}
	for all $t \in (0,2b)$. It follows that
	\begin{equation}\label{eq:minmax2}
		\begin{aligned}
			&\left(\prod_{n\in \NN} \minone{f_n(t)} \right)
			\left(\prod_{n\in \NN \setminus I} \maxone{f_n(t)}\right) 
			\geq \left(\frac{D}{\prod_{n\in I} \maxone{M_n}} \right)^{1/m_t},
		\end{aligned}
	\end{equation}
	for every $t \in (0,2b)$. Hence, by \eqref{eq:minmax1} and \eqref{eq:minmax2},
	\begin{align*}
		\prod_{n\in I}  \maxone{f_n(t)}
		&\leq (1+\varepsilon) K
		\left(\frac{\prod_{n\in I} \maxone{M_n}}{D} \right)^{1/m_t},
		\qquad t \in (0,2b).
	\end{align*}
	Therefore
	\[
	\limsup_{t \to 0^+} \left( \prod_{n \in I} \maxone{f_n(t)} \right) \leq (1+\varepsilon ) K,
	\]
	and the claim follows since $\varepsilon>0$ was arbitrary. \medskip
	
	We now prove~(ii). Define
	$$g_n(t) := e^{-d_n t} f_n(t), \qquad t >0,
	$$
	where $d_n:= \frac{1}{\delta}\log f_n(\delta)$, so that $g_n(\delta) =1$ for all $n\in \NN$. Since 
	$$f(\delta) = \prod_{n\in \NN} f_n(\delta) \neq 0,
	$$
	where the product is unconditionally convergent, we obtain from Lemma~\ref{lem:maxmin} that
	$$c:=\sum_{n\in \NN} |d_n| < \infty,
	$$
	and if $d:= \sum_{n\in \NN} d_n$, then
	$$g(t) := \prod_{n\in \NN} g_n(t) = e^{-d t}  f(t), \qquad t > 0.
	$$
	
	Set
	\[
	M_0 := \sup_{t \in (0,\delta)} g(t),
	\]
	and note that $M_0 < \infty$ as 
	$$\limsup_{t\to 0^+} g(t) = \limsup_{t\to 0^+} f(t) < \infty.$$
	Indeed, boundedness of \(g\) on a sufficiently small interval near \(0\),
	together with the submultiplicativity of $g$, gives boundedness on every bounded
	interval.
	Next define
	$$A_n:=\{ t \in (0,\delta) \, : \, g_n(t) \geq 1 \}, \qquad n \in \NN.
	$$
	Note that
	\begin{align*}
		1 = g_n(\delta) \leq g_n(\delta - t) g_n(t)  < g_n(\delta - t), \qquad n\in \NN, \, t \in (0,\delta) \setminus A_n,
	\end{align*}
	that is, $\delta - \left((0,\delta) \setminus A_n \right) \subseteq A_n$ and hence
	\begin{align*}
		\nu_\delta \left(\delta - \left((0,\delta) \setminus A_n \right) \right) \leq \nu_\delta (A_n), \qquad n\in \NN.
	\end{align*}
	By the identities from Section~\ref{subsec:periodic-means}, we also have
	\begin{align*}
		\nu_\delta \left( \delta - \left((0,\delta) \setminus A_n \right) \right) = \nu_\delta \left( (0,\delta) \setminus A_n \right) = 1 - \nu_\delta (A_n), \qquad n \in \NN,
	\end{align*}
	and we conclude that $\nu_\delta(A_n) \geq 1/2$. It also follows that
	\begin{align*}
		\prod_{\ell \neq n} f_\ell(t) = \exp \left( t \sum_{\ell \neq n} d_\ell \right)  \prod_{\ell \neq n} g_\ell(t) \leq e^{c\delta} g(t) \leq e^{c\delta} M_0, \qquad t \in A_n,
	\end{align*}
	and the claim follows with $M := e^{c\delta} M_0.$
\end{proof}


Assuming that each of the factors is measurable, the lemma below provides a uniform bound, analogous to that of Lemma~\ref{lem:unif_bound}(i), but now valid across all factors simultaneously. While the bounds obtained for the $\limsup$ of finite subproducts do not, in general, extend to the infinite product, the additional regularity assumptions imposed below ensure that this uniform bound holds. Moreover, this estimate will be necessary for the proof of Lemma~\ref{strongContInfProp}.

\begin{lemma}\label{unifBoundProdLemma}
	For each $n\in \NN$, let $f_n:(0,\infty) \to [0,\infty)$ be a measurable, submultiplicative function. Assume that, for every \(t>0\), the product
	\[
	f(t):=\prod_{n\in\mathbb N} f_n(t)
	\]
	is unconditionally convergent to a finite value, that \(f\) is not
	identically zero on \((0,\infty)\), and that
	\[
	\limsup_{t\to0+} f(t)<\infty .
	\]
	Then, 
	\begin{equation*}
		\limsup_{t \to 0^+} \left( \prod_{n\in \NN} \maxone{f_n(t)} \right) \leq \limsup_{t\to0^+} f(t).
	\end{equation*}
\end{lemma}
\begin{proof}
	For each $t>0$, set 
	\begin{align*}
	f_1(t) := \prod_{n\in \NN}  \maxone{f_n(t)}, \quad f_2(t) := \prod_{n\in \NN} \minone{f_n(t)}.
	\end{align*}
	Then \(f(t)=f_1(t)f_2(t),\) and,  by Lemma~\ref{lem:vn1},  \(f(t)<\infty\) implies \(f_1(t)<\infty\) for every
	\(t>0\). As $f_1$ is a pointwise limit of finite partial products, it is measurable, and being submultiplicative it is bounded above on every compact subset of $(0,\infty)$. Note that the same arguments apply to $f$.
	
	Fix $\delta>0$ such that $f(2\delta) \neq 0$. We then obtain that
	\begin{equation*}
		f(2\delta) \leq f(2\delta-s) f(s), \qquad s \in (\delta, 2\delta),
	\end{equation*}
	and thus
	$$
	\inf_{s\in (\delta, 2\delta)} f(s) \geq \frac{f(2\delta)}{\sup_{t\in (\delta, 2\delta)} f(t)} > 0.
	$$
%
%
%
	
	Now let us see that $\lim_{t\to 0} f_2(t) = 1$ by contradiction. Otherwise, there exist $c \in (0,1)$ and $(t_n)_{n\in \NN} \subset (0,\infty)$ such that $\lim_{n\to \infty} t_n = 0$ and 
	$$f_2(t_n) \leq c, \qquad n \in \NN.
	$$
	For each $n \in \NN$ big enough, take $m_n \in \NN$ such that $m_n t_n \in (\delta, 2\delta)$, so $\lim_{n\to \infty} m_n  = \infty$. Then, for all $n\in \NN$ big enough,
	\begin{align*}
		f(m_n t_n) &= f_1(m_n t_n) \,  f_2 (m_n t_n) 
		\leq \left(\sup_{t\in (\delta, 2\delta)} f_1(t)\right) (f_2(t_n))^{m_n} 
		\leq \left(\sup_{t\in (\delta, 2\delta)} f_1(t)\right) c^{m_n},
	\end{align*}
	where we have used that $f_2(nt) \leq (f_2(t))^n$ for all $n\in \NN$ and $t>0$. The above would imply $\inf_{t\in (\delta, 2\delta)} f(t) = 0$, reaching a contradiction and proving $\lim_{t\to 0} f_2(t) = 1$. Hence, we conclude
	\begin{align*}
		\limsup_{t\to 0^+} f_1(t) = \limsup_{t\to 0^+} \frac{f(t)}{f_2(t)} = \frac{\limsup_{t\to 0^+} f(t)}{\lim_{t\to0^+} f_2(t)} = \limsup_{t\to 0^+} f(t) < \infty.
	\end{align*}
\end{proof}


\section{Infinite tensor products: definitions}\label{sec:tensor-def}


Having collected the numerical input in Section~\ref{sec:num-prod}, we now recall the complete
and incomplete infinite tensor products of Hilbert spaces and then pass to the corresponding
tensor products of operators and semigroups. The purpose of the present section is mainly
foundational: we fix the language, notation, and basic constructions that will be used in the
rest of the paper. Together with Section~\ref{sec:incomplete-operators-semigroups}, this yields
a workable framework for infinite tensor products of Hilbert spaces, operators, and semigroups,
adapted to the similarity questions studied later on.
\medskip

\noindent\textbf{Convention.} 
Unless explicitly stated otherwise, infinite sums are understood in the
unconditional sense, and infinite products are understood either in the
unconditional or in the quasi-convergent sense, as appropriate.
In particular, if $b_n\ge 1$ for all $n\in \NN$, then $\prod_{n\in \NN} b_n$ converges if and only if it converges unconditionally, since the net of partial products is monotone. 
\medskip

We will also occasionally consider sums and products indexed by other countable sets (typically countable subsets of $\NN$); in such cases, the corresponding notions are defined in the natural way by identifying the index set with $\NN$ via a bijection.

\subsection{Infinite tensor products of Hilbert spaces}\label{InfiniteTensorProdDefSect}\label{InfiniteTensorSection}

We begin with the Hilbert-space level, since both the operator and the semigroup constructions
are built on it. The original construction of the complete and incomplete
tensor products of a family $(H_\alpha)_{\alpha \in I}$ of Hilbert spaces, where $I$ is an index set of arbitrary cardinality, is due to von Neumann
\cite{von1939infinite}. This construction was further developed by Guichardet
\cite{guichardet1966produits}, and the incomplete tensor products of von Neumann are often
referred to as Guichardet tensor products.

We restrict ourselves to countably infinite tensor products, since this is the relevant setting
for the operator- and semigroup-theoretic questions considered here; see
\cite[Th. 3.1.II (ii)]{nakagami1970infinite} and \cite[Th. 5.5]{gill1978infinite}. From now on, $(H_n)_{n\in \NN}$ is a family of Hilbert spaces. We will also occasionally consider tensor products indexed by other countable sets, say $I$. These tensor products are understood in the natural way through an identification of $I$ with $\NN$.


We say that a sequence
\[
(h_n)_{n\in\NN}\in\times_{n\in\NN}H_n
\]
is a \textit{\(C\)-sequence} for the family \((H_n)_{n\in\NN}\) if
\(\prod_{n\in\NN}\|h_n\|\) is unconditionally convergent to a finite value.
Given two \(C\)-sequences
\((h_n)_{n\in\NN},(f_n)_{n\in\NN}\in\times_{n\in\NN}H_n\), the product
\(\prod_{n\in\NN}\langle h_n,f_n\rangle\) is quasi-convergent. 
Then, each $C$-sequence $(h_n)_{n\in \NN}$ defines an ``element-wise'' anti-linear functional $\otimes_{n\in \NN} h_n$ on the space of $C$-sequences given by
$$\left(\motimes_{n\in \NN} h_n\right) \left( (f_n)_{n\in \NN}\right) = \prod_{n\in \NN} \langle h_n, f_n\rangle, \qquad (f_n)_{n\in \NN} \, \, C\mbox{-sequence for } (H_n)_{n\in \NN} .
$$
It is readily seen that the functional $\otimes_{n\in \NN} h_n$ satisfies the following desirable properties:
\begin{enumerate}
	\item [(i)] $\otimes_{n\in \NN} z_n h_n = \left(\prod_{n\in \NN} z_n \right) \left( \otimes_{n\in \NN} h_n\right)$ for every $C$-sequence $(h_n)_{n\in \NN}$ and for every $(z_n)_{n\in \NN} \subset \CC$ such that $\prod_{n\in \NN} z_n$ converges unconditionally to a value in $\CC$;
	\item [(ii)] if $(h_n)_{n\in \NN}$ and $(f_n)_{n\in \NN}$ are two $C$-sequences such that $h_n=f_n$ for all $n\in \NN$ except for $n=j \in \NN$, then
	$$\motimes_{n\in \NN} h_n + \motimes_{n\in \NN} f_n =   \motimes_{n\in \NN} g_n,
	$$
	where $g_n = h_n = f_n$ if $n\neq j$ and $g_j = h_j + f_j$.
\end{enumerate}
The functionals on the space of $C$-sequences of the form $\otimes_{n\in \NN} h_n$, where $(h_n)_{n\in \NN}$ is a $C$-sequence itself, are labeled as \textit{elementary tensors}. An inner product is defined on the linear span of elementary tensors by setting
\[
\left\langle \otimes_{n\in \NN} h_n, \otimes_{n\in \NN} f_n\right\rangle
=
\prod_{n\in \NN} \langle h_n, f_n\rangle,
\]
for $C$-sequences $(h_n)_{n\in \NN}$ and $(f_n)_{n\in \NN}$ for the family
\((H_n)_{n\in\NN}\). 
That this is indeed an inner product was proved in \cite[Th. II]{von1939infinite}. The completion of this space is the \textit{complete infinite tensor product} $\otimes_{n\in \NN} H_n$.

The \textit{incomplete infinite tensor products} are a special class of closed subspaces of $\otimes_{n\in \NN} H_n$, and are defined as follows. A $C$-sequence $(h_n)_{n\in \NN} \in \times_{n\in \NN} H_n$ is said to be a \textit{$C_0$-sequence} if 
$$\sum_{n\in \NN} |1 - \|h_n\| | < \infty.
$$
Then two $C_0$-sequences $(h_n)_{n\in \NN}, (f_n)_{n\in \NN} \in \times_{n\in \NN} H_n$ are said to be \textit{equivalent} if 
$$\sum_{n\in \NN} |1- \langle h_n, f_n \rangle | < \infty,
$$
and then we write 
$$(h_n)_{n\in \NN} \sim (f_n)_{n\in \NN}.
$$ 
That this defines an equivalence relation on the set of $C_0$-sequences was proved by von Neumann in \cite[Lemma 3.3.3]{von1939infinite}. This equivalence relation guarantees that the product $\prod_{n\in \NN} \langle h_n, f_n\rangle$ is unconditionally convergent to a finite value, see \cite[Section 2.4]{von1939infinite}. 

Let $\Gamma$ denote the set of all those equivalence classes. Given $\mathcal A \in \Gamma$, the \textit{incomplete infinite tensor product} $\otimes_{n\in \NN}^{\mathcal A} H_n$ (\textit{associated to $\mathcal A$}) is given by
\begin{align*}
	\motimes_{n\in \NN}^{\mathcal A} H_n := \overline{\operatorname{span}} \bigg\{\motimes_{n\in \NN} h_n \, : \, (h_n)_{n\in \NN} \in \mathcal A \bigg\}.
\end{align*}
The incomplete infinite tensor products $\otimes_{n\in \NN}^{\mathcal A} H_n$, where $\mathcal A$ runs over all equivalence classes $\Gamma$, are mutually orthogonal, and their closed linear span is $\otimes_{n\in \NN} H_n$, see \cite[Lemma 4.1.1]{von1939infinite}. 

Fixed $\mathcal A \in \Gamma$ and given a representative $\boldsymbol y=(y_n)_{n\in \NN}\in \mathcal A$,
we say that an elementary tensor $\otimes_{n\in \NN} h_n$ is \textit{finitely $\boldsymbol y$-elementary}
if $h_n=y_n$ for all but finitely many $n\in \NN$. Then the linear span of finitely $\boldsymbol{y}$-elementary tensors is dense in $\otimes_{n\in \NN}^{\mathcal A} H_n$, that is,
\[
\motimes_{n\in \NN}^{\mathcal A} H_n
=
\overline{\operatorname{span}}
\left\{
\motimes_{n\in \NN} h_n : \motimes_{n\in \NN} h_n \text{ is finitely }\boldsymbol y\text{-elementary}
\right\},
\]
see \cite[Lemma~4.1.4 \& Th.~V]{von1939infinite}. Following Guichardet's notation
\cite{guichardet1966produits}, we also write $\otimes_{n\in \NN}^{\boldsymbol y} H_n$ for
$\otimes_{n\in \NN}^{\mathcal A} H_n$.

We will use only a restricted form of associativity: if $F$ is a non-empty proper subset of $\NN$,
then the map initially defined on elementary tensors $\otimes_{n\in \NN} h_n$ by
\begin{equation}\label{eq:assoc_law1}
	\motimes_{n\in \NN} h_n \mapsto \Bigl(\motimes_{n\in F} h_n\Bigr)\otimes \Bigl(\motimes_{n\in \NN\setminus F} h_n\Bigr),
\end{equation}
extends to a canonical unitary identification
\[
\motimes_{n\in \NN} H_n \simeq 
\Bigl(\motimes_{n\in F} H_n\Bigr)\otimes
\Bigl(\motimes_{n\in \NN\setminus F} H_n\Bigr),
\]
see \cite[Theorem VII]{von1939infinite}. Moreover, the above identification preserves the equivalence classes of $C_0$-sequences and thus it also preserves the incomplete tensor products \cite[Theorem VI]{von1939infinite}. 
Consequently, given \(A\in\Gamma\), \((y_n)_{n\in\mathbb N}\in A\), 
the mapping above extends to a
canonical unitary identification
\begin{equation}\label{eq:assLawIncomplete}
\bigotimes_{n\in\mathbb N}^{A}H_n
\simeq
\left(\bigotimes_{n\in F}^{A_F}H_n\right)
\otimes
\left(\bigotimes_{n\in\mathbb N\setminus F}^{A_{\mathbb N\setminus F}}H_n\right),
\end{equation}
where \(A_F\) is the equivalence class, for the family \((H_n)_{n\in F}\),
containing \((y_n)_{n\in F}\), and \(A_{\mathbb N\setminus F}\) is defined in
the same way.



\subsection{Infinite tensor products of operators}\label{operatorTensorSubsection}

Having fixed the Hilbert-space framework, we now pass to infinite tensor products of bounded
operators. We first recall the complete tensor product, formulated in a way compatible with our
unconditional-product convention, and then turn to the incomplete setting associated with a
reference $C_0$-sequence.

Given $A_n \in \linearOp(H_n)$ for each $n\in \NN$,
we say that the complete tensor product $\motimes_{n\in \NN} A_n$ of $(A_n)_{n\in \NN}$ is \emph{well-defined} if, for every elementary tensor $\otimes_{n\in \NN} h_n$ in $\otimes_{n\in \NN} H_n$, $\otimes_{n\in \NN} A_n h_n$ is also an elementary tensor, and moreover the linear mapping defined on the linear span of elementary vectors by
$$\motimes_{n\in \NN} h_n \mapsto \motimes_{n\in \NN} A_n h_n,
$$
extends to a bounded operator on $\otimes_{n\in \NN} H_n$. We then denote such a bounded operator by
\[
\motimes_{n\in \NN} A_n.
\]
Note that if such extension exists, it is unique since the linear span of elementary tensors is dense in $\otimes_{n\in \NN} H_n$.

The existence and uniqueness of the complete tensor product $\otimes_{n\in \NN} A_n$ is characterized by Nakagami in \cite[Theorem 3.1]{nakagami1970infinite} in the more general case where the index set $I$ may have non-countable cardinality. Here, we restate his result for the special case in which $I$ is countably infinite, for which we first introduce some notation. In the countable case, the zero product case can be described explicitly in terms of
our unconditional-product convention, which is why we record the formulation
below.

Let $A_n = U_n |A_n|$ be the polar decomposition of $A_n$, so $U_n$ is a partial isometry and $|A_n|$ is a positive operator for each $n\in \NN$. It is shown in \cite[Lemma 3.1]{nakagami1970infinite} that the complete tensor product $\otimes_{n\in \NN} U_n$ is a well-defined partial isometry on $\otimes_{n\in \NN} H_n$. Using the associative law for complete tensor products, it follows that if $F\in \finiteN$, then 
$$A_F := \left(\motimes_{n\in F} A_n \right) \otimes \left(\motimes_{n\in \NN\setminus F} U_n\right)
$$
is a well-defined bounded operator on $\otimes_{n\in \NN} H_n$. Recall that $\finiteN$ denotes the net of finite subsets of $\NN$, ordered by inclusion.

\begin{proposition}\label{prop:nakagami}
	Let $(H_n)_{n\in \NN}$ be a family of Hilbert spaces and let $A_n \in \linearOp(H_n)$ for each $n\in \NN$. Then the following are equivalent:
	\begin{itemize}
		\item [(i)] the complete tensor product $\otimes_{n\in \NN} A_n$ is well-defined;
		\item [(ii)] the net $(A_F)_{F\in \finiteN}$ 
		converges strongly to a bounded operator in $\otimes_{n\in \NN} H_n$;
		\item [(iii)] the product $\prod_{n\in \NN} \|A_n\|$ converges unconditionally to a finite value.
	\end{itemize}
	If any of the above equivalent statements holds, then $\otimes_{n\in \NN} A_n$ is the strong limit of the net $(A_F)_{F\in \finiteN}$, and in addition
	\begin{equation}\label{eq:norm-complete}
		\|\otimes_{n\in \NN} A_n\| = \prod_{n\in \NN} \|A_n\|.
	\end{equation}
\end{proposition}
\begin{proof}
	(iii) $\implies$ (i): The claim is trivial if $\prod_{n\in \NN} \|A_n \|=0$, while the case $\prod_{n\in \NN} \|A_n\| \in (0,\infty)$ is proved in Theorem~3.1(II) of~\cite{nakagami1970infinite}. 
	
	(i) $\implies$ (iii): If $\otimes_{n\in \NN} A_n \neq 0$, this is proved in Theorem~3.1(II) of~\cite{nakagami1970infinite}. Thus, assume that $\otimes_{n\in \NN} A_n = 0$. Then for every $C_0$-sequence $(h_n)_{n\in\NN}$ we have that $(A_n h_n)_{n\in \NN}$ is a $C_0$-sequence satisfying
	\begin{equation}\label{eq:hnzero}
		\prod_{n\in \NN} \|A_n h_n\| = \|(\otimes_{n\in \NN} A_n) (\otimes_{n\in \NN} h_n) \|= 0.
	\end{equation}
	Suppose by contradiction that the product
	\(
	\prod_{n\in\mathbb N}\|A_n\|
	\)
	is not unconditionally convergent to a finite value.  By Lemma~\ref{lem:vn1}(i), this
	implies that
	\[
	A_n\ne0\quad(n\in\mathbb N),
	\qquad
	\sum_{n\in\mathbb N}(\|A_n\|-1)_+=\infty .
	\]
	Then, for each \(n\in\mathbb N\), we choose a unit vector \(g_n\in H_n\)
	such that
	\[
	A_ng_n\ne0,
	\qquad
	\|A_ng_n\|\geq \|A_n\|-2^{-n}.
	\]
	This is possible since \(A_n\ne0\).  Then
	\[
	\sum_{n\in\mathbb N}(\|A_ng_n\|-1)_+=\infty .
	\]
	Indeed,
	\[
	(\|A_n\|-1)_+
	\leq
	(\|A_ng_n\|-1)_+ + 2^{-n},\qquad n\in\NN,
	\]
	and \(\sum_{n\in\NN} 2^{-n}<\infty\). 
	By Lemma~\ref{lem:vn1}, this contradicts \eqref{eq:hnzero}, since
	\((g_n)_{n\in\NN}\) is a \(C_0\)-sequence.  Thus the factorwise tensor product
	\(\motimes_{n\in\NN} A_n\) cannot be well-defined, and this proves
	(i) \(\Rightarrow \) (iii).
	 \medskip

	(ii) $\iff$ (iii): This is Theorem~3.1(I) of~\cite{nakagami1970infinite}. 
	In particular, the zero strong-limit case is included here: if
	\((A_F)_{F\in\finiteN}\) converges strongly to the zero operator, then
	Nakagami's Theorem~3.1(I) implies that
	\(
	\prod_{n\in\NN}\|A_n\|
	\)
	converges unconditionally to a finite value.  If this value were nonzero,
	then the nonzero part of Nakagami's criterion, Theorem~3.1(II), would give a
	nonzero strong limit.  
	Hence the zero strong-limit case is precisely the zero product case.\medskip

	Finally, Theorem~3.1 in~\cite{nakagami1970infinite} also shows that whenever the net $(A_F)_{F\in \finiteN}$ converges in the strong operator topology to a bounded operator $A$, one has
	$$\|A\| = \prod_{n\in \NN} \|A_n\|,
	$$ 
	and, if $A \neq 0$, then
	$$A (\otimes_{n\in \NN} h_n) = \otimes_{n\in \NN} A_n h_n
	$$
	for every $C_0$-sequence $(h_n)_{n\in \NN}$. Equivalently, $A = \otimes_{n\in \NN} A_n$.
	\bigskip

\end{proof}

\begin{remark}\label{rem:tails}
	Given Hilbert spaces $(H_n)_{n\in \NN}$ and bounded operators $A_n \in \linearOp(H_n)$, $n\in \NN$, define
	$$\widetilde{A_F} := \left(\otimes_{n\in F} A_n \right) \otimes I, \quad F \in \finiteN. 
	$$
	Then the well-definedness of $\otimes_{n\in\NN} A_n$ does not imply that the net $(\widetilde{A_F})_{F\in\finiteN}$ is strongly convergent. 
	
	Indeed, set $A_n := -I$. Then
	$$\prod_{n\in \NN} \|A_n\|=1,
	$$
	and thus $\otimes_{n\in\NN} A_n$ is well-defined by Proposition~\ref{prop:nakagami} (in fact, it is a unitary involution). On the other hand,
	$$\widetilde{A_F} = \begin{cases}
		I, \quad & \mbox{if } |F| \mbox{ is even},
		\\ -I, \quad & \mbox{if } |F| \mbox{ is odd},
	\end{cases}
	$$
	so the net $(\widetilde{A_F})_{F\in\finiteN}$ is not strongly convergent.
\end{remark}

Now we turn to the definition of the incomplete tensor product of a sequence of bounded operators $(A_n)_{n\in \NN}$. Unfortunately, the literature does not provide a canonical definition. On the one hand, Guichardet proves in \cite[Prop. 6]{guichardet1969tensor} that certain conditions on $(A_n)_{n\in \NN}$ and on a $C_0$-sequence $\boldsymbol{y}=(y_n)_{n\in \NN}$ are sufficient to ensure that $\otimes_{n\in \NN} A_n$ is well-defined and leaves $\otimes_{n\in \NN}^{\boldsymbol{y}} H_n$ invariant. The incomplete tensor product of $(A_n)_{n\in \NN}$ is then defined as the restriction of  $\otimes_{n\in \NN} A_n$ to $\otimes_{n\in \NN}^{\boldsymbol{y}} H_n$, see~\eqref{eq:Gichardet} for details. This approach is used in several works, see, for instance, \cite{badea2017kazhdan,bedos2004infinite}. However, it is not known whether these sufficient conditions are also necessary. Moreover, it seems natural to ask whether one can define the incomplete tensor product of $(A_n)_{n\in \NN}$ on $\otimes_{n\in \NN}^{\boldsymbol{y}} H_n$ without requiring the boundedness of $\otimes_{n\in \NN} A_n$, since this is a condition imposed on the entire complete tensor product.
  
Hence, with the aim to clarify this situation, we propose two different definitions for the incomplete tensor product of $(A_n)_{n\in \NN}$. These two definitions are motivated by the characterization, given in Proposition~\ref{prop:nakagami}, of the well-definedness of the complete tensor product $\otimes_{n\in\NN} A_n$, which may be defined in two equivalent ways: either through its factorwise action on elementary tensors, or as the strong limit of a suitable net of bounded operators. 
The first definition is more convenient for the purposes of this paper, whereas
the second one is the incomplete-component analogue of Nakagami's strong-limit
point of view for complete tensor products.  Later, we relate these definitions
to the existing constructions in the literature and establish conditions under
which they coincide, see Subsection~\ref{InfOperatorsSubsection}.

Before that, note that we must consider a different net from $(A_F)_{F\in\finiteN}$, since a given incomplete tensor product need not be invariant under the action of $A_F$. With this in mind, let $\boldsymbol{y} = (y_n)_{n\in \NN} \in \times_{n\in \NN} H_n$ be a $C_0$-sequence and let $F \subset \NN$ be a finite set. Define $\boldsymbol y_{\NN \setminus F} := (y_n)_{n\in \NN\setminus F} \in \times_{n\in \NN\setminus F} H_n$, which is a $C_0$-sequence. Then, using the canonical unitary identification
\[
\motimes_{n\in \NN}^{\boldsymbol y} H_n
\simeq
\Bigl(\motimes_{n\in F} H_n\Bigr)\otimes
\Bigl(\motimes_{n\in \NN\setminus F}^{\boldsymbol{y}_{\NN \setminus F}} H_n\Bigr)
\]
we define the bounded operator
\[
A_F^{(\boldsymbol y)}:=\left(\motimes_{n\in F} A_n\right)\otimes I
 \in \linearOp \left(\motimes_{n\in \NN}^{\boldsymbol y} H_n \right).
\]
Equivalently, on finitely $\boldsymbol y$-elementary tensors $\otimes_{n\in \NN} h_n$ one has
\[
A_F^{(\boldsymbol y)}\!\left(\motimes_{n\in \NN}h_n\right)=\motimes_{n\in \NN}g_n,
\qquad
g_n=
\begin{cases}
	A_nh_n,& n\in F,\\
	h_n,& n\notin F.
\end{cases}
\]

\begin{definition}
	Let $(H_n)_{n\in\NN}$ be a sequence of Hilbert spaces and, for each $n\in\NN$, let $A_n\in\linearOp(H_n)$. Let $\boldsymbol{y} = (y_n)_{n\in \NN} \in \times_{n\in \NN} H_n$ be a $C_0$-sequence. 
	\begin{enumerate}
		\item [(i)] We say that the incomplete tensor product $\motimes_{n\in \NN}^{\boldsymbol y} A_n$ is \emph{well-defined in the factorwise sense} if, for every elementary tensor $\otimes_{n\in \NN} h_n \in \otimes_{n\in \NN}^{\boldsymbol y} H_n$, $\otimes_{n\in \NN} A_n h_n$ is an elementary tensor in $\otimes_{n\in \NN}^{\boldsymbol{y}} H_n$, and the linear mapping defined on the linear span of elementary tensors in $\otimes_{n\in \NN}^{\boldsymbol y} H_n$ by
		\[
		\motimes_{n\in \NN} h_n \mapsto  \motimes_{n\in \NN} A_n h_n
		\]
		extends to a bounded operator on $\otimes_{n\in \NN}^{\boldsymbol {y}} H_n$, which we denote by $\otimes_{n\in\NN}^{\boldsymbol y}A_n$.
		\item[(ii)] We say that the incomplete tensor product $\widehat{\motimes}_{n\in \NN}^{\boldsymbol y} A_n$ is \emph{well-defined in the strong-limit sense} if the net $(A_F^{(\boldsymbol y)})_{F\in\finiteN}$
		converges strongly to a bounded operator on $\motimes_{n\in \NN}^{\boldsymbol y} H_n$, which we denote by $\widehat\otimes_{n\in\NN}^{\boldsymbol y} A_n$.
	\end{enumerate}	
\end{definition}

To simplify notation, we write $\otimes_{n\in \NN} A_n \in \linearOp(\otimes_{n\in \NN} H_n)$ when the complete tensor product is well-defined. For the well-definedness of the incomplete tensor product, we use $\otimes_{n\in \NN}^{\boldsymbol y} A_n \in \linearOp(\otimes_{n\in \NN}^{\boldsymbol y} H_n)$ for the factorwise sense, and $\widehat\otimes_{n\in \NN}^{\boldsymbol y} A_n \in \linearOp(\otimes_{n\in \NN}^{\boldsymbol y} H_n)$ for the strong-limit sense.

We show in Theorem~\ref{incompleteOperatorTh} that if at least one of $\otimes_{n\in \NN}^{\boldsymbol{y}} A_n$ or $\widehat\otimes_{n\in \NN}^{\boldsymbol{y}} A_n$ is well-defined and nonzero, then so is the other one and
$$\otimes_{n\in \NN}^{\boldsymbol{y}} A_n = \widehat\otimes_{n\in \NN}^{\boldsymbol{y}} A_n.
$$
That is, both notions are equivalent in the nonzero case. Under these assumptions,  we further show that $\otimes_{n\in \NN}^{\boldsymbol{y}} A_n$ and $\widehat\otimes_{n\in \NN}^{\boldsymbol{y}} A_n$ agree with Guichardet's definition and satisfy the identity
\begin{equation}\label{eq:norm-incomplete}
	\|\otimes_{n \in \NN}^{\boldsymbol{y}} A_n\| = \|\widehat\otimes_{n \in \NN}^{\boldsymbol{y}} A_n\| = \prod_{n\in \NN} \|A_n\|.
\end{equation}
This will be crucial for our similarity arguments in Section~\ref{sec:similarity-tensor}, where the identities involving the operator norm of tensor products, \eqref{eq:norm-complete} and~\eqref{eq:norm-incomplete}, will play a fundamental role.

However, as we will see in Propositions~\ref{prop:counterEx1} and ~\ref{prop:counterEx2}, the factorwise definition and the strong-limit definition are not equivalent in general, and \eqref{eq:norm-incomplete} may fail too. In particular, the analogue of Proposition~\ref{prop:nakagami} fails in the setting of incomplete tensor products.

To finish this section, let us point out that finite tensor products may be regarded as a special case of incomplete tensors. Indeed, if
$H_1,\dots,H_N$ are Hilbert spaces, one chooses $y_n \in H_n$ for $n=1,\ldots, N$, and one sets $H_n=\mathbb C$ for $n>N$, with reference
vector $y_n=1\in\mathbb C$, then $\boldsymbol y := (y_n)_{n \in \NN} \in \times_{n\in \NN} H_n$ is a $C_0$-sequence and the map
\begin{equation}\label{eq:finite_tensors}
U:\ H_1\otimes\cdots\otimes H_N \to \motimes_{n \in \NN}^{\boldsymbol y} H_n,
\qquad
U(x)=x\otimes\Big(\motimes_{n>N} y_n\Big),
\end{equation}
extends to a unitary isomorphism. 

In the same way, finite tensor products of operators
and semigroups fit into the present framework as special incomplete infinite tensor products. More precisely, if $A_n \in \linearOp(H_n)$ for $1\le n\le N$ and $A_n=I_{\CC}$ for
$n>N$, then under the identification~\eqref{eq:finite_tensors} one has
\[
U^{-1}\left(\bigotimes_{n\in\NN}^{\boldsymbol{y}} A_n\right)U
=
\bigotimes_{n=1}^{N} A_n,
\]
This observation shows in particular that the results from \cite{OlivaMazaTomilovTensorI} are
consistent with, and naturally embedded into, the more general theory developed here.
Still, 
the proofs in Section~\ref{sec:similarity-tensor} do not use the
finite splitting theorem of \cite{OlivaMazaTomilovTensorI} as an input; instead, they rely on the one-semigroup
similarity theorem from Section~\ref{sec:sim-reg-th} together with the infinite tensor-product
framework developed in Sections~\ref{sec:tensor-def}--\ref{sec:incomplete-operators-semigroups}.

%
%

\section{Structure of infinite tensor products of operators and semigroups}\label{sec:incomplete-operators-semigroups}

We now pass from the basic definitions to the operator- and
semigroup-theoretic core of the infinite tensor-product framework. The purpose of this section is twofold. First, we isolate the existence,
non-vanishing, and regularity properties of incomplete tensor products that will be used later in
the similarity arguments. The discussion is organized from operators to semigroups and then to
the regularity issues that are specific to the infinite setting. Second, we study the structure of complete tensor products of semigroups and
show that strong measurability and \(C_0\)-regularity, when they hold on the
complete tensor product, impose highly restrictive structural conditions which
ultimately reduce the analysis to incomplete tensor-product components.

Subsection~\ref{InfOperatorsSubsection} first addresses incomplete infinite tensor products of operators.  Theorem~\ref{incompleteOperatorTh} characterizes when such a product defines a bounded, nonzero operator, both in the factorwise sense and in the strong-limit sense, and it relaxes sufficient conditions from the existing literature.  The remarks following Theorem~\ref{incompleteOperatorTh} clarify the relation with earlier criteria.  As an application of the operator criterion, Proposition~\ref{prop:equivalent-infinite-tensor} records the tensorization of equivalent Hilbertian norms, a tool used repeatedly in the similarity arguments of Section~\ref{sec:similarity-tensor}.

Subsection~\ref{InfSemigrSubsection} then turns to incomplete infinite tensor products of semigroups.  Using Theorem~\ref{incompleteOperatorTh}, Lemma~\ref{lem:nonZeroInfiniteSemigroup} characterizes when such a product defines a well-defined, nonzero semigroup.  We then show that weak measurability, strong measurability, and strong continuity at \(0\) pass from the factors to the incomplete tensor product under the natural nondegeneracy assumptions.  These results extend the corresponding unitary constructions to the present nonunitary setting, see, for instance, \cite[Th.~A]{arveson2001infinite}, \cite[Section~2]{reents1974infinite}, and \cite[Appendix~A]{badea2017kazhdan}.
	
Finally, in Subsection~\ref{subsec:complete-regularity}, we consider complete tensor products of semigroups.  Weak measurability of the factors is preserved under the complete tensor product, but stronger regularity behaves differently: strong measurability may fail even for uniformly continuous one-dimensional unitary factors.  More importantly, strong continuity on \((0,\infty)\) already imposes severe structural restrictions on the complete tensor product, and in the \(C_0\)-case the action decomposes along incomplete tensor-product components. Thus the complete \(C_0\)-theory does not replace the incomplete theory, rather, it reduces structurally to it.

\subsection{Existence and structure of infinite incomplete tensor products of operators}\label{InfOperatorsSubsection}

Let $(H_n)_{n\in \NN}$ be a sequence of Hilbert spaces. Given $A_n \in \linearOp(H_n)$ for $n \in \NN$, recall that $\otimes_{n\in\NN} A_n \in \linearOp\left( \otimes_{n \in \NN} H_n\right)$ if and only if $\prod_{n \in \NN} \|A_n \| < \infty$ (unconditionally), and then
$$\left\|\otimes_{n\in \NN} A_n \right\| = \prod_{n \in \NN} \|A_n\|,
$$
see  Proposition~\ref{prop:nakagami}. 


The situation is more delicate for incomplete tensor products. We show in Theorem~\ref{incompleteOperatorTh} that the analogue of Proposition~\ref{prop:nakagami} holds whenever $\otimes_{n\in \NN}^{\boldsymbol{y}} A_n \neq 0$ (in the factorwise sense) or if $\widehat\otimes_{n\in \NN}^{\boldsymbol{y}} A_n \neq 0$ (in the strong-limit sense), and that, in this case, the norm identity~\eqref{eq:norm-incomplete} also holds.
The nonvanishing assumption is essential, as shown by  Propositions~\ref{prop:counterEx1} and \ref{prop:counterEx2}. We then prove an identity relating the spectral radius of the infinite tensor product to those of each of the factors (Lemma~\ref{infiniteSpectralRadiusLemma}). This identity will play a crucial role in the similarity results of Section~\ref{sec:similarity-tensor}.

The following lemma shows that, under suitable assumptions, the inequality
$$\sum_{n\in \NN} (1 - \|h_n\|)_+ < \infty
$$
holds automatically. This observation will be used repeatedly throughout the paper (in particular, in Theorem~\ref{incompleteOperatorTh}) and slightly streamlines several arguments. 

\begin{lemma}\label{C0obviousLemma}
	Let $(H_n)_{n\in \NN}$ be a sequence of Hilbert spaces, let \((y_n)_{n\in\NN}\in\times_{n\in\NN}H_n\) be a \(C_0\)-sequence, and let
	\((h_n)_{n\in\NN}\in\times_{n\in\NN}H_n\) be such that 
	\begin{equation}\label{eq:hnEquivalent}
		\sum_{n\in \NN} (\|h_n\|-1)_+ < \infty, \qquad  \sum_{n\in \NN} | 1 - \langle y_n, h_n\rangle| < \infty. 
	\end{equation}
	Then $(h_n)_{n\in \NN}$ is a $C_0$-sequence equivalent to $(y_n)_{n\in \NN}$.
\end{lemma}
\begin{proof}
	We have that
	\begin{equation}\label{eq:C0-check}
	\begin{aligned}
		\sum_{n\in \NN} |1 - \langle y_n, h_n \rangle|
		&\geq \sum_{n\in \NN} (1 - |\langle y_n, h_n \rangle|)_+ 
		\geq \sum_{n\in \NN} (1-\|y_n\| \|h_n\|)_+ \\
		&\geq \sum_{n\in \NN} (1- \|h_n\|)_+
		+ \sum_{n\in \NN} \|h_n\| (1-\|y_n\|)_-. 
	\end{aligned}
	\end{equation}
	In view of~\eqref{eq:hnEquivalent}, we have $\sup_{n\in \NN} \|h_n\| < \infty$. Then, using that $(y_n)_{n\in \NN}$ is a $C_0$-sequence,
	we deduce that the second series in the last line in~\eqref{eq:C0-check} is finite. Hence,
	$$\sum_{n\in \NN} (1 - \|h_n\|)_+ < \infty,
	$$
	so $\sum_{n\in \NN} | \|h_n\| - 1 | < \infty$ and our claim follows.
\end{proof}



\begin{theorem}\label{incompleteOperatorTh}
	Let $(H_n)_{n\in\NN}$ be a sequence of Hilbert spaces, let \(\boldsymbol y=(y_n)_{n\in\NN}\in\times_{n\in\NN}H_n\) be a
	\(C_0\)-sequence, and let $A_n\in\linearOp(H_n)$ for all $n\in\NN$. Then the following are equivalent:
	\begin{enumerate}
		\item[(a)] $\motimes_{n\in\NN}^{\boldsymbol y}A_n$ is well-defined (in the factorwise sense) and nonzero.
		\item[(b)]  $\widehat\motimes_{n\in\NN}^{\boldsymbol y}A_n$ is well-defined (in the strong-limit sense) and nonzero.
		\item[(c)] The following conditions hold:
		\begin{enumerate}
			\item[(i)] $A_n\neq0$ for all $n\in\NN$;
			\item[(ii)] $\prod_{n\in\NN} \maxone{\|A_n\|} <\infty$; and
			\item[(iii)] $\sum_{n\in\NN}|1-\langle A_ny_n,y_n\rangle|<\infty$.
		\end{enumerate}
	\end{enumerate}
	If these conditions hold, then $\motimes_{n\in\NN}A_n\in\linearOp\bigl(\motimes_{n\in\NN}H_n\bigr)$, the operator $\motimes_{n\in\NN}A_n$ leaves $\motimes_{n\in\NN}^{\boldsymbol y}H_n$ invariant, we have that
	$$\otimes_{n\in \NN}^{\boldsymbol{y}} A_n = \widehat\otimes_{n\in \NN}^{\boldsymbol{y}} A_n = \otimes_{n\in\NN}A_n\Big|_{\motimes_{n\in\NN}^{\boldsymbol y}H_n},
	$$ and
	\[
	\bigl\|\otimes_{n\in\NN}^{\boldsymbol y}A_n\bigr\|= \bigl\|\widehat\otimes_{n\in\NN}^{\boldsymbol y}A_n\bigr\|= 
	\prod_{n\in\NN}\|A_n\|.
	\]
\end{theorem}
\begin{proof}
	Assume first (c) holds. By (ii), $\otimes_{n\in \NN} A_n$ is a bounded operator on the complete infinite tensor product $\otimes_{n\in \NN} H_n$ and $\|\otimes_{n\in \NN} A_n \| = \prod_{n\in \NN} \|A_n\| \in [0, \infty)$; see Proposition~\ref{prop:nakagami}.

	We first note that under assumption \textup{(iii)} the product $\prod_{n\in\NN}\|A_n\|$ cannot vanish.
	Indeed, since $\boldsymbol{y}$ is a $C_0$-sequence, we have $\sum_{n\in\NN}|1-\|y_n\||<\infty$ and hence also
	$\sum_{n\in\NN}|1-\|y_n\|^2|<\infty$, so we may assume $\inf_{n\in \NN}\|y_n\|>0$ by replacing any zero vector in $(y_n)_{n\in \NN}$ if necessary.
	Define
	\[
	c_n:=\frac{|\langle A_n y_n,y_n\rangle|}{\|y_n\|^2},
	\qquad n\in\NN.
	\]
	By Cauchy--Schwarz, $c_n\le \|A_n\|$ for every $n\in \NN$. Moreover,
	\[
	|1-c_n|
	\le
	\frac{|1-\langle A_n y_n,y_n\rangle|}{\|y_n\|^2}
	+
	\frac{|1-\|y_n\|^2|}{\|y_n\|^2}, \qquad n \in \NN,
	\]
	so assumption \textup{(iii)} and the $C_0$ property of $\boldsymbol{y}$ yield $\sum_{n\in\NN}|1-c_n|<\infty$.
	Hence $\prod_{n\in\NN}c_n\in(0,\infty)$ by Lemma~\ref{lem:vn1}(ii), and therefore $\prod_{n\in\NN}\|A_n\|\ge \prod_{n\in\NN}c_n>0$.

	Let $\otimes_{n\in \NN} h_n$ be a finitely $\boldsymbol{y}$-elementary tensor, i.e.,  $h_n = y_n$ for all but finitely many $n \in \NN$. One has
	$$\prod_{n\in \NN} \maxone{\|A_n h_n\|} \leq \left(\prod_{n\in \NN} \maxone{\|A_n\|} \right) \left(\prod_{n\in \NN} \maxone{\|h_n\|} \right) < \infty.
	$$
	Thus $\sum_{n\in \NN} (\|A_n h_n\|-1)_+ < \infty$ by Lemma~\ref{lem:vn1}(i). Using (iii) we obtain that 
	\begin{equation*}
		\sum_{n\in \NN} |1-\langle A_n h_n, y_n\rangle | < \infty,
	\end{equation*}
	and thus it follows from Lemma~\ref{C0obviousLemma} that $(A_n h_n)_{n\in\NN}$ is a $C_0$-sequence equivalent to $\boldsymbol{y}$. As the finitely $\boldsymbol{y}$-elementary tensors are dense in $\otimes_{n\in \NN}^{\boldsymbol{y}} H_n$, we have that $\otimes_{n\in \NN}^{\boldsymbol{y}} H_n$ is a closed, invariant subspace of $\otimes_{n\in \NN} A_n$, and we conclude then that $\otimes_{n\in \NN}^{\boldsymbol{y}} A_n \in \linearOp\left(\otimes_{n\in \NN}^{\boldsymbol{y}} H_n\right)$ (in the factorwise sense) and that 
	\begin{equation}\label{eq:factorwiseRestr}
		\otimes_{n\in \NN}^{\boldsymbol{y}} A_n = \otimes_{n\in \NN} A_n \big|_{\otimes_{n\in \NN}^{\boldsymbol{y}} H_n}.
	\end{equation}
	
	On the other hand, let $\otimes_{n\in \NN} h_n$ be a finitely $\boldsymbol{y}$-elementary tensor. Then, by what we proved above, $(A_n h_n)_{n\in \NN}$ is $C_0$-sequence equivalent to $(y_n)_{n\in \NN}$ and hence to $(h_n)_{n\in \NN}$. Therefore, for each $\varepsilon>0$ there exists $F\in \finiteN$ such that
	\begin{equation}
		\|\otimes_{n\notin F} h_n - \otimes_{n\notin F} A_n h_n \| < \varepsilon,
	\end{equation}
	see \cite[Lemma 3.3]{nakagami1970infinite}. It then follows that, if $J \in \finiteN$ is such that $F\subseteq J$, then
	\begin{align*}
		\| A_J^{(\boldsymbol{y})} \left(\otimes_{n\in \NN} h_n\right) - \otimes_{n\in \NN}^{\boldsymbol{y}} A_n h_n \| & = \|\otimes_{n\in J} A_n h_n \| \|\otimes_{n\notin J} h_n - \otimes_{n\notin J} A_n h_n \| \\
		& \leq \varepsilon \prod_{n\in \NN} \maxone{\|A_n h_n\|}.
	\end{align*}
	That is, the net 
	$$(A_J^{(\boldsymbol{y})} (\otimes_{n\in \NN} h_n))_{J\in\finiteN}
	$$ converges to the elementary tensor $\otimes_{n\in \NN} A_n h_n$. Since the linear span of finitely $\boldsymbol{y}$-elementary vectors is dense in $\otimes_{n\in \NN}^{\boldsymbol{y}} H_n$, and the net $(A_F^{(\boldsymbol{y})})_{F\in\finiteN}$ is uniformly norm bounded by $\prod_{n\in \NN} \maxone{\|A_n\|}$, it follows that $\widehat\otimes_{n\in \NN}^{\boldsymbol{y}} A_n$ is well-defined (in the strong-limit sense) and that
	$$\otimes_{n\in \NN}^{\boldsymbol{y}} A_n = \widehat\otimes_{n\in \NN}^{\boldsymbol{y}} A_n.
	$$

	Then, to prove that (c) implies both (a) and (b) we need to show that $\otimes_{n\in \NN}^{\boldsymbol{y}} A_n \neq 0$ in the factorwise sense. In fact, let us prove that
	\[
	\left\| \otimes_{n\in \NN}^{\boldsymbol{y}} A_n \right\|
	\geq \prod_{n\in \NN} \|A_n\|.
	\]
	 For each $N\in\NN$ and \(1\leq n\leq N\), choose
	\(f_n^N\in H_n\) such that
	\[
	\|f_n^N\|=1,
	\qquad
	\|A_n f_n^N\|\geq (1-N^{-2})\|A_n\| .
	\]
	Then
	\[
	(1-N^{-2})^N\prod_{n=1}^{N}\|A_n\|
	\leq
	\prod_{n=1}^{N}\|A_n f_n^N\|
	\leq
	\prod_{n=1}^{N}\|A_n\|, \qquad N \in \NN.
	\]
	Since \((1-N^{-2})^N\to1\), this gives
	$$\lim_{N\to \infty} \prod_{n=1}^N \|A_n f_n^N \| = \prod_{n\in \NN} \|A_n\|.
	$$
	On the other hand, since $(A_n y_n)_{n\in \NN}$ is a $C_0$-sequence, we have
	\[
	\lim_{N\to \infty} \prod_{n\geq N} \max\{\|A_n y_n\|, \|A_n y_n\|^{-1}\} = 1.
	\]
	Hence, for every $N \in \NN$, we deduce that
	\[
	g_N := (\otimes_{n\leq N} f_n^N) \otimes (\otimes_{n > N} y_n)\in \otimes_{n\in \NN}^{\boldsymbol{y}} H_n,
	\]
	and that
	\begin{align*}
		\lim_{N\to\infty} \left\| g_N \right\| 
		& = \lim_{N\to\infty} \left(\prod_{n > N} \| y_n \|\right) 
		= 1.
	\end{align*}
	Thus, writing
	\[
	u_N := (\otimes_{n\in\NN} ^{\boldsymbol{y}} A_n)  g_N
	= (\otimes_{n\leq N}  A_n f_n^N)\otimes (\otimes_{n > N} A_n y_n), \qquad N \in \NN,
	\]
	we have
	\begin{align*}
		\lim_{N\to\infty}\|u_N\|
		&= \lim_{N\to\infty}
		\left( \prod_{n=1}^N\|A_n f_n^N\| \right)
		\left(\prod_{n > N} \|A_n y_n \|\right) = \prod_{n\in \NN} \|A_n\|.
	\end{align*}
	We conclude $\left\| \otimes_{n\in \NN}^{\boldsymbol{y}} A_n \right\| \geq \prod_{n\in \NN} \|A_n\|$, which proves (a) and (b). Using~\eqref{eq:factorwiseRestr}, we deduce that
	$$\left\| \otimes_{n\in \NN}^{\boldsymbol{y}} A_n \right\| = \prod_{n\in \NN} \|A_n\|.
	$$ \bigskip

	Assume now that (a) holds. Since $ \otimes_{n\in \NN}^{\boldsymbol{y}} A_n \neq 0$ 
	and the elementary tensors are dense in $ \otimes_{n\in \NN}^{\boldsymbol{y}} H_n$, there exists a finitely $\boldsymbol{y}$-elementary tensor $\otimes_{n\in \NN} h_n$ such that $(A_n h_n)_{n\in \NN}$ is a nonzero $C_0$-sequence equivalent to $\boldsymbol{y}$.
	
	Thus (iii) follows, because $y_n = h_n$ for all but finitely many $n\in\NN$. One also deduces from this that $(A_n y_n)_{n\in \NN}$ is a $C_0$-sequence equivalent to $\boldsymbol{y}$.
	
	Clearly (i) holds since otherwise $\otimes_{n\in \NN}^{\boldsymbol{y}} A_n =0$. Let us prove by contradiction that (ii) holds. Suppose $\prod_{n\in \NN} \maxone{\|A_n\|}  = \infty$, so there exist an infinite subset $I \subset \NN$ and vectors $h_n \in H_n$ for $n \in I$ such that $\|h_n\| =1$ and $\prod_{n\in I} \|A_n h_n\| = \infty$. Since $(A_n y_n)_{n\in \NN}$ is a $C_0$-sequence (see the paragraph above), there exists $J \in \finiteN$ such that $\prod_{n\in \NN\setminus J} \max\left\{\|A_n y_n\|, \|A_n y_n\|^{-1} \right\} < \infty$. As $A_n\neq 0$ for all $n\in \NN$, by replacing the $J$-factors of $\otimes_{n\in \NN} y_n$ we may assume $\prod_{n\in \NN} \max\left\{\|A_n y_n\|, \|A_n y_n\|^{-1} \right\} < \infty$ (otherwise, it may happen that $\|A_n y_n\|=0$). Now, for each $N \in \NN$, define
	\begin{align*}
		f_n^N := \begin{cases}
			h_n, \qquad & n\in I \mbox{ and } n\leq N,
			\\ y_n, \qquad & \mbox{otherwise}.
		\end{cases}
	\end{align*}
	Then, for every $N\in \NN$, $\otimes_{n\in \NN} f_n^N$ is a finitely $\boldsymbol{y}$-elementary tensor such that
	\[
	\sup_{N\in \NN} \left\| \otimes_{n\in \NN} f_n^N \right\| \leq \prod_{n\in\NN} \maxone{\|y_n\|} < \infty
	\]
	and
	\[
	\lim_{N\to \infty} \left\| \otimes_{n\in \NN} A_n f_n^N \right\|
	\geq
	\left( \prod_{n\in \NN} \minone{\|A_n y_n\|} \right) \lim_{N\to \infty} \left( \prod_{n \leq N, \, n\in I} \|A_n h_n\| \right) 
	= \infty.
	\]
	This contradicts that $ \otimes_{n\in \NN}^{\boldsymbol{y}} A_n$ is bounded. We conclude that (ii) holds, and so does (c). \bigskip
	
	Finally, assume (b) holds. Clearly $A_n\neq 0$ for all $n\in \NN$, since otherwise the net $(A_F^{(\boldsymbol{y})})_{F\in\finiteN}$ would converge in the strong topology to the zero operator.

	Let us assume by contradiction that (ii) does not hold. Then we claim that there exists an increasing
	cofinal sequence of finite subsets \(F_N \in \finiteN\), where $N\in\NN$, such that
	\begin{equation}\label{eq:cofinalDiv}
		\prod_{n\in F_N}\|A_n\|\to\infty, \qquad \mbox{as } N\to \infty.
	\end{equation}
	
	To prove this, put
	\[
	b_n:=\log\|A_n\|,\qquad n\in\mathbb N .
	\]
	Since \(A_n\ne0\), all \(b_n\) are finite.  Moreover, it follows from Lemma~\ref{lem:maxmin} that
	\[
	\sum_{n\in\mathbb N} (b_n)_+=\infty .
	\]
	
	We now construct \(F_N\) inductively.  Suppose that \(F_{N-1}\) has been
	chosen.  First insert the index \(N\).  Since the positive tail
	\[
	\sum_{n\notin E} (b_n)_+
	\]
	is infinite after removing any finite set \(E\), we may add finitely many
	further indices and obtain a finite set
	\[
	F_N\supseteq F_{N-1}\cup\{N\}
	\]
	such that
	\[
	\sum_{n\in F_N}b_n\geq N .
	\]
	Then \((F_N)_{N\in\mathbb N}\) is increasing and cofinal in 
	$\finiteN$ since \(F_N\) contains \(\{1,\ldots,N\}\).
	Moreover,
	\[
	\prod_{n\in F_N}\|A_n\|
	=
	\exp\left(\sum_{n\in F_N}b_n\right)
	\geq e^N
	\to\infty,
	\]
	as claimed. \medskip

	Then, using the associativity of incomplete tensor products~\eqref{eq:assLawIncomplete}, we have that
	\begin{equation*}
		\|A_F^{(\boldsymbol{y})} \| = \prod_{n\in F} \|A_n\|, 
	\end{equation*}
	and thus
	\begin{equation}\label{eq:unboundedYNet}
		\|A_{F_N}^{(\boldsymbol{y})}  \|\to\infty, \qquad \mbox{as } N\to \infty.
	\end{equation}
	If the net \((A_F^{(\boldsymbol{y})})_{F\in\finiteN}\) converged strongly, then the cofinal sequence
	\((A_{F_N}^{(\boldsymbol{y})})_{N\in\mathbb N}\) would converge strongly as well.  Hence, for
	every $h \in \otimes_{n\in\NN}^{(\boldsymbol{y})} H_n$, the sequence
	\((A_{F_N}^{(\boldsymbol{y})} h)_{N\in\mathbb N}\) would be bounded.  By the uniform boundedness
	principle, the sequence \((A_{F_N}^{(\boldsymbol{y})})_{N\in\mathbb N}\) would then be norm
	bounded, contradicting the above limit.
	Consequently, the net \((A_F^{(\boldsymbol{y})})_{F\in\finiteN}\) does not converge in the strong operator
	topology, reaching the aforementioned contradiction.  Thus, we conclude that (ii) holds.


	Now fix any finitely $\boldsymbol{y}$-elementary tensor $\otimes_{n\in \NN} h_n$ such that the limit of the net 
	$$(A_F^{(\boldsymbol{y})} (\otimes_{n\in \NN} h_n))_{F\in\finiteN}
	$$
	is a nonzero vector.  
	Then $\otimes_{n\in \NN} h_n \neq 0$, and thus $\prod_{n\in\NN} \maxone{\|h_n\|} < \infty$ and $\prod_{n\in \NN} \minone{\|h_n\|} >0$. Moreover, the estimate (ii) implies that the product $\prod_{n\in \NN} \|A_n h_n\|$ converges unconditionally to a finite value, and we obtain that
	\begin{align*}
		\prod_{n\in \NN} \|A_n h_n\| &= \lim_{m\to\infty} \left(\prod_{n=1}^m \|A_n h_n\| \right) \left(\prod_{n>m} \|h_n\|\right) 
		\\ &= \lim_{m\to\infty} \| A_{\{1,\ldots,m\}}^{(\boldsymbol{y})} (\otimes_{n\in\NN} h_n) \| 
		= \left\|\left(\widehat\otimes_{n\in\NN}^{(\boldsymbol{y})} A_n\right) \left(\otimes_{n\in\NN} h_n\right) \right\| > 0.
	\end{align*}
	Consequently, $\prod_{n\in \NN} \minone{\|A_n h_n\|} > 0$ and we obtain that $(A_n h_n)_{n\in \NN}$ is a $C_0$-sequence. We then set
	$$C := \left(\prod_{n\in \NN} \minone{\|A_n h_n\|} \right) \left(\prod_{n\in \NN} \minone{\|h_n\|} \right) > 0.
	$$
	
	Let now $\varepsilon>0$ be arbitrary, and let $F\in \finiteN$ be
	such that, for every $J \in \finiteN$ with $F\cap J = \emptyset$, 
	we have
	$$\left\| \left( A_F^{(\boldsymbol{y})} - A_{F\cup J}^{(\boldsymbol{y})} \right) \left(\otimes_{n\in \NN} h_n\right) \right\| < \varepsilon. 
	$$
	Then, using the associative law for incomplete tensor products, we obtain that
	\begin{align*}
		\varepsilon &> \left\| \otimes_{n\in F} A_n h_n \right\| \left\| \otimes_{n \notin F \cup J} h_n \right\| \left\| \otimes_{n\in J} A_n h_n - \otimes_{n \in J} h_n\right\|
		\geq C \left\| \otimes_{n\in J} A_n h_n - \otimes_{n \in J} h_n\right\|,
	\end{align*} 
	for every $J\in\finiteN$ with $F\cap J = \emptyset$. Then \cite[Lemma 3.3]{nakagami1970infinite} yields that $(h_n)_{n\in \NN}$ and $(A_n h_n)_{n\in \NN}$ are equivalent $C_0$-sequences, that is, 
	$$\sum_{n\in \NN} | 1 - \langle A_n h_n, h_n \rangle| < \infty,
	$$
	and (iii) follows since $\otimes_{n\in \NN} h_n$ is a finitely $\boldsymbol{y}$-elementary tensor. Hence (c) holds, completing the proof.
	
\end{proof}

Theorem~\ref{incompleteOperatorTh} substantially improves the sufficient conditions obtained by
Guichardet for incomplete tensor products of operators.
To clarify this, let $(H_n)_{n\in \NN}$ be a sequence of Hilbert spaces, let $A_n \in \linearOp(H_n)$, and let $y_n \in H_n$ be such that $\|y_n\| = 1$ for $n\in \NN$. It is proven in \cite[Prop. 6]{guichardet1969tensor} (see also \cite[p. 472]{bedos2004infinite} and \cite[Sect. 2]{nakagami1974infinite}) that the following properties
\begin{enumerate}
	\item [(i)] $\prod_{n\in\NN} \|A_n\| < \infty$,
	\item [(ii)] $\sum_{n\in \NN} |1 - \langle A_n y_n, y_n\rangle| < \infty$, and
	\item [(iii)] $\sum_{n\in \NN} | 1 - \|A_n y_n\| | < \infty$,
\end{enumerate}
imply that $ \otimes_{n\in \NN} A_n$ leaves $ \otimes_{n\in \NN}^{\boldsymbol{y}} H_n$ invariant, and then one defines the incomplete tensor product of $(A_n)_{n\in\NN}$ as
\begin{equation}\label{eq:Gichardet}	
	\otimes_{n\in \NN} A_n \big|_{\otimes_{n\in \NN}^{\boldsymbol{y}} H_n}.
\end{equation}
Theorem~\ref{incompleteOperatorTh} shows that, in the nonzero case, these sufficient conditions are in fact intrinsic.  More precisely, if the incomplete tensor product of \((A_n)_{n\in\NN}\) is well defined and nonzero, either in the sense of \eqref{eq:Gichardet}, in the factorwise sense or in the strong-limit sense, then the norm-growth condition \textup{(i)} and the reference-vector condition \textup{(ii)} are necessary. Moreover, condition \textup{(iii)} follows automatically from them.
\bigskip

We record next a consequence of Theorem~\ref{incompleteOperatorTh} which will be used repeatedly in
Section~\ref{sec:similarity-tensor}.  It says that equivalent Hilbertian norms tensor provided the
corresponding distortion constants have a convergent product. Unlike the finite tensor-product case, we first have to check that the
admissible \(C\)-sequences and the corresponding incomplete components are
unchanged by the renorming.  Only after this identification can the tensor
product of the implementing operators be used to compare the norms.

\begin{proposition}[Tensorization of equivalent Hilbertian norms]\label{prop:equivalent-infinite-tensor}
	Let $(H_n)_{n\in\NN}$ be a sequence of Hilbert spaces and, for each $n\in\NN$, let
	$\|\cdot\|_{\mathscr H_n}$ be an equivalent Hilbertian norm on $H_n$ such that
	\[
	\|h\|_{H_n}\le \|h\|_{\mathscr H_n}\le \mathcal C_n\|h\|_{H_n},
	\qquad h\in H_n,
	\]
	where $\mathcal C_n := \mathcal C_{H_n}(\|\cdot\|_{\mathscr H_n})$. 
	Let $\mathscr H_n$ denote the underlying vector space of $H_n$ endowed with
	$\|\cdot\|_{\mathscr H_n}$, and set
	\[
	\mathcal C:=\prod_{n\in\NN}\mathcal C_n.
	\]
	Then the space of $C$-sequences in $\times_{n\in \NN} \mathscr H_n$ coincides with the
	space of $C$-sequences in $\times_{n\in \NN} H_n$ if and only if $\mathcal C<\infty$. If $\mathcal C<\infty$,
	then the following hold.
	\begin{enumerate}
		\item[(i)] The complete tensor products
		\[
		\otimes_{n\in\NN}H_n
		\qquad\text{and}\qquad
		\otimes_{n\in\NN} \mathscr H_n
		\]
		agree as vector spaces up to equivalent norms, and
		\[
		\mathcal C_{\otimes_{n\in\NN} H_n}\!\left(
		\|\cdot\|_{\otimes_{n\in\NN} \mathscr H_n}
		\right)
		= \mathcal C.
		\]
		
		
		\item[(ii)] The set of equivalence classes of $C_0$-sequences in
		$\times_{n\in \NN} \mathscr H_n$ coincides with the one of $C_0$-sequences in $\times_{n\in \NN} H_n$.
		Moreover, given a $C_0$-sequence $\boldsymbol y=(y_n)_{n\in\NN} \in \times_{n\in \NN} H_n$,
		the incomplete tensor products
		\[
		\otimes_{n\in\NN}^{\boldsymbol y}H_n
		\qquad\text{and}\qquad
		\otimes_{n\in\NN}^{\boldsymbol y}\mathscr H_n
		\]
		agree as vector spaces up to equivalent norms, and
		\[
		\mathcal C_{\otimes_{n\in\NN}^{\boldsymbol y}H_n}\!\left(
		\|\cdot\|_{\otimes_{n\in\NN}^{\boldsymbol y}\mathscr H_n}
		\right)
		= \mathcal C.
		\]
	\end{enumerate}
\end{proposition}
\begin{proof}
	It is straightforward to check that the set of $C$-sequences in
	$\times_{n\in \NN} H_n$ coincides with the one of $\times_{n\in \NN} \mathscr H_n$ if $\mathcal C<\infty$. Conversely, if \(\mathcal C=\infty\), for each $n\in \NN$ choose \(h_n\in H_n\) with
	\(\|h_n\|_{H_n}=1\) and
	\[
	\|h_n\|_{\mathscr H_n}\ge \frac{\mathcal C_n}{1+2^{-n}} .
	\]
	Then \((h_n)_{n\in \NN}\) is a \(C\)-sequence for \((H_n)_{n\in \NN}\), while
	\[
	\prod_{n\in\mathbb N}\|h_n\|_{\mathscr H_n}
	\ge
	\prod_{n\in\mathbb N}\frac{C_n}{1+2^{-n}}
	=\infty,
	\]
	and then \((h_n)_{n\in \NN}\) is not a \(C\)-sequence for \((\mathscr H_n)_{n\in \NN}\). Thus the two classes of \(C\)-sequences cannot coincide.\medskip

	
	Assume now $\mathcal C<\infty$, and for each $n\in\NN$ let $R_n\in\linearOp(H_n)$ be the
	positive operator such that
	\[
	\|h\|_{\mathscr H_n}=\|R_n h\|_{H_n},
	\qquad h\in H_n.
	\]
	Then $R_n$ is invertible, $\|R_n\|=\mathcal C_n$, and $\|R_n^{-1}\|=1$ for all $n\in\NN$.
	Hence, by Proposition~\ref{prop:nakagami}, the complete tensor product
	\[
	\mathcal R:=\otimes_{n\in\NN}R_n
	\]
	is a bounded invertible operator on $\otimes_{n\in\NN}H_n$, and
	\[
	\|\mathcal R\|=\mathcal C,
	\qquad
	\|\mathcal R^{-1}\|=1.
	\]
	Moreover, for every pair of \(C\)-sequences \((h_n)_{n\in\NN}\) and \((f_n)_{n\in\NN}\)
	for the family \((H_n)_{n\in\NN}\) one has
	\[
	\left\langle \otimes_{n\in\NN}h_n,\otimes_{n\in\NN}f_n
	\right\rangle_{\otimes_{n\in\NN}\mathscr H_n}
	=
	\left\langle
	\mathcal R(\otimes_{n\in\NN}h_n),\mathcal R(\otimes_{n\in\NN}f_n)
	\right\rangle_{\otimes_{n\in\NN}H_n}.
	\]
	Since the linear span of elementary tensors is dense, it follows that
	\[
	\|u\|_{\otimes_{n\in\NN}\mathscr H_n}
	=
	\|\mathcal R u\|_{\otimes_{n\in\NN}H_n},
	\qquad
	u\in \otimes_{n\in\NN}H_n.
	\]
	This yields that 
	$$\mathcal C_{\otimes_{n\in\NN} H_n}\!\left(
	\|\cdot\|_{\otimes_{n\in\NN} \mathscr H_n}
	\right) = \|\mathcal R\| \|\mathcal R^{-1}\| = \mathcal C.
	$$
	
	Now let $(h_n)_{n\in\NN},  (f_n)_{n\in\NN} \in \times_{n\in \NN} H_n$ be two equivalent $C_0$-sequences. Then
	\begin{equation}\label{eq:eqclasses}
		\begin{aligned}
			\sum_{n\in\NN}|1-\langle h_n,f_n\rangle_{\mathscr H_n}|
			&=
			\sum_{n\in\NN}|1-\langle R_n h_n,R_n f_n\rangle_{H_n}|
			\\
			&\le
			\sum_{n\in\NN}\Bigl(
			|\langle (R_n^2-I)h_n,f_n\rangle_{H_n}|
			+
			|1-\langle h_n,f_n\rangle_{H_n}|
			\Bigr)
			\\
			&\le \sum_{n\in\NN}(\|R_n\|^2-1)\|h_n\|_{H_n}\|f_n\|_{H_n}
			+ \sum_{n\in\NN}|1-\langle h_n,f_n\rangle_{H_n}|
			<\infty,
		\end{aligned}
	\end{equation}
	where we used that $\|R_n^2-I\|=\|R_n\|^2-1$ since $R_n\ge I$. Arguing in the reverse
	direction in the same way, one obtains the converse implication as well, so the sets of
	equivalence classes of $C_0$-sequences coincide.
	
	Let \(\boldsymbol y=(y_n)_{n\in\NN}\in\times_{n\in\NN}H_n\) be a
	\(C_0\)-sequence. Arguing as in~\eqref{eq:eqclasses}, one obtains
	\[
	\sum_{n\in\NN}|1-\langle R_n y_n,y_n\rangle_{H_n}|<\infty,
	\qquad
	\sum_{n\in\NN}|1-\langle R_n^{-1}y_n,y_n\rangle_{H_n}|<\infty.
	\]
	Hence, it follows from Theorem~\ref{incompleteOperatorTh} that $\mathcal R ^{\boldsymbol y} := \widehat\otimes_{n\in\NN}^{\boldsymbol{y}} R_n$, is well-defined and invertible, with
	$$(\mathcal R ^{\boldsymbol y})^{-1} = \widehat\otimes_{n\in\NN}^{\boldsymbol{y}} R_n^{-1},
	\qquad \quad  \|\mathcal R ^{\boldsymbol{y}} \| = \mathcal C , \qquad \quad \|(\mathcal R ^{\boldsymbol{y}})^{-1} \| = 1.
	$$
	Moreover, it follows from its factorwise action on elementary tensors that $\mathcal R^{\boldsymbol y}$ is the similarity operator between
	$\otimes_{n\in\NN}^{\boldsymbol y}H_n$ and $\otimes_{n\in\NN}^{\boldsymbol y}\mathscr H_n$.
	Therefore
	\[
	\mathcal C_{\otimes_{n\in\NN}^{\boldsymbol y}H_n}\!\left(
	\|\cdot\|_{\otimes_{n\in\NN}^{\boldsymbol y}\mathscr H_n}
	\right)
	= \|\mathcal R ^{\boldsymbol{y}} \| \|(\mathcal R ^{\boldsymbol{y}})^{-1} \| = \mathcal C,
	\]
	which proves \textup{(ii)}.

\end{proof}

The nonzero assumptions in Theorem~\ref{incompleteOperatorTh} are essential and cannot be weakened, and we illustrate this fact in the following results. First, in Propositions~\ref{prop:stongFactor} and \ref{prop:counterEx1} we prove that, in the vanishing case, the strong-limit definition is strictly stronger than, while not equivalent to, the factorwise one. In addition,  we construct in Proposition~\ref{prop:counterEx2} an example of a zero incomplete tensor product
$$\otimes_{n\in \NN}^{\boldsymbol{y}} A_n = \widehat\otimes_{n\in \NN}^{\boldsymbol{y}} A_n  =0,
$$
for which the identity \eqref{eq:norm-incomplete} fails, that is
$$\left\| \otimes_{n\in \NN}^{\boldsymbol{y}} A_n\right\| \neq \prod_{n\in \NN} \|A_n\|.
$$

\begin{proposition}\label{prop:stongFactor}
	Let $(H_n)_{n\in \NN}$ be a sequence of Hilbert spaces, let \(\boldsymbol y=(y_n)_{n\in\NN}\in\times_{n\in\NN}H_n\) be a
	\(C_0\)-sequence, and let $A_n \in \linearOp(H_n)$ be such that $\widehat\otimes_{n\in \NN}^{\boldsymbol{y}} A_n = 0$.
	Then 
	$$\otimes_{n\in \NN}^{\boldsymbol{y}} A_n = 0.
	$$
	Moreover, $\otimes_{n\in \NN} A_n \in \linearOp(\otimes_{n\in \NN} H_n)$ and
	$$\otimes_{n\in \NN}^{\boldsymbol{y}} H_n \subseteq \ker \left( \otimes_{n\in \NN} A_n\right),
	$$
	so, in particular,
	$$\otimes_{n\in \NN}^{\boldsymbol{y}} A_n = \widehat \otimes_{n\in \NN}^{\boldsymbol{y}} A_n = \otimes_{n\in \NN} A_n\Big|_{\otimes_{n\in \NN}^{\boldsymbol{y}} H_n}.
	$$
\end{proposition}
\begin{proof}
	The result is trivial if $A_n= 0$ for some $n\in \NN$, so assume otherwise. Let $\otimes_{n\in \NN} h_n$ be a nonzero elementary tensor in $\otimes_{n\in \NN}^{\boldsymbol{y}} H_n$. Then
	\begin{equation}\label{eq:zeroAnhn}
		\begin{aligned}
			\prod_{n\in \NN} \|A_n h_n\| &= \lim_F \left( \prod_{n\in F} \|A_n h_n\| \right) \left(\prod_{n\in \NN \setminus F} \|h_n\| \right)
			= \lim_F \left\|A_F^{(\boldsymbol{y})} h_n \right\| = 0,
		\end{aligned}
	\end{equation}
	where $\lim_F$ denotes the limit in the net $\finiteN$.
	Consequently, $(A_n h_n)_{n\in \NN}$ is a $C$-sequence and $\otimes_{n\in \NN} A_n h_n = 0$. Since the choice of $\otimes_{n\in \NN} h_n$ in $\otimes_{n\in \NN}^{\boldsymbol{y}} H_n$ was arbitrary, we obtain that $\otimes_{n\in \NN}^{\boldsymbol{y}} A_n = 0$ in the factorwise sense.

	Since we assumed that $A_n\neq 0$ for all $n\in \NN$, then arguing as in the discussion following~\eqref{eq:unboundedYNet} we obtain that
	$$\prod_{n\in\NN} \maxone{\|A_n\|} <\infty.
	$$
	Therefore, Proposition~\ref{prop:nakagami} implies that $\otimes_{n\in \NN} A_n$ is well-defined, and we obtain from \eqref{eq:zeroAnhn} that every elementary tensor $\otimes_{n\in \NN} h_n$ of $\otimes_{n\in \NN}^{\boldsymbol{y}} H_n$ belongs to $\ker \left(\otimes_{n\in \NN} A_n \right)$. Thus, by a density argument,
	$$\otimes_{n\in \NN}^{\boldsymbol{y}} H_n \subseteq \ker \left( \otimes_{n\in \NN} A_n\right),
	$$
	and the proof is finished.
\end{proof}
\begin{proposition}\label{prop:counterEx1}
	There exist a sequence of Hilbert spaces $(H_n)_{n\in \NN}$, bounded operators $A_n\in \linearOp(H_n)$, $n\in \NN$, and a $C_0$-sequence $\boldsymbol{y} = (y_n)_{n\in \NN} \in \times_{n\in \NN} H_n$ such that the following holds.
	\begin{enumerate}
		\item [(i)] $\otimes_{n\in \NN}^{\boldsymbol{y}} A_n = 0$ 
		(in the factorwise sense).
		\item[(ii)] $\widehat\otimes_{n\in \NN}^{\boldsymbol{y}} A_n$ is not well-defined 
		(in the strong-limit sense).
		\item [(iii)] $\otimes_{n\in \NN} A_n$ is not well-defined.
	\end{enumerate}
\end{proposition}
\begin{proof}
	Take Hilbert spaces $(H_n)_{n\in \NN}$ with $\dim H_n \geq 2$, let $A_n \in \linearOp(H_n)$ be non-injective such that $\|A_n\| =2$, and take $y_n\in \ker(A_n)$ with $\|y_n\|=1$ for all $n\in \NN$. Thus (iii) follows from Proposition~\ref{prop:nakagami} since the infinite product $\prod_{n\in \NN} \|A_n\|$ diverges to $\infty$. 
	
	Moreover, (iii) implies (ii). In fact, if $\widehat\otimes_{n\in \NN}^{\boldsymbol{y}} A_n$ was well-defined, then (iii) would contradict either Theorem~\ref{incompleteOperatorTh} if such an operator is nonzero, or Proposition~\ref{prop:stongFactor} otherwise.
	
	Finally, for every nonzero elementary tensor $\otimes_{n\in \NN} h_n$ in $\otimes_{n\in \NN}^{\boldsymbol y} H_n$, one has that $\lim_{n\to \infty} \|y_n-h_n\| = 0$ and then
	\begin{align*}
		\left\| \otimes_{n\in \NN} A_n h_n \right\| =  \prod_{n\in \NN} \left\| A_n h_n\right\| = \prod_{n\in \NN} \left\| A_n (h_n-y_n)\right\| = 0.
	\end{align*}
	
	Consequently $\otimes_{n\in \NN}^{\boldsymbol{y}} A_n= 0$, as claimed. 
\end{proof}
\begin{proposition}\label{prop:counterEx2}
	There exist a sequence of Hilbert spaces $(H_n)_{n\in \NN}$, bounded operators $A_n\in \linearOp(H_n)$, $n\in \NN$, and a $C_0$-sequence $\boldsymbol{y} = (y_n)_{n\in \NN} \in \times_{n\in \NN} H_n$ such that the following holds.
	\begin{enumerate}
		\item[(i)] $\otimes_{n\in \NN}^{\boldsymbol{y}} A_n = \widehat{\otimes}_{n\in \NN}^{\boldsymbol{y}} A_n= 0$.
		\item[(ii)] $\otimes_{n\in \NN} A_n$ is well-defined and nonzero.
	\end{enumerate}
	Consequently,
	$$\left\| \otimes_{n\in \NN}^{\boldsymbol{y}} A_n\right\| \neq \prod_{n\in \NN} \|A_n\|.
	$$
\end{proposition}
\begin{proof}
	Let $(H_n)_{n\in \NN}$ be a sequence of Hilbert spaces with $\dim H_n \geq 2$, let $A_n \in \linearOp(H_n)$ be non-injective and such that $\|A_n\| = 1$, and let $y_n \in \ker A_n$ with $\|y_n\| = 1$ for all $n\in \NN$.\newline 	
	For every non-empty $F \in \finiteN$ and every finitely $\boldsymbol{y}$-elementary tensor $\otimes_{n\in \NN} h_n$, we have 
	$$A_F^{(\boldsymbol{y})} \left(\otimes_{n\in \NN} h_n \right) = 0.
	$$
	Since the net $(A_F^{(\boldsymbol{y})})_{F\in\finiteN}$ is uniformly bounded and the linear span of finitely $\boldsymbol{y}$-elementary tensors is dense in $\otimes_{n\in \NN}^{\boldsymbol{y}} H_n$, it follows that $\widehat\otimes_{n\in \NN}^{\boldsymbol {y}} A_n = 0$, and then $\otimes_{n\in \NN}^{\boldsymbol {y}} A_n = 0$ by Proposition~\ref{prop:stongFactor}.
	\newline
	On the other hand, we have that $\prod_{n\in \NN} \|A_n\| = 1$. Hence, it follows from Proposition~\ref{prop:nakagami} that $\otimes_{n\in \NN} A_n$ is well-defined and nonzero. The claim on the norms follows trivially, and the proof is finished.
\end{proof}

Next we prove that the failure of the identity \eqref{eq:norm-incomplete} depicted in the proposition above cannot happen for powers of operators $(A_n^m)_{n\in \NN}$, with $m\in \NN$, provided that $\otimes_{n \in \NN}^{\boldsymbol{y}} A_n \neq 0$ (equivalently, $\widehat\otimes_{n\in \NN}^{\boldsymbol{y}} A_n \neq 0$). This observation will be useful to give a spectral radius formula and to deal with semigroups of operators.

\begin{lemma}\label{powersIncompleteLem}
Let $(H_n)_{n\in \NN}$ be a sequence of Hilbert spaces, let \(\boldsymbol y=(y_n)_{n\in\NN}\in\times_{n\in\NN}H_n\) be a
\(C_0\)-sequence. Take $A_n \in \linearOp(H_n)$ for $n\in \NN$ such that 
$$\otimes_{n\in \NN}^{\boldsymbol{y}} A_n \in  \linearOp \left(\otimes_{n\in \NN}^{\boldsymbol{y}} H_n\right) \setminus\{0\}.
$$
Then, for every $m\in \NN$, 
$$
\left(\otimes_{n\in \NN}^{\boldsymbol y} A_n\right)^m = \otimes_{n\in \NN}^{\boldsymbol y} A_n^m = \widehat\otimes_{n\in \NN}^{\boldsymbol y} A_n^m, \qquad 
\left\| \left(\otimes_{n\in \NN}^{\boldsymbol y} A_n\right)^m  \right\| = \prod_{n\in \NN} \|A_n^m\|,
$$
 and $(A_n^m y_n)_{n\in \NN}$ is a $C_0$-sequence equivalent to $\boldsymbol y$.
\end{lemma}
\begin{proof}On elementary tensors of $\otimes_{n\in \NN}^{\boldsymbol{y}} H_n$ one has the factorwise identity $(\motimes_{n\in\NN}^{\boldsymbol y}A_n)^m=\motimes_{n\in\NN}^{\boldsymbol y}A_n^m$ (by repeated application of the defining action). In particular, to study powers it suffices to consider the sequence $(A_n^m)_{n\in\NN}$ and apply Theorem~\ref{incompleteOperatorTh} to it.

 Let us prove the claim by complete induction. It follows from Theorem~\ref{incompleteOperatorTh} that the claim holds for every $m \in \NN$ such that $\left( \otimes_{n\in \NN}^{\boldsymbol{y}} A_n\right)^m \neq 0$. Thus, we may suppose it holds for $m=1,\ldots, N$, and that $\left( \otimes_{n\in \NN}^{\boldsymbol{y}} A_n\right)^{N+1} = 0$.
 %
 %
 %
 Note that 
 $$\prod_{n\in \NN} \maxone{\|A_n^\ast y_n\|} \leq \left(\prod_{n\in \NN} \maxone{\|A_n^\ast\|} \right) \left(\prod_{n\in \NN} \maxone{\|y_n\|} \right) < \infty,
 $$
 where we used that $\maxone{\|A_n\|} = \maxone{\|A_n^\ast\|}$ and $\prod_{n\in \NN} \|A_n\|< \infty$, see Theorem~\ref{incompleteOperatorTh}(c)(ii). As $(A_n y_n)_{n\in \NN}$ and $(A_n^N y_n)_{n\in \NN}$ are $C_0$-sequences equivalent to $\boldsymbol{y}$ by the induction assumption, it follows from Lemma~\ref{C0obviousLemma} and
 $$\sum_{n\in \NN} |1 - \langle A_n^\ast y_n, y_n\rangle| = \sum_{n\in \NN} |1 - \langle A_n y_n, y_n\rangle| < \infty,
 $$
 that $(A_n^\ast y_n)_{n\in\NN}$ is a $C_0$-sequence equivalent to $\boldsymbol{y}$, and then to $(A_n^N y_n)_{n\in \NN}$ by transitivity. Consequently,
 \begin{equation}\label{eqClassEq}
  \sum_{n\in \NN} | 1 - \langle  A_n^{N+1} y_n , y_n \rangle | = \sum_{n\in \NN} | 1 - \langle  A_n^N y_n , A_n^\ast y_n \rangle |<\infty.
 \end{equation}
Then,
\begin{equation}\label{anotherAuxEqInc}
	\prod_{n\in \NN} \maxone{\|A_n^{N+1}\|}
	\leq
	\left(\prod_{n\in \NN} \maxone{\|A_n\|} \right)^{N+1}
	< \infty.
\end{equation}
implies that
$\prod_{n\in \NN} \maxone{\|A_n^{N+1} y_n\|} < \infty$, and therefore $(A_n^{N+1} y_n)_{n\in \NN}$ is a $C_0$-sequence equivalent to $\boldsymbol{y}$ by Lemma~\ref{C0obviousLemma} and \eqref{eqClassEq}.

Regarding the identity involving the norms, \eqref{anotherAuxEqInc} and our assumption $\left( \otimes_{n\in \NN}^{\boldsymbol{y}} A_n\right)^{N+1} = 0$ imply that $A_\ell^{N+1} = 0$ for some $\ell\in \NN.$ In fact, otherwise Theorem~\ref{incompleteOperatorTh}~(c)(i) and~(c)(ii) would imply that $\otimes_{n\in \NN}^{\boldsymbol{y}} A_n \neq 0$. Then, we have that
$$\left\| \left(\otimes_{n\in \NN}^{\boldsymbol y} A_n\right)^m  \right\|= 0 = \prod_{n\in \NN} \|A_n^m\|.
$$ 
Therefore, the net given by
$$(A^m)^{(\boldsymbol{y})}_F = \left(\otimes_{n\in F} A_n^m \right) \otimes I, \quad F \in \finiteN,
$$
is strongly convergent to $0$, and hence $\widehat \otimes_{n\in\NN}^{\boldsymbol{y}} A_n^m =0$, finishing the proof.

 %

\end{proof}


In view of the preceding equivalence results, the factorwise and strong-limit definitions of incomplete tensor products agree in all cases relevant to the sequel. We therefore use the factorwise notation unless emphasis on the strong-limit construction is required.

We finish this subsection with a formula for the spectral radius of $\otimes_{n\in \NN} A_n$, which will be crucial for our similarity arguments in Section~\ref{sec:similarity-tensor}, in particular for Lemma~\ref{lem:omega_tensor}.

\begin{lemma}\label{infiniteSpectralRadiusLemma}
 Let $(H_n)_{n\in \NN}$ be a sequence of Hilbert spaces, and let $A_n \in \linearOp(H_n)$ for $n \in \NN$ be such that $\otimes_{n\in \NN} A_n \in \linearOp\left(\otimes_{n\in \NN} H_n\right)$. Then 
 $$r\left( \otimes_{n\in \NN} A_n \right) = \prod_{n\in \NN} r(A_n).
 $$
 In addition, let \(\boldsymbol y=(y_n)_{n\in\NN}\in\times_{n\in\NN}H_n\) be a
 \(C_0\)-sequence and assume that
$\otimes_{n\in \NN}^{\boldsymbol{y}} A_n \in \linearOp \left( \otimes_{n\in \NN}^{\boldsymbol{y}}H_n \right) \setminus\{0\}$. Then
$$r\left( \otimes_{n\in \NN}^{\boldsymbol{y}} A_n \right) = \prod_{n\in \NN} r(A_n).
$$
\end{lemma}
\begin{proof}
If $\prod_{n\in \NN}\|A_n\|=0$, then $\|\otimes_{n\in \NN}A_n\|=0$ and hence $\otimes_{n\in \NN}A_n=0$, so
$$r(\otimes_{n\in \NN}A_n)=0.
$$
Since $r(A_n)\le \|A_n\|$ for each $n\in \NN$, we also have 
$$\prod_{n\in \NN} r(A_n)=0,
$$
and the claim follows. Thus we may assume $0<\prod_{n\in \NN}\|A_n\|<\infty$.

For each $n\in \NN$ and $j\in \ZZ_+$, set
\[
a_j(n):=\frac{1}{2^j}\log\|A_n^{2^j}\|\in[-\infty,\infty).
\]
Since $\|A_n^{2^{j+1}}\|\le \|A_n^{2^j}\|^2$, we have $a_{j+1}(n)\le a_j(n)$ for all $n$ and $j$.
Moreover,
\[
\sum_{n\in \NN}|a_0(n)|=\sum_{n\in \NN}|\log\|A_n\||<\infty,
\]
because $0<\prod_{n\in \NN}\|A_n\|<\infty$; see Lemma~\ref{lem:maxmin}.
Hence, for each $j\in \ZZ_+$, the series $\sum_{n\in \NN}a_j(n)$ converges in $[-\infty,\infty)$ (in the unconditional sense).

Since $(\otimes_{n\in \NN}A_n)^{2^j}=\otimes_{n\in \NN}A_n^{2^j}$ and
$\bigl\|\otimes_{n\in \NN}A_n^{2^j}\bigr\|=\prod_{n\in \NN}\|A_n^{2^j}\|$, we have
\[
\log\Bigl(\bigl\|(\otimes_{n\in \NN}A_n)^{2^j}\bigr\|^{1/2^j}\Bigr)=\sum_{n\in \NN}a_j(n),
\qquad j\in \ZZ_+.
\]
Set $b_j(n):=a_0(n)-a_j(n)\in[0,\infty]$. Then $b_j(n)$ is monotone increasing in $j$ for each fixed $n$.
By the monotone convergence theorem for series of nonnegative terms,
\[
\lim_{j\to\infty}\sum_{n\in \NN}b_j(n)=\sum_{n\in \NN}\lim_{j\to\infty}b_j(n).
\]
Since $\lim_{j\to\infty}a_j(n)=\log r(A_n)$ (spectral radius formula along dyadic powers),
it follows that
\[
\lim_{j\to\infty}\sum_{n\in \NN}a_j(n)=\sum_{n\in \NN}\log r(A_n)\in[-\infty,\infty).
\]
Exponentiating and using $r(B)=\lim_{j\to\infty}\|B^{2^j}\|^{1/2^j}$, we obtain
\[
r\Bigl(\otimes_{n\in \NN}A_n\Bigr)
=
\exp\Bigl(\lim_{j\to\infty}\sum_{n\in \NN}a_j(n)\Bigr)
=
\exp\Bigl(\sum_{n\in \NN}\log r(A_n)\Bigr)
=
\prod_{n\in \NN}r(A_n),
\]
as claimed.

For the incomplete tensor product, argue in the same way, using the identity
\[
\Bigl\|\bigl(\otimes_{n\in \NN}^{\boldsymbol y}A_n\bigr)^m\Bigr\|=\prod_{n\in \NN}\|A_n^m\|,
\qquad m\in \NN,
\]
from Lemma~\ref{powersIncompleteLem} in place of the corresponding complete tensor product identity.
\end{proof}

\begin{remark}
	The nonvanishing assumption $\otimes_{n\in \NN}^{\boldsymbol{y}} A_n\neq 0$ cannot be omitted in the lemma above. In fact, let $A_n\in \linearOp(H_n)$ be as in the example constructed in Proposition~\ref{prop:counterEx2}, and such that $r(A_n) = 1$ for all $n\in \NN$. Then, $\otimes_{n\in \NN}^{\boldsymbol{y}}A_n = \widehat \otimes_{n\in \NN}^{\boldsymbol{y}}A_n= 0$ and
	$$r\left(\otimes_{n\in \NN}^{\boldsymbol{y}} A_n\right) = 0 \neq 1 = \prod_{n\in \NN} r(A_n). 
	$$
\end{remark}

\subsection{Infinite incomplete tensor products of semigroups: existence and regularity}\label{InfSemigrSubsection}

We now turn to incomplete tensor products of semigroups. Since no continuity is assumed a priori, the first issue is existence and non-vanishing of the operators $\motimes_{n\in\NN}^{\boldsymbol y}T_n(t)$ for fixed $t>0$. Once this is settled, one can discuss measurability or continuity of the tensor product family. Thus the subsection separates two issues: first, the pointwise existence and non-vanishing of
the operators \(\motimes_{n\in\NN}^{\boldsymbol y}T_n(t)\), and then the regularity of the
resulting semigroup near \(0\).
Recall that a semigroup $\mathcal T = (T(t))_{t\geq 0}$ is strongly continuous in $(0,\infty)$ if and only if it is strongly measurable if and only if it is weakly measurable and the mapping $t\mapsto T(t)h$ is almost separably-valued on $[0,\infty)$ for every $h \in H$, i.e., there exists a null set $\Omega \subset [0,\infty)$ such that the set $\{T(t)h \, : \, t \in [0,\infty)\setminus \Omega\}$ is separable in $H$, see for instance \cite[Sect. 10.2]{hille1957functional}. Also, note that the weak measurability of $\mathcal T$ does not necessarily imply that the mapping $t \mapsto \|T(t)\|$ is measurable, as explained in \cite[Sect. 10.2]{hille1957functional}. 

It will be useful for us to state the following immediate consequence of Lemma~\ref{powersIncompleteLem}, where we specialize Theorem~\ref{incompleteOperatorTh} to semigroups by applying it pointwise in time.
Here and below, saying that a semigroup $\mathcal T = (T(t))_{t\geq0}$ is
nonzero means that there exists $t>0$ such that $T(t)\neq 0$.

\begin{lemma}\label{lem:nonZeroInfiniteSemigroup}
 Let $(H_n)_{n\in \NN}$ be a sequence of Hilbert spaces and let \(\boldsymbol y=(y_n)_{n\in\NN}\in\times_{n\in\NN}H_n\) be a \(C_0\)-sequence. For each $n \in \NN$, let $\mathcal T_n = (T_n(t))_{t\geq0} \subset \linearOp(H_n)$ be a semigroup. Then $\otimes_{n\in \NN}^{\boldsymbol{y}} \mathcal T_n \subset \linearOp \left(\otimes_{n\in \NN}^{\boldsymbol{y}} H_n\right)$ is a nonzero semigroup if and only if there exists $\tau>0$ such that, for all $t\in [0,\tau]$,
 \begin{enumerate}
  \item [(i)] $T_n(t) \neq 0$ for all $n\in \NN$;
  \item [(ii)] $\prod_{n\in \NN} \maxone{\|T_n(t)\|} < \infty$; and
  \item [(iii)] $\sum_{n\in \NN} \left| 1 - \langle T_n(t) y_n, y_n\rangle \right|<\infty$.
 \end{enumerate}
 If {\normalfont (i), (ii)} and {\normalfont (iii)} hold then, for every $t>0$, 
 $$\otimes_{n\in \NN}^{\boldsymbol{y}} T_n(t) = \widehat\otimes_{n\in \NN}^{\boldsymbol{y}} T_n(t),
 $$  
 $(T_n(t)y_n)_{n\in \NN}$ is a $C_0$-sequence equivalent to $\boldsymbol{y}$, and moreover
 $$ \left\| \otimes_{n\in\NN}^{\boldsymbol{y}} T_n(t) \right\|= \prod_{n\in \NN} \|T_n(t)\|.
 $$
\end{lemma}
\begin{proof}

We first prove the equivalence.
Assume that
\(
\motimes_{n\in\NN}^{\boldsymbol y}\mathcal T_n
\)
is a nonzero semigroup, i.e. 
\(
\motimes_{n\in\NN}^{\boldsymbol y}T_n(t_0)\ne0
\)
for some \(t_0>0\).  The semigroup law then implies that
\[
\motimes_{n\in\NN}^{\boldsymbol y}T_n(t)\ne0,
\qquad 0<t\le t_0.
\]
Indeed, if the product were zero at some \(0<t\le t_0\), then writing
\(t_0=mt+r\), where \(m\in\NN\) and \(0\le r<t\), would force the product
at time \(t_0\) to be zero.  
Thus, for every \(0<t\le t_0\), applying
Theorem~\ref{incompleteOperatorTh} to the operators \(A_n=T_n(t)\) gives
conditions \textup{(i)}, \textup{(ii)} and \textup{(iii)}.  At \(t=0\)
these conditions are immediate.  
Hence the stated
small-time conditions hold, with \(\tau=t_0\).

Conversely, assume that there exists \(\tau>0\) such that conditions
\textup{(i)}, \textup{(ii)} and \textup{(iii)} hold for every
\(t\in[0,\tau]\).  For each \(t\in(0,\tau]\), Theorem~\ref{incompleteOperatorTh}
applied to \(A_n=T_n(t)\) shows that
\(
\motimes_{n\in\NN}^{\boldsymbol y}T_n(t)
\)
is well defined and nonzero, and
\[
\otimes_{n\in\NN}^{\boldsymbol y}T_n(t)
=
\widehat\otimes_{n\in\NN}^{\boldsymbol y}T_n(t).
\]
Moreover, the theorem shows that \((T_n(t)y_n)_{n\in\NN}\) is a \(C_0\)-sequence equivalent to
\(\boldsymbol y\), and 
\[
\left\|\motimes_{n\in\NN}^{\boldsymbol y}T_n(t)\right\|
=
\prod_{n\in\NN}\|T_n(t)\|,
\qquad 0<t\le\tau .
\]
At \(t=0\) the product is, of course, the identity on
\(\motimes_{n\in\NN}^{\boldsymbol y}H_n\).

It remains to pass from small times to arbitrary positive times.  Fix
\(t>0\), and choose \(m\in\NN\) such that \(s:=t/m\le\tau\).  Put
\(A_n:=T_n(s)\).  Since
\(
\motimes_{n\in\NN}^{\boldsymbol y}A_n
\)
is well defined and nonzero, Lemma~\ref{powersIncompleteLem} applied to the family
\((A_n)_{n\in\NN}\) yields
\[
\left(\motimes_{n\in\NN}^{\boldsymbol y}A_n\right)^m
=
\motimes_{n\in\NN}^{\boldsymbol y}A_n^m
=
\motimes_{n\in\NN}^{\boldsymbol y}T_n(t),
\]
and also
\[
\otimes_{n\in\NN}^{\boldsymbol y}T_n(t)
=
\widehat\otimes_{n\in\NN}^{\boldsymbol y}T_n(t),
\qquad
\left\|\otimes_{n\in\NN}^{\boldsymbol y}T_n(t)\right\|
=
\prod_{n\in\NN}\|T_n(t)\|.
\]
Moreover, the same application of Lemma~\ref{powersIncompleteLem} gives that
\((T_n(t)y_n)_{n\in\NN}\) is a \(C_0\)-sequence equivalent to
\(\boldsymbol y\).

It remains only to note that the operator obtained in this way is
independent of the particular integer \(m\).
Indeed,
if two choices are made, the resulting bounded operators agree on the dense
linear span of finitely \(\boldsymbol y\)-elementary tensor vectors, since on
such tensors both act by
\[
\motimes_{n\in\NN}h_n
\longmapsto
\motimes_{n\in\NN}T_n(t)h_n .
\]
The same density argument shows the semigroup law.  Therefore
\[
\motimes_{n\in\NN}^{\boldsymbol y}\mathcal T_n
\subset
\linearOp\left(\motimes_{n\in\NN}^{\boldsymbol y}H_n\right)
\]
is a well-defined nonzero semigroup, and the asserted identities hold for all
\(t>0\).
\end{proof}


Since the strong-limit construction preserves both weak measurability and strong measurability, the latter being equivalent to strong continuity on $(0,\infty)$, the previous result immediately yields the following corollaries. These provide natural sufficient conditions on the semigroups $\mathcal T_n = (T_n(t))_{t\geq0}$ ensuring that the semigroup $\otimes_{n\in \NN}^{\boldsymbol y} \mathcal T_n$ is either weakly measurable or strongly continuous on $(0,\infty)$. 

\begin{corollary}\label{cor:weakStrongMeasInf}
 Let $(H_n)_{n\in \NN}$ be a sequence of Hilbert spaces, and let $\boldsymbol{y} = (y_n)_{n\in \NN} \in \times_{n\in \NN} H_n$ be a $C_0$-sequence. For each $n\in \NN$, let $\mathcal T_n = (T_n(t))_{t\geq0} \subset \linearOp \left(H_n\right)$ be a semigroup such that 
 $$\motimes_{n\in\NN}^{\boldsymbol{y}} \mathcal T_n \subset \linearOp \left(\motimes_{n\in\NN}^{\boldsymbol{y}} H_n\right)$$
 is a nonzero semigroup. Then, the following holds.
 \begin{enumerate}
 	\item [(i)] If $\mathcal T_n$ is weakly measurable for each $n\in \NN$, then $\otimes_{n\in\NN}^{\boldsymbol{y}} \mathcal T_n$ is weakly measurable.
 	\item [(ii)] If $\mathcal T_n$ is strongly continuous in $(0,\infty)$ for each $n\in \NN$, then $\otimes_{n\in\NN}^{\boldsymbol{y}} \mathcal T_n$ is strongly continuous in $(0,\infty)$.
 \end{enumerate} 
\end{corollary}

The next result provides a characterization for $\otimes_{n\in \NN}^{\boldsymbol{y}} \mathcal T_n$ to be a $C_0$-semigroup, assuming each of the factors $\{\mathcal T_n\}_{n\in \NN}$ is itself a $C_0$-semigroup. Surprisingly, the local boundedness at $0$ of $\otimes_{n\in \NN}^{\boldsymbol{y}} \mathcal T_n$ is enough to obtain the $C_0$-property.

\begin{proposition}\label{strongContInfProp}
 Let $(H_n)_{n\in \NN}$ be a sequence of Hilbert spaces, let $\boldsymbol{y} = (y_n)_{n\in \NN} \in \times_{n\in \NN} H_n$ be a $C_0$-sequence, and let $\mathcal T_n = (T_n(t))_{t\geq0} \subset \linearOp \left(H_n\right)$ be a $C_0$-semigroup for each $n\in \NN$.
 Then the following are equivalent.
 \begin{itemize}
  \item [(a)] $\otimes_{n\in\NN}^{\boldsymbol{y}} \mathcal T_n \subset \linearOp \left(\otimes_{n\in\NN}^{\boldsymbol{y}} H_n\right)$ is a $C_0$-semigroup.
  \item [(b)] There exists $\tau>0$ such that $\otimes_{n\in \NN}^{\boldsymbol{y}} T_n(t) \in \linearOp\left(\otimes_{n\in\NN}^{\boldsymbol{y}} H_n \right) \setminus \{0\}$ for all $t\in [0,\tau]$, and 
  $$\displaystyle{\limsup_{t\to0^+} \left\| \otimes_{n\in\NN}^{\boldsymbol{y}}T_n(t)\right\| <\infty}.
  $$
  \item [(c)] There exist $\tau>0$ and $M>0$ such that, for all $t\in [0,\tau]$,
  \begin{enumerate}
   \item [(i)] $T_n(t)\neq 0$ for each $n\in \NN$;
   \item [(ii)] $\prod_{n\in \NN} \|T_n(t)\| \leq M$; and
   \item [(iii)] $\sum_{n\in \NN} | 1 -\langle T_n(t) y_n, y_n \rangle | < \infty$.
  \end{enumerate}
 \end{itemize}  
\end{proposition}
\begin{proof}
 
 (a) $\implies$ (b): The strong continuity at $0$ of $\otimes_{n\in\NN}^{\boldsymbol{y}} \mathcal T_n$ implies that $\left\| \otimes_{n\in\NN}^{\boldsymbol{y}}T_n(t)\right\|$ is bounded for small $t$ and that there exists $\tau>0$ such that $(T_n(t))_{n\in \NN}$ induces a nonzero bounded operator $\otimes_{n\in\NN}^{\boldsymbol{y}} T_n(t)$ on $\otimes_{n\in\NN}^{\boldsymbol{y}} H_n$ for every $t\in [0,\tau]$. \medskip 

 (b) $\iff$ (c): If the product is well-defined and nonzero, Lemma~\ref{lem:nonZeroInfiniteSemigroup} identifies
 \(\|\motimes_{n\in\NN}^{\boldsymbol{y}} T_n(t)\|\) with \(\prod_{n\in\NN}\|T_n(t)\|\); thus the bound in
 (ii) is written for the actual product norm. Once this is clarified, the equivalence follows from Lemma~\ref{lem:nonZeroInfiniteSemigroup}.

 \medskip

 (b) $\implies$ (a): 
 Using Lemma~\ref{lem:nonZeroInfiniteSemigroup} and the implication $(b) \implies (c)$, it follows that $\otimes_{n\in \NN}^{\boldsymbol{y}} \mathcal T_n$ is a nonzero semigroup on $\otimes_{n\in \NN}^{\boldsymbol{y}} H_n$. Let us see that $\otimes_{n\in \NN}^{\boldsymbol{y}} \mathcal T_n$ is non-degenerate, that is,
 \begin{equation}\label{nonDegSem}
  \bigcap_{t>0} \ker \motimes_{n\in \NN}^{\boldsymbol{y}} T_n(t) = \{0\},
 \end{equation}
 by contradiction. Thus, let $f \in \otimes_{n\in \NN}^{\boldsymbol{y}} H_n \setminus\{0\}$ be such that $\otimes_{n\in \NN}^{\boldsymbol{y}} T_n(t) f = 0$ for all $t>0$. For any $\varepsilon >0$, there exists a nonzero $h$ in the linear span of the finitely $\boldsymbol{y}$-elementary tensors such that $\|f - h\|\leq \varepsilon$. Choose $N_\varepsilon \in \NN$ so that $N_\varepsilon > 1/\varepsilon$ and
 $$h = g \otimes \left( \motimes_{n > N_\varepsilon} y_n\right),
 $$
 for some $g \in \otimes_{n=1}^{N_\varepsilon} H_n \setminus\{0\}$. 
 
 For each $n\in \NN$, set 
 $$f_n(t) := \|T_n(t)\|, \qquad t \in (0,\infty).
 $$
 Then each $f_n$ is a measurable, submultiplicative function on $(0,\infty)$ and, using that $\otimes_{n\in \NN} ^{\boldsymbol y} \mathcal T_n \neq 0$ and Lemma~\ref{lem:nonZeroInfiniteSemigroup}, we obtain that
 $$\limsup_{t\to 0^+} \prod_{n\in \NN} f_n(t) = \limsup_{t\to 0^+} \left\| \otimes_{n\in \NN}^{\boldsymbol{y}} T_n(t) \right\| < \infty. 
 $$
 It then follows from Lemma~\ref{unifBoundProdLemma} that
 $$\limsup_{t\to 0^+} \prod_{n\in \NN} \maxone{f_n(t)} 
 = \limsup_{t\to 0^+} \prod_{n\in \NN} \maxone{\|T_n(t)\|}  \leq \limsup_{t\to 0^+} \left\| \otimes_{n\in \NN}^{\boldsymbol{y}} T_n(t) \right\|,
 $$
 and hence, by the submultiplicativity of each $f_n$,
 $$M:= \sup_{t\in [0,1]} \prod_{n\in \NN} \maxone{\|T_n(t)\|}< \infty.
 $$
 Therefore,
 \begin{align*}
  \limsup_{t \to 0^+} \prod_{n>N_\varepsilon} \|T_n(t) y_n\|
  &= \frac{1}{\|g\|} \limsup_{t \to 0^+}
     \left( \left\|\otimes_{n=1}^{N_\varepsilon} T_n(t)g \right\|
     \prod_{n>N_\varepsilon} \|T_n(t) y_n\| \right) \\
  &= \frac{1}{\|g\|} \limsup_{t \to 0^+}
     \left\| \otimes_{n\in \NN}^{\boldsymbol{y}} T_n(t) h\right\| \\
  &= \frac{1}{\|g\|} \limsup_{t \to 0^+}
     \left\| \otimes_{n\in \NN}^{\boldsymbol{y}} T_n(t) (f-h)\right\| \leq \frac{M}{\|g\|} \varepsilon,
 \end{align*}
 where we used $\lim_{t\to0^+} \left\|\otimes_{n=1}^{N_\varepsilon} T_n(t)g \right\| = \|g\|$ since finite tensor products of $C_0$-semigroups are strongly continuous at $0$. For \(0\le t\le1\), submultiplicativity gives
 \[
 \prod_{n>N_\varepsilon}\|T_n(1)y_n\|
 \le
 \Big(\prod_{n>N_\varepsilon} \maxone{\|T_n(1-t)\|}\Big)
 \Big(\prod_{n>N_\varepsilon}\|T_n(t)y_n\|\Big)
 \le
 M\prod_{n>N_\varepsilon}\|T_n(t)y_n\|.
 \]
 Taking the infimum over \(t\in[0,1]\), and then using the preceding
 limsup estimate, yields
 \[
 \liminf_{N\to\infty}
 \prod_{n>N}\|T_n(1)y_n\|=0.
 \]
 This implies that $(T_n(1)y_n)_{n\in \NN}$ is not a $C_0$-sequence, reaching a contradiction with Lemma~\ref{lem:nonZeroInfiniteSemigroup}. Thus, \eqref{nonDegSem} holds. As $\limsup_{t\to0^+} \left\| \otimes_{n\in\NN}^{\boldsymbol{y}}T_n(t)\right\| <\infty$, $\otimes_{n\in \NN}^{\boldsymbol{y}} \mathcal T_n$ is strongly continuous in $(0,\infty)$ (see Corollary~\ref{cor:weakStrongMeasInf}(ii)) and $\otimes_{n\in \NN}^{\boldsymbol{y}} H_n$ is reflexive, we obtain that $\otimes_{n\in \NN}^{\boldsymbol{y}} \mathcal T_n$ is a $C_0$-semigroup; see, for instance, \cite[Prop. 3.3.8 \& Corollary 3.3.11]{arendt2011vector}.

\end{proof}

Proposition \ref{strongContInfProp} is a partial extension to our setting of the result given in \cite[Th. 2.5]{reents1974infinite} for infinite incomplete tensor products of unitary groups. On the other hand, let $\mathcal T_n \subset \linearOp \left(H_n\right)$ be a $C_0$-semigroup for each $n\in \NN$ such that $\otimes_{n\in \NN}^{\boldsymbol{y}} \mathcal T_n \subset \linearOp \left(\otimes_{n\in\NN}^{\boldsymbol{y}} H_n\right)$ is a nonzero semigroup and $\limsup_{t \to 0} \|\otimes_{n\in \NN}^{\boldsymbol{y}} T_n(t)\|<\infty$. It is well known that, if there exists $\tau>0$ such that  
\begin{equation}\label{GillEq}
 \sum_{n\in \NN} \sup_{t\in [0,\tau]} \|T_n(t)y_n - y_n\| < \infty,
\end{equation} 
then $\otimes_{n\in \NN}^{\boldsymbol{y}} \mathcal T_n$ is a $C_0$-semigroup; see \cite{arendt1998infinite, arveson2001infinite, gill1978infinite}, where it is shown that this fact holds even in the Banach space setting. Proposition~\ref{strongContInfProp} (ii) shows that, in Hilbert spaces, condition \eqref{GillEq} is redundant. \medskip 


%

We conclude this section by providing a simple example of an incomplete tensor product $C_0$-semigroup.


\begin{example}
	For each $n\in \NN$, let $H_n=L^2(\TT)$, $y_n= \chi_\TT$,
	and let $(U_n(t))_{t\in\RR}$ be the unitary shift semigroup on $L^2(\TT)$.  Since $U_n(t) y_n= y_n$ and $\|U_n(t)\| =1$ for all $n\in \NN$ and $t\in\RR$, then
	$$\left(\motimes_{n \in \NN}^{\boldsymbol{y}} U_n (t) \right)_{t\in\RR} \subset \linearOp \left(\motimes_{n \in \NN}^{\boldsymbol{y}} H_n\right)
	$$
	is a well-defined $C_0$-group by Proposition~\ref{strongContInfProp}. 
	Under the standard identification
	\[
	\motimes_{n\in\NN}^{\boldsymbol{y}} L^2(\TT)
	\simeq
	L^2(\TT^{\NN},m^{\otimes\NN}),
	\]
	it acts on functions $f\in L^2(\TT^{\NN},m^{\otimes\NN})$ by
	$$
	\left(\motimes_{n \in \NN}^{\boldsymbol{y}} U_n (t) \right) f 
	=
	f \circ \varphi_t, \qquad t \in \RR,
	$$
	where $\varphi_t: \TT^\NN \to \TT^\NN$ is the simultaneous shift in every coordinate, that is,
	$$\varphi_t \left(\times_{n\in\NN} z_n\right) = \times_{n\in\NN} (z_n + t), \qquad (z_n)_{n\in\NN} \subset \TT.
	$$
\end{example}

\subsection{Infinite complete tensor products of semigroups and \(C_0\)-regularity}\label{subsec:complete-regularity}

We now turn from incomplete tensor products to complete infinite tensor
products.  The contrast with Subsection~\ref{InfSemigrSubsection} is substantial.  In the incomplete
setting, the reference sequence fixes one component of the complete tensor
product, and the preceding results show that existence, non-vanishing,
measurability, and \(C_0\)-regularity can be described in terms of norm bounds
and compatibility with that reference sequence.  For complete tensor products
there is no distinguished component, and regularity of the full tensor product
semigroup imposes additional restrictions on how the semigroup may move
between incomplete tensor-product components.

This explains why complete tensor products of semigroups are much more rigid
than incomplete ones.  Weak measurability is still preserved by the complete
tensor product construction, and we begin with this positive result.
Examples~\ref{ex:strong-comp-counter-1} and~\ref{ex:strong-comp-counter-2} below illustrate
how regularity may be lost on the complete tensor product, even when the factors are very
regular.  The main point of the subsection, however, is the converse structural phenomenon:
once strong continuity is imposed, the positive-time dynamics is forced to respect the
decomposition into incomplete tensor-product components. In the \(C_0\)-case this yields the
orthogonal direct-sum splitting of
Theorem~\ref{cor:C0-complete-implies-invariant-incomplete}.

Consequently, the complete tensor product behaves differently from the
incomplete tensor product with respect to regularity.  In particular, one
should not expect strong measurability, and a fortiori strong continuity, of
the complete tensor product semigroup to follow merely from the corresponding
properties of the factors.  This is in contrast with the incomplete setting
of Corollary~\ref{cor:weakStrongMeasInf} and Proposition~\ref{strongContInfProp}.  The next result records the positive
part of the picture: weak measurability is preserved by the complete tensor
product construction.

\begin{lemma}\label{lem:weak-complete}
	Let $(H_n)_{n\in\NN}$ be a sequence of Hilbert spaces and, for each $n\in\NN$,
	let $\mathcal T_n=(T_n(t))_{t\ge 0}\subset\linearOp(H_n)$ be a weakly measurable semigroup such that
	\[
	\motimes_{n\in\NN}\mathcal T_n\subset \linearOp\!\left(\motimes_{n\in\NN}H_n\right)
	\]
	is a nonzero semigroup. Then $\otimes_{n\in\NN}\mathcal T_n$ is weakly measurable.
\end{lemma}
\begin{proof}
	Since the linear span of elementary tensors is dense in $\otimes_{n\in \NN} H_n$, it is enough to show that, for every pair of elementary tensors $\otimes_{n\in \NN} h_n$ and $\otimes_{n\in \NN} f_n$, the mapping from $[0,\infty)$ into $\CC$ given by
	\begin{equation*}
		t \mapsto \langle \otimes_{n\in \NN} T_n(t) h_n, f_n \rangle, \qquad t \geq 0,
	\end{equation*}
	is measurable.
	The claim is trivial if $\otimes_{n\in \NN} h_n = 0$ or $\otimes_{n\in \NN} f_n = 0$, so assume otherwise. Then $(h_n)_{n\in \NN}$ and $(f_n)_{n\in\NN}$ are $C_0$-sequences. Set
	$$E:= \left\{ t \in [0,\infty) \, : \, (T_n(t)h_n)_{n\in \NN} \, \mbox{ is a } C_0 \mbox{-sequence equivalent to } (f_n)_{n\in \NN}\right\},
	$$
	and let $g:[0,\infty) \to [0,\infty]$ be given by
	$$g(t):= \sum_{n\in \NN} | 1 - \langle T_n(t) h_n, f_n\rangle|, \qquad t \in [0,\infty).
	$$
	Then $g$ is a measurable function since it is the infinite sum of positive measurable functions. Furthermore, using Proposition~\ref{prop:nakagami} we have that
	\begin{equation}\label{eq:another-aux-eq}
		\prod_{n\in \NN} \maxone{\|T_n(t)\|} < \infty.
	\end{equation}
	Thus,
	$$\sum_{n\in \NN} (\|T_n(t) h_n\|-1)_+ <\infty \quad \mbox{for all } t\geq 0,
	$$
	and it follows from Lemma~\ref{C0obviousLemma} that
	\begin{equation*}
		E = \{t \geq 0 \, : \, g(t) < \infty\}.
	\end{equation*}
	Hence $E$ is a measurable set.

	Note that, for all $t \in [0,\infty) \setminus E$, then either $(T_n(t)h_n)_{n\in \NN}$ is a $C_0$-sequence not equivalent to $(f_n)_{n\in \NN}$, or $(T_n(t)h_n)_{n\in \NN}$ is not a $C_0$-sequence. In the former case, $\otimes_{n\in \NN} T_n(t) h_n$ and $\otimes_{n\in \NN} f_n$ are orthogonal elementary tensors. In the latter case, we obtain that  $\otimes_{n\in \NN} T_n(t) h_n = 0$. Indeed, $(T_n(t) h_n)_{n\in \NN}$ is a $C$-sequence by~\eqref{eq:another-aux-eq},
	and every $C$-sequence that induces a nonzero elementary vector is a $C_0$-sequence (see Lemma~\ref{lem:vn1}(ii) or \cite[Lemma 3.3.1]{von1939infinite}). Consequently,
	$$\langle\otimes_{n\in \NN} T_n(t) h_n, f_n \rangle = 0, \qquad t \in [0,\infty) \setminus E,
	$$
	and therefore
	\begin{equation*}\label{eq:weak-meas-E}
	\begin{aligned}
		t \mapsto \langle \otimes_{n\in \NN} T_n(t) h_n, f_n \rangle =  \chi_E(t)\prod_{n\in\mathbb N}
		\langle T_n(t) h_n,f_n\rangle, \qquad t \geq 0,
	\end{aligned}
	\end{equation*}
	where the above product is unconditionally convergent for every $t \in E$. Since each of the factors is a measurable function and $E$ is a measurable set, we deduce that the above function is measurable, completing the proof.
	
\end{proof}

The next lemma shows that strong continuity on $(0,\infty)$ forces each incomplete tensor product to be mapped into a single incomplete tensor product for all positive times. Hence, the behavior of the semigroup on the
complete infinite tensor product immediately reduces to the incomplete tensor product setting. To formalize this idea, we introduce the following notation. Given a complete tensor product semigroup $\otimes_{n\in \NN} \mathcal T_n$ on $\otimes_{n\in \NN} H_n$, we define
$$\Gamma_0 := \bigg\{ \mathcal A \in \Gamma \, : \, \motimes_{n \in \NN}^{\mathcal A} \mathcal T_n \mbox{ is a well-defined, nonzero semigroup}\bigg\}.
$$
Recall that $\Gamma$ denotes the set of equivalence classes of $C_0$-sequences.

\begin{lemma}\label{lem:complete-semigroup-target-incomplete}
Let $(H_n)_{n\in\NN}$ be a sequence of Hilbert spaces and, for each $n\in\NN$,
let $\mathcal T_n=(T_n(t))_{t\ge 0}\subset\linearOp(H_n)$ be a semigroup such that
\[
\motimes_{n\in\NN}\mathcal T_n\subset \linearOp\,\Big(\motimes_{n\in\NN}H_n\Big)
\]
is a nonzero semigroup which is strongly continuous on \((0,\infty)\). Then, for each
$\mathcal B \in \Gamma$, there exists 
$\mathcal A \in \Gamma_0$ such that
\[
\Big(\motimes_{n\in\NN}T_n(t)\Big)
\Big(\motimes_{n\in\NN}^{\mathcal B} H_n\Big)
\subseteq
\motimes_{n\in\NN}^{\mathcal A} H_n,
\qquad t>0.
\]
\end{lemma}

\begin{proof}
We first record a simple claim.

\smallskip
\noindent
\emph{Claim.} Let $\eta=\otimes_{n\in\NN}h_n$ be an elementary tensor such that
\[
\motimes_{n\in\NN}T_n(t)h_n\neq 0
\]
for some $t>0$. Then there exists $\mathcal A (\eta) \in \Gamma_0$
such that
\[
\motimes_{n\in\NN}T_n(t)h_n\in \motimes_{n\in\NN}^{\mathcal A(\eta)} H_n.
\qquad t>0,
\]

Indeed, define
\[
I_{\eta}:=\Big\{t>0:\ \motimes_{n\in\NN}T_n(t)h_n\neq 0\Big\}.
\]
Since $\otimes_{n\in\NN}\mathcal T_n$ is strongly continuous on $(0,\infty)$, the set
$I_{\eta}$ is open. It is also downward closed: if $s\in I_{\eta}$ and $0<t<s$, then
\[
\motimes_{n\in\NN}T_n(s)h_n
=
\Big(\motimes_{n\in\NN}T_n(s-t)\Big)
\Big(\motimes_{n\in\NN}T_n(t)h_n\Big),
\]
so $\otimes_{n\in\NN}T_n(t)h_n\neq 0$. Consequently,
\[
I_{\eta}=(0,t_{\eta})
\qquad \text{for some } t_{\eta}\in(0,\infty].
\]
Choose any $s_{\eta}\in I_{\eta}$. For every $t\in I_{\eta}$, the vector
$
\motimes_{n\in\NN}T_n(t)h_n
$
is a nonzero elementary tensor and therefore belongs to
\[
X\setminus\{0\},
\qquad
X:=\bigcup_{\mathcal B\in\Gamma}\,\,\motimes_{n\in\NN}^{\mathcal B} H_n.
\]
Moreover,
\[
X\setminus\{0\}
   =
\bigcup_{\mathcal B\in\Gamma}
\Big(\motimes_{n\in\NN}^{\mathcal B} H_n\setminus\{0\}\Big),
\]
and its connected components are precisely the sets
\[
\motimes_{n\in\NN}^{\mathcal B} H_n\setminus\{0\},
\qquad \mathcal B\in\Gamma.
\]
Indeed, if $\mathcal A, \mathcal B \in \Gamma$ with \(\mathcal A\ne \mathcal B\), then the spaces
\[
H_{\mathcal A}:=\motimes^{\mathcal A}_{n\in\mathbb N}H_n,
\qquad
H_{\mathcal B}:=\motimes^{\mathcal B}_{n\in\mathbb N}H_n
\]
are orthogonal. Hence, for \(0\ne \xi\in H_{\mathcal A}\), the ball
\(B(\xi,\|\xi\|/2)\) meets \(X\setminus\{0\}\) only in
\(H_{\mathcal A}\setminus\{0\}\). Thus each \(H_{\mathcal A}\setminus\{0\}\) is relatively open,
and its complement in \(X\setminus\{0\}\) is relatively open as well.
Since \(H_{\mathcal A}\setminus\{0\}\) is path connected over the complex field, these
sets are precisely the connected components of \(X\setminus\{0\}\).

Since
\[
I_{\eta}\ni t\longmapsto \motimes_{n\in\NN}T_n(t)h_n
\]
is continuous and $I_{\eta}$ is connected, its image is contained in one connected component
of $X\setminus\{0\}$. Hence, if we set $\mathcal A(\eta)$ to be the equivalence class containing
\[
\boldsymbol{x}:=(T_n(s_{\eta})h_n)_{n\in\NN},
\]
then
\[
\motimes_{n\in\NN}T_n(t)h_n\in \motimes_{n\in\NN}^{\mathcal A(\eta)}H_n,
\qquad t\in I_{\eta}.
\]
Combining the preceding argument on the nonzero part of the orbit with the
trivial zero case, we obtain the desired conclusion for every \(t>0\).

Next we show that $\mathcal A(\eta) \in \Gamma_0$. Choose $\tau_{\eta}>0$ so that $s_{\eta}+\tau_{\eta}<t_{\eta}$ if $t_{\eta}<\infty$, and choose
any $\tau_{\eta}>0$ if $t_{\eta}=\infty$. Then, for every $0\le r\le \tau_{\eta}$,
\[
(T_n(r)x_n)_{n\in\NN}=(T_n(r+s_{\eta})h_n)_{n\in\NN}
\sim (T_n(s_{\eta})h_n)_{n\in\NN}=\boldsymbol{x},
\]
so
\[
\sum_{n\in\NN}|1-\langle T_n(r)x_n,x_n\rangle|<\infty,
\qquad 0\le r\le \tau_{\eta}.
\]
Since the complete tensor product semigroup is not identically zero, the
semigroup law gives a number \(\tau_0>0\) such that
\[
\motimes_{n\in\NN}T_n(r)\ne0,
\qquad 0<r\le \tau_0 .
\]
Replacing \(\tau_\eta\) by \(\min\{\tau_\eta,\tau_0\}\), if necessary, and
using Proposition~\ref{prop:nakagami}, we obtain
\[
\prod_{n\in\NN}\|T_n(r)\|
=
\left\|\motimes_{n\in\NN}T_n(r)\right\|
\in(0,\infty),
\qquad 0<r\le\tau_\eta .
\]
Then Lemma~\ref{lem:nonZeroInfiniteSemigroup} shows that
\[
\motimes_{n\in\NN}^{\mathcal A(\eta)}\mathcal T_n
\]
is a well-defined nonzero semigroup on
\(\motimes_{n\in\NN}^{\mathcal A(\eta)}H_n\), and hence
\(\mathcal A(\eta)\in\Gamma_0\). This proves the claim. \medskip


We now continue with the proof of the lemma. Fix $\mathcal B \in \Gamma$ and a $C_0$-sequence $\boldsymbol{y} = (y_n)_{n\in \NN} \in \mathcal B$. If
\(
\motimes_{n\in\NN}T_n(t)h_n=0
\)
for every $t>0$ and every finitely $\boldsymbol{y}$-elementary tensor $\otimes_{n\in\NN}h_n$,
then
\[
\left(\motimes_{n\in\NN}T_n(t)\right)
\left(\motimes_{n\in\NN}^{\mathcal B}H_n\right)=\{0\},
\qquad t>0.
\]
Since $\otimes_{n\in\NN}\mathcal T_n$ is nonzero on the complete infinite tensor product,
there exists some elementary tensor $\eta_0$ with nonzero orbit. Applying the claim to $\eta_0$,
we obtain $\mathcal A(\eta_0) \in \Gamma_0$,
and the desired
inclusion is then trivial.

Assume now that there exists a finitely $\boldsymbol{y}$-elementary tensor
$
\eta:=\otimes_{n\in\NN}h_n
$
with nonzero orbit. Applying the claim to $\eta$, we obtain $\mathcal A(\eta) \in \Gamma_0$
such that
\[
\motimes_{n\in\NN}T_n(t)h_n\in \motimes_{n\in\NN}^{\mathcal A(\eta)}H_n.
\qquad t>0,
\]

Let
$
\zeta:=\otimes_{n\in\NN}g_n
$
be any other finitely $\boldsymbol{y}$-elementary tensor and fix $t>0$. 
If  $\otimes_{n\in \NN} T_n(t) \zeta = 0$, there is nothing to prove. Otherwise,  $(T_n(t) g_n)_{n\in \NN}$ is a $C_0$-sequence, which implies that 
$$(T_n(t)g_n)_{n\in \NN} \sim (T_n(t) h_n)_{n\in \NN},
$$
since the sequences $(g_n)_{n\in \NN}$ and $(h_n)_{n\in \NN}$ differ only in finitely many coordinates. Consequently, 
$$(T_n(t) g_n)_{n\in \NN} \in \mathcal A(\eta).
$$

Therefore, we conclude that for every $t>0$ and every finitely $\boldsymbol{y}$-elementary tensor $\zeta$,
\[
\motimes_{n\in \NN} T_n(t) \zeta \in 
\motimes_{n\in\NN}^{\mathcal A(\eta)}H_n.
\]
%
Since the linear span of finitely $\boldsymbol{y}$-elementary tensors is dense in
$\otimes_{n\in\NN}^{\mathcal B}H_n$, it follows that
\[
\left(\motimes_{n\in\NN}T_n(t)\right)
\left(\motimes_{n\in\NN}^{\mathcal B}H_n\right)
\subseteq
\motimes_{n\in\NN}^{\mathcal A(\eta)}H_n,
\qquad t>0.
\]

\end{proof}

Under the hypothesis of previous lemma, assume furthermore that 
$$\limsup_{t\to0^+} \|\otimes_{n\in \NN}  T_n(t)\| < \infty,
$$
and let $P \in \linearOp(\motimes_{n\in \NN} H_n)$ be the projection into the strong-continuity space of $\motimes_{n\in \NN} \mathcal T_n$. Then, it follows from Lemma~\ref{lem:complete-semigroup-target-incomplete} that 
$$\ran P \subseteq \bigoplus_{\mathcal A \in \Gamma_0} \bigg(\motimes_{n\in \NN}^{\mathcal A} H_n \bigg),
$$
and then
$$\motimes_{n\in \NN} T_n(t) = \bigg(\bigoplus_{\mathcal A \in \Gamma_0} \bigg(\motimes_{n\in \NN}^{\mathcal A} T_n(t) \bigg)\bigg) P, \qquad t>0.
$$
That is, $\motimes_{n\in \NN} \mathcal T_n$ factorizes through its restriction to the incomplete tensor products belonging to $\Gamma_0$.

Despite the restrictiveness imposed by strong continuity on $(0,\infty)$, it is not difficult to construct complete tensor product semigroups satisfying this regularity. The example below provides a canonical way to build such semigroups.

\begin{example}
	For each $n\in \NN$, let $H_n=L^2(\TT)$	and let
	\[
	T_n(t)=\e^{t\Delta_{\TT}},
	\qquad t\geq0,
	\]
	be the heat semigroup on $L^2(\TT)$.  Since $\|T_n(t)\|= 1$ for all $t\geq0$ and $n\in \NN$,
	$$\motimes_{n\in\NN} \mathcal T_n \subset \linearOp\left(\motimes_{n \in \NN} H_n\right)
	$$
	is a well-defined nonzero semigroup.
	Since \(e^{t\Delta_{\mathbb T}}\) is positive and injective for every \(t>0\), the
	partial isometry in its polar decomposition is the identity. Hence, Proposition~\ref{prop:nakagami}(ii) gives that
	\begin{equation*}
		\lim_{N\to \infty} \left(\motimes_{n=1}^N T_n(t)\right) \otimes I = \motimes_{n\in\NN} T_n(t), \qquad t\geq 0,
	\end{equation*}
	in the strong operator topology. As for every $N\in \NN$ the finite tensor $\otimes_{n=1}^N \mathcal T_n$ is a $C_0$-semigroup and the strong measurability is preserved under limits in the strong operator topology, we obtain that $\otimes_{n\in\NN} \mathcal T_n$ is a strongly measurable semigroup, thus strongly continuous in $(0,\infty)$.
\end{example}
	
%

The next surprising result, based on Lemma~\ref{lem:complete-semigroup-target-incomplete}, explains the additional rigidity forced by strong continuity at $0$.

\begin{theorem}\label{cor:C0-complete-implies-invariant-incomplete}
Let \((H_n)_{n\in\mathbb N}\) be a sequence of Hilbert spaces and, for each
\(n\in\mathbb N\), let $\mathcal T_n = (T_n(t))_{t\geq 0} \subset \linearOp(H_n)$ be a semigroup. Assume that $\otimes_{n\in\NN}\mathcal T_n$
is a $C_0$-semigroup on the complete infinite tensor product $H:=\otimes_{n\in\NN}H_n$. 
Then, for every equivalence class $\mathcal A\in\Gamma$, the corresponding incomplete tensor
product
\[
H_{\mathcal A}:=\motimes_{n\in\NN}^{\,\mathcal A}H_n
\]
is invariant under $\otimes_{n\in\NN}T_n(t)$ for all $t\ge 0$. In particular,
$$\mathcal T_\mathcal A:= \motimes_{n\in \NN}^{\mathcal A} \mathcal T_n
$$
is a well-defined $C_0$-semigroup on $H_\mathcal A$ that coincides with the restriction of $\otimes_{n\in \NN} \mathcal T_n$ to $H_\mathcal A$, and
\[
\motimes_{n\in\NN}\mathcal T_n=\bigoplus_{\mathcal A\in\Gamma} \mathcal T_{\mathcal A}.
\]
\end{theorem}

\begin{proof}
Fix an equivalence class $\mathcal B\in\Gamma$. By Lemma~\ref{lem:complete-semigroup-target-incomplete},
there exists $\mathcal A \in \Gamma_0$ such that
\[
\Big(\motimes_{n\in\NN}T_n(t)\Big)
\Big(\motimes_{n\in\NN}^{\mathcal B}H_n \Big)
\subseteq
\motimes_{n\in\NN}^{\mathcal A}H_n,
\qquad t>0.
\]
Assume that $\mathcal A \neq \mathcal B$. Then the incomplete tensor products
$\otimes_{n\in \NN}^{\mathcal A}H_n$ and $\otimes_{n\in \NN}^{\mathcal B}H_n$ are orthogonal. Choose a nonzero
vector $
\xi\in \otimes_{n\in\NN}^{\mathcal B}H_n$. For every $t>0$, we have
\[
\Big(\motimes_{n\in\NN}T_n(t)\Big)\xi
   \perp \xi,
\]
and therefore
\[
\left\|\left(\otimes_{n\in\NN}T_n(t)\right)\xi-\xi\right\|^2
  =
\left\|\left(\otimes_{n\in\NN}T_n(t)\right)\xi\right\|^2+\|\xi\|^2
  \ge \|\xi\|^2.
\]
This contradicts the strong continuity at $0$ of $\otimes_{n\in\NN}\mathcal T_n$. Hence
$\mathcal A = \mathcal B$, and so
\[
\Big(\motimes_{n\in\NN}T_n(t)\Big)
\Big(\motimes_{n\in\NN}^{\mathcal B}H_n\Big)
\subseteq
\motimes_{n\in\NN}^{\mathcal B}H_n,
\qquad t\ge 0.
\]
Since $\mathcal B$ was arbitrary, every incomplete tensor product $H_{\mathcal B}$ is invariant. Hence, it follows from Theorem~\ref{incompleteOperatorTh} and Lemma~\ref{lem:nonZeroInfiniteSemigroup} that $\mathcal T_\mathcal B := \otimes_{n\in \NN}^{\mathcal B} \mathcal T_n$ is a well-defined nonzero semigroup on $H_\mathcal B$ and that
$$\mathcal T_\mathcal B = \motimes_{n\in \NN} \mathcal T_n \bigg|_{H_\mathcal B}.
$$

Now recall that the complete tensor product decomposes as the orthogonal direct sum
\[
H=\bigoplus_{\mathcal A\in\Gamma}H_{\mathcal A}.
\]
Therefore each operator $\otimes_{n\in\NN}T_n(t)$ is block diagonal with respect to this
decomposition, which yields
\[
\motimes_{n\in\NN}T_n(t)=\bigoplus_{\mathcal A\in\Gamma}T_{\mathcal A}(t),
\qquad t\ge0.
\]
Finally, if $\xi\in H_{\mathcal A}$, then
\[
T_{\mathcal A}(t)\xi=\Big(\motimes_{n\in\NN}T_n(t)\Big)\xi\to\xi
\qquad (t\downarrow0),
\]
so each restricted family $\mathcal T_{\mathcal A}$ is a $C_0$-semigroup on $H_{\mathcal A}$.
\end{proof}

Theorem~\ref{cor:C0-complete-implies-invariant-incomplete} shows that, if one insists on a $C_0$-semigroup on the complete infinite tensor product, 
then it must decompose along incomplete tensor products. In this sense, nontrivial
infinite tensor product semigroup theory naturally belongs to the incomplete infinite tensor
product setting. 
This explains why, in the existing literature, infinite tensor products of
semigroups \(\mathcal T_n=(T_n(t))_{t\geq 0}\) on Hilbert spaces \(H_n\) are often not
treated directly on the complete infinite tensor product
$
\motimes_{n\in\mathbb N} H_n,
$
or are considered under assumptions which ensure that the product semigroup
preserves the relevant incomplete tensor products; see, for instance,
\cite{arveson2001infinite, badea2017kazhdan, bedos2004infinite,gill1978infinite, reents1974infinite}.  
Moreover, the converse of Theorem~\ref{cor:C0-complete-implies-invariant-incomplete} holds in the sense explained in the following remark.

\begin{remark}
	Let $(H_n)_{n\in \NN}$ be a sequence of Hilbert spaces, and for each $n\in \NN$, let $\mathcal T_n \subset \linearOp(H_n)$ be a semigroup. Assume that for every $\mathcal A \in \Gamma$,
	$$\mathcal T_\mathcal A := \motimes_{n\in \NN}^{\mathcal A} \mathcal T_n \subset \linearOp\left(\motimes_{n\in \NN}^{\mathcal A} H_n\right)
	$$
	is a $C_0$-semigroup. 
	Fix \(\mathcal A\in\Gamma\).  Since \(\mathcal T_\mathcal A\) is a \(C_0\)-semigroup, it is locally
	bounded near \(0\).  By the norm identity in Lemma~\ref{lem:nonZeroInfiniteSemigroup},
	\[
	\|T_\mathcal A (t)\|=\prod_{n\in\mathbb N}\|T_n(t)\|,
	\]
	for \(t>0\).  Hence \(\prod_{n\in\NN}\|T_n(t)\|\) is locally bounded as \(t\to0\).
	Then Proposition~\ref{prop:nakagami} yields that
	$$\mathcal T := \motimes_{n\in \NN} \mathcal T_n = \bigoplus_{\mathcal A \in \Gamma} \mathcal T_\mathcal A \subset \linearOp \left(\motimes_{n\in \NN} H_n\right),
	$$
	and a simple density argument on the linear span of elementary tensors shows that $\mathcal T$ is a $C_0$-semigroup.
\end{remark}
\medskip

Below we provide a simple example illustrating the loss of regularity of the complete tensor product, even when each of the factors is a unitary, uniformly continuous group.

\begin{example}\label{ex:strong-comp-counter-1}
	Set $H := \otimes_{n\in \NN} \CC$ and define 
	$$E(t) := \motimes_{n\in\NN} e^{i t}, \qquad t \in \RR,
	$$
	where we regard $e^{it}$ as a multiplication operator on $\CC$. 
	Then $\mathcal E = (E(t))_{t\in\RR}$ is the complete tensor product of unitary and uniformly continuous groups. 
	
	It is clear that, for each $t\in[0,2\pi)$, the $C_0$-sequence $\boldsymbol e_t :=(e^{it})_{n\in\NN}$ belongs to a different equivalence class. Hence, the elementary tensors 
	$$\otimes_{n\in \NN} e^{it} = E(t) \left(\otimes_{n\in \NN} 1\right), \qquad  t \in [0,2\pi),
	$$
	belong to different incomplete tensor products. Using Lemma~\ref{lem:complete-semigroup-target-incomplete},  we deduce that $\mathcal E$ is not strongly continuous in $(0,\infty)$.
\end{example} 

Still, if each semigroup \(\mathcal T_n\) is strongly continuous and \((h_n)_{n\in\NN}\in\times_{n\in\NN}H_n\) is a \(C_0\)-sequence, then the mapping given by
\begin{equation}\label{nonMeasMappingEq}
	t \mapsto \left\|\otimes_{n\in \NN} T_n(t) h_n\right\| = \prod_{n\in \NN} \|T_n(t) h_n\|, \qquad t \geq 0,
\end{equation}
is measurable on $[0,\infty)$. However, the measurability of such a mapping may fail if the canonical norm on $\otimes_{n\in \NN} H_n$ is changed in \eqref{nonMeasMappingEq} by another equivalent Hilbertian norm, as we show in the remark below. The generality of our standing assumptions, in particular the absence of measurability assumptions on the factors, is partly motivated by this phenomenon.

\begin{example}\label{ex:strong-comp-counter-2}
	Let $H$, $\mathcal E$ and $(\boldsymbol e_t)_{t\in[0,2\pi)}$ be as in Example~\ref{ex:strong-comp-counter-1}, and take any non-measurable function $\alpha:[0,2\pi) \to [1,2]$. We define $\|\cdot\|_{\rm{eq}}$ to be the only Hilbertian norm on $H$ satisfying that:
	\begin{itemize}
		\item[(i)] all the incomplete infinite tensor products of $H$ are mutually $\|\cdot\|_{\rm{eq}}$-orthogonal;
		\item[(ii)] if $h \in H$ belongs to an incomplete tensor product $\otimes_{n \in \NN}^{\mathcal A} \CC$, where $\mathcal A$ is an equivalence class of $C_0$-sequences, then
		\begin{align*}
			\|h\|_{\rm{eq}} = \begin{cases}
				\alpha(t) \|h\|, \qquad & \mbox{if } \boldsymbol{e}_t \in \mathcal A \mbox{ for some (hence unique) } t \in [0,2\pi),
				\\ \|h\|, \qquad & \mbox{otherwise}. 
			\end{cases}
		\end{align*}
	\end{itemize}
	Since $\alpha$ and $1/\alpha$ are bounded, it follows that $\|\cdot\|_{\rm{eq}}$ is an equivalent Hilbertian norm on $\otimes_{n\in \NN} \CC$. Moreover, one has
	\begin{align*}
		\|E(t) \left(\otimes_{n\in\NN} 1\right) \|_{\rm{eq}} = \|\otimes_{n\in \NN} e^{it} \|_{\rm{eq}} = \alpha(t), \qquad t\in [0,2\pi), 
	\end{align*}
	which is not a measurable function. 
\end{example}

\section{Similarity to contraction semigroups on infinite tensor products}\label{sec:similarity-tensor}

At this point the tensor-product framework developed in the previous sections is in place, and we
can return to the main similarity problem. We first study semigroups acting on incomplete
infinite tensor products of Hilbert spaces, where the theory is both more flexible and more
natural from the semigroup viewpoint. We then turn to the complete infinite tensor product,
whose role is complementary: it shows how much of the splitting mechanism survives in the
more rigid complete setting and where that rigidity becomes decisive.

\subsection{Similarity on incomplete infinite tensor products}\label{IncompleteSubSect}

We begin with the incomplete setting, which is the natural environment for
the semigroup-theoretic part of the paper. The main result of this subsection
is Theorem~\ref{thm:main_incomplete}. Its proof has two distinct converse
mechanisms. In the case of negative growth bound, we rebalance the factors
and then combine uniform control of the complementary tensor products with
the one-semigroup similarity theorem from Section~\ref{sec:sim-reg-th}. In the zero-growth
case, we extract contractive norms on the individual factors by means of
auxiliary norms and a Banach-limit argument. The preliminary lemmas below
isolate these mechanisms and the numerical input needed for the proof. Proposition~\ref{prop:locboundednes} shows that the local boundedness assumption in Theorem~\ref{thm:main_incomplete} is essential.


For a semigroup $\mathcal T=(T(t))_{t\ge0}$ locally bounded on $(0,\infty)$, we define its exponential type by
\[
 \omega(\mathcal T)
 :=
 \inf_{t>0}\frac{1}{t}\log\|T(t)\|
 \in[-\infty,\infty),
\]
where $\log 0:=-\infty$. By the subadditivity of $t\mapsto \log\|T(t)\|$, one has
\[
 \omega(\mathcal T)=\lim_{t\to\infty}\frac{1}{t}\log\|T(t)\|.
\]
Moreover,
\[
 r(T(t))=e^{t\, \omega(\mathcal T)}, \qquad t >0,
\]
with the convention $e^{-\infty}=0$.

We need the following elementary fact on unconditionally convergent series.
\begin{lemma}\label{strictlynegativeLemma}
 Let $(b_n)_{n\in \NN} \subset[-\infty,\infty)$ be such that
 \[
  \sum_{n\in\NN} (b_n)_+<\infty
  \qquad\text{and}\qquad
  \sum_{n\in\NN} b_n\in[-\infty,0),
 \]
 where the latter sum is understood in the extended unconditional sense (as defined in Section~\ref{sec:num-prod}). Then there exists $(a_n)_{n\in\NN} \subset \RR$ such that
 \[
  \sum_{n\in \NN} |a_n| < \infty,\qquad \sum_{n\in \NN} a_n =0,
 \]
 and $a_n + b_n < 0$ for all $n \in \NN$.
\end{lemma}
\begin{proof}
 Set $J^+:= \{n \in \NN \, : \, b_n \geq 0\}$ and $J^-:= \NN \setminus J^+$. Then
 \[
  \sum_{n\in J^+} b_n<\infty
 \]
 and
 \[
  \sum_{n\in J^+} b_n< -\sum_{n\in J^-} b_n,
 \]
 where the latter sum is allowed to be $-\infty$.
 Choose a summable family $(\varepsilon_n)_{n\in J^+}\subset(0,\infty)$ such that
 \[
  \sum_{n\in J^+} (b_n+\varepsilon_n)< -\sum_{n\in J^-} b_n.
 \]
 Then there exists $(a_n)_{n\in J^-}$ such that $0 \leq a_n < -b_n$ for all $n\in J^-$ and
 \[
  \sum_{n\in J^-} a_n = \sum_{n\in J^+} (b_n + \varepsilon_n).
 \]
 For $n\in J^+$, set $a_n = - (b_n + \varepsilon_n)$. Then $a_n + b_n < 0$ for every $n\in \NN$. Moreover,
 \[
 \sum_{n\in\NN}|a_n|
 =
 \sum_{n\in J^-} a_n + \sum_{n\in J^+}(b_n+\varepsilon_n)
 =
 2\sum_{n\in J^+}(b_n+\varepsilon_n)
 <\infty,
 \]
 so the sequence $(a_n)_{n\in \NN}$ is absolutely summable. Finally,
 \[
  \sum_{n\in \NN} a_n =  \sum_{n\in J^-} a_n  -\sum_{n\in J^+} (b_n + \varepsilon_n)= 0,
 \]
 as claimed.
 
\end{proof}

We also record here the following semigroup counterpart of Lemma~\ref{infiniteSpectralRadiusLemma}, which will be used repeatedly in the proof of the main theorem below. This result is crucial for the proof of Theorem~\ref{thm:main_incomplete}.

\begin{lemma}\label{lem:omega_tensor}
 Let $(H_n)_{n\in \NN}$ be a sequence of Hilbert spaces, and for each $n\in\NN$ let
 $\mathcal T_n=(T_n(t))_{t\ge0}\subset\linearOp(H_n)$ be a semigroup locally bounded on $(0,\infty)$.
 Assume first that $\otimes_{n\in\NN} \mathcal T_n\subset \linearOp\bigl(\otimes_{n\in\NN} H_n\bigr)$ is a well-defined semigroup that is locally bounded in $(0,\infty)$. Then
 \[
  \omega\Bigl(\otimes_{n\in\NN} \mathcal T_n\Bigr)
  =
  \sum_{n\in\NN}\omega(\mathcal T_n),
 \]
 where the sum is understood in the extended unconditional sense (as defined in Section~\ref{sec:num-prod}).
 
 In addition, let \(\boldsymbol y=(y_n)_{n\in\NN}\in\times_{n\in\NN}H_n\) be a
 \(C_0\)-sequence and assume that
 $\otimes_{n\in\NN}^{\boldsymbol{y}} \mathcal T_n\subset \linearOp\bigl(\otimes_{n\in\NN}^{\boldsymbol{y}} H_n\bigr)$ is a well-defined nonzero semigroup that is locally bounded in $(0,\infty)$. Then
 \[
  \omega\Bigl(\otimes_{n\in\NN}^{\boldsymbol{y}} \mathcal T_n\Bigr)
  =
  \sum_{n\in\NN}\omega(\mathcal T_n),
 \]
 again in the extended unconditional sense (as defined in Section~\ref{sec:num-prod}).
\end{lemma}
\begin{proof}
 Apply Lemma~\ref{infiniteSpectralRadiusLemma} at time $1$. In the complete case,
 \[
  r\Bigl(\otimes_{n\in\NN} T_n(1)\Bigr)=\prod_{n\in\NN} r(T_n(1)),
 \]
 and in the incomplete case,
 \[
  r\Bigl(\otimes_{n\in\NN}^{\boldsymbol{y}} T_n(1)\Bigr)=\prod_{n\in\NN} r(T_n(1)).
 \]
 Since $r(T_n(1))=e^{\omega(\mathcal T_n)}$ for every $n\in\NN$, and likewise
 $r(S(1))=e^{\omega(\mathcal S)}$ for every semigroup $\mathcal S$ locally bounded on $(0,\infty)$,
 the result follows by taking logarithms, with the convention $\log 0=-\infty$.
\end{proof}

We are ready to state the main result of this section. 

\begin{theorem}\label{thm:main_incomplete}
 Let $(H_n)_{n\in\NN}$ be a family of Hilbert spaces, let \(\boldsymbol y=(y_n)_{n\in\NN}\in\times_{n\in\NN}H_n\) be a
 \(C_0\)-sequence, and for each $n\in\NN$ let $\mathcal T_n=(T_n(t))_{t\ge0}\subset\linearOp(H_n)$ be a semigroup locally bounded on $(0,\infty)$. Assume that $\motimes_{n\in\NN}^{\boldsymbol y}\mathcal T_n\subset \linearOp\bigl(\motimes_{n\in\NN}^{\boldsymbol y}H_n\bigr)$ is a nonzero semigroup. Then the following assertions hold.
 \begin{enumerate}
  \item[(i)] If there exists a sequence $(d_n)_{n\in\NN}\subset\RR$ such that
  \begin{equation}\label{eq:dn}
   \sum_{n\in\NN}|d_n|<\infty,
   \qquad
   \sum_{n\in\NN} d_n=0,
  \end{equation}
  $e_{d_n}\mathcal T_n\in\mathcal{SC}(H_n)$ for all $n\in\NN$, and
  \[
   \prod_{n\in\NN}\mathcal C(e_{d_n}\mathcal T_n)<\infty,
  \]
  then
  \[
   \motimes_{n\in\NN}^{\boldsymbol y}\mathcal T_n
   \in
   \mathcal{SC}\Bigl(\motimes_{n\in\NN}^{\boldsymbol y}H_n\Bigr).
  \]
  
  \item[(ii)] Conversely, if
  \[
   \motimes_{n\in\NN}^{\boldsymbol y}\mathcal T_n
   \in
   \mathcal{SC}\Bigl(\motimes_{n\in\NN}^{\boldsymbol y}H_n\Bigr),
  \]
  then there exists a sequence $(d_n)_{n\in\NN}\subset\RR$ satisfying \eqref{eq:dn} such that
  \[
   e_{d_n}\mathcal T_n\in\mathcal{SC}(H_n)\qquad\text{for all } n\in\NN,
  \]
  and
  \[
   \sup_{n\in\NN}\mathcal C(e_{d_n}\mathcal T_n)<\infty.
  \]
 \end{enumerate}
\end{theorem}
\begin{proof}
 To prove (i), set $C_n:=\mathcal C(e_{d_n}\mathcal T_n)$ 
  and for each $n\in\NN$, let $\|\cdot\|_n^{\sim}$ be an equivalent Hilbertian norm on $H_n$ such that
  \[
   \|h\| \le \|h\|_n^{\sim} \le C_n \|h\|,\qquad h\in H_n,
  \]
  and
  \[
   \|e^{d_n t}T_n(t)h\|_n^{\sim} \le \|h\|_n^{\sim},\qquad h\in H_n,\ t\ge0.
  \]
  Using Proposition~\ref{prop:equivalent-infinite-tensor}(ii), we obtain that $\boldsymbol y$ is also a $C_0$-sequence with respect to the norms $\|\cdot\|_n^{\sim}$, that the corresponding incomplete tensor product norm on $\otimes_{n\in\NN}^{\boldsymbol y}H_n$ is then an equivalent Hilbertian norm, with distortion controlled by $\prod_n C_n$, and by construction the semigroup
  \(
   \motimes_{n\in\NN}^{\boldsymbol y} e_{d_n}\mathcal T_n
  \)
  is contractive for that tensorized norm. Since $\sum_{n\in\NN} d_n=0$, we have
  \[
   \motimes_{n\in\NN}^{\boldsymbol y} e^{d_n t}T_n(t)
   =
   \motimes_{n\in\NN}^{\boldsymbol y} T_n(t),\qquad t\ge0.
  \]
  Hence $\otimes_{n\in \NN}^{\boldsymbol{y}} \mathcal T_n \in \mathcal{SC}\left(\otimes_{n\in \NN}^{\boldsymbol{y}} H_n\right)$.

 We now prove (ii). 
 The proof is rather long, so we divide it into several steps.

 \smallskip
 \noindent\textbf{Step 1. Reduction to negative growth.} First assume that $\omega \left(\otimes_{n\in\NN}^{\boldsymbol{y}} \mathcal T_n\right) < 0$ and choose $d \in \RR$ such that
  \[
   \omega \Big(\motimes_{n\in\NN}^{\boldsymbol{y}}  \mathcal T_n\Big) < d < 0.
  \]
  Recall that $\omega \left(\otimes_{n\in\NN}^{\boldsymbol{y}} \mathcal T_n\right) = \sum_{n\in \NN} \omega(\mathcal T_n)$ in the extended unconditional sense (as defined in Section~\ref{sec:num-prod}), see Lemma~\ref{lem:omega_tensor}. Hence we may choose $N \in \NN$ such that
  \[
   \sum_{n=1}^N \omega(\mathcal T_n) < d.
  \]
  Using Lemma~\ref{lem:nonZeroInfiniteSemigroup}, choose $\delta>0$ such that
  \[
   \otimes_{n\in\NN}^{\boldsymbol{y}} T_n(\delta) \neq 0
   \qquad \text{and} \qquad
   \prod_{n\in\NN} \maxone{\|T_n(\delta)\|}<\infty.
  \]
  Since the tails of a convergent nonzero product converge to $1$, enlarging $N$ if necessary we may moreover assume that
  \[
   \prod_{n>N} \maxone{\|T_n(\delta)\|}< e^{-d\delta}.
  \]
  Then, up to replacing $\mathcal T_N$ and $\mathcal T_{N+1}$ by $e_{-d} \mathcal T_N$ and $e_d \mathcal T_{N+1}$ respectively, we may assume $\sum_{n=1}^N \omega\left( \mathcal T_n\right) < 0$ and $\prod_{n>N} \|T_n(\delta)\| < 1$.
  Thus, if we set
  \[
  b_n=
  \begin{cases}
  	\omega(T_n),& n=1,\ldots,N,\\[2mm]
  	\delta^{-1}\log\|T_n(\delta)\|,& n>N,
  \end{cases}
  \]
  then
  \(
  \sum_{n\in\mathbb N}b_n<0
  \)
  in the extended unconditional sense.  By Lemma~\ref{strictlynegativeLemma}, there exists
  \((a_n)_{n\in\mathbb N}\in\ell^1\) with
  \[
  \sum_{n\in\mathbb N}a_n=0,
  \qquad
  a_n+b_n<0\quad(n\in\mathbb N).
  \]
  Replacing \(\mathcal T_n\) by \(e_{a_n} \mathcal T_n\), we may therefore suppose that
  \[
  \omega(\mathcal T_n)<0\quad(n=1,\ldots,N),
  \qquad
  \|T_n(\delta)\|<1\quad(n>N).
  \]
  
  \smallskip
  \noindent\textbf{Step 2. Uniform control of complementary products.}
  For each \(n\in\NN\), set
  \(\boldsymbol y^n=(y_\ell)_{\ell\ne n}\in\times_{\ell\ne n}H_\ell\), which is a
  \(C_0\)-sequence. Using Lemma~\ref{lem:nonZeroInfiniteSemigroup}, it follows that there exists $\tau>0$ such that, for all $t \in [0,\tau]$,
  $$T_\ell(t) \neq 0 \quad (\ell\in \NN), \qquad \prod_{\ell\in \NN} \maxone{\|T_\ell(t)\|} < \infty, \qquad \sum_{\ell \in \NN} |1 - \langle T_\ell (t) y_\ell, y_\ell \rangle | < \infty.
  $$ 
  Hence, applying Lemma~\ref{lem:nonZeroInfiniteSemigroup} to the family $(\mathcal T_\ell)_{\ell \neq n}$, where $n\in \NN$, we obtain that
  \begin{equation}\label{eq:ellSemigroup}
  	\motimes_{\ell \neq n}^{\boldsymbol{y}^n} \mathcal T_\ell \subset \linearOp \Big(\motimes_{\ell \neq n}^{\boldsymbol{y}^n} H_\ell\Big)
  \end{equation}
  is a well-defined nonzero semigroup.
  
  Now define
  $$f_n(t) := \|T_n(t)\|, \qquad t > 0, \, n\in \NN,
  $$
  and note that, since $\otimes_{n\in \NN}^{\boldsymbol{y}} \mathcal T_n$ is similar to a contraction semigroup, Lemma~\ref{lem:nonZeroInfiniteSemigroup} yields that
  $$\limsup_{t\to 0^+} \prod_{n\in \NN} f_n(t) = \limsup_{t\to0^+} \|\otimes_{n\in \NN}^{\boldsymbol{y}} T_n(t)\| < \infty.
  $$
  Then, using Lemma~\ref{lem:unif_bound}(ii) and the identity
  $$\prod_{\ell \neq n} f_\ell(t) = \|\otimes_{\ell \neq n}^{\boldsymbol{y}^n} T_\ell(t)\|, \qquad n\in \NN, \, t>0,
  $$
  let $M > 0$ be such that, if 
  \begin{equation*}
   A_n :=  \left\{ t \in (0,\delta) \, : \, \left\|\otimes_{\ell \neq n}^{\boldsymbol{y}^n}  T_\ell(t) \right\| \leq M\right\}, \qquad n \in \NN,
  \end{equation*}
  then $\nu_\delta(A_n) \geq 1/2$ for all $n\in \NN$. Note also that
  $$\inf_{n\in \NN} \left\| \otimes_{\ell \neq n}^{\boldsymbol{y}^n} T_\ell(\delta) \right\| =\frac{\left\| \otimes_{\ell \in \NN}^{\boldsymbol{y}} T_\ell(\delta) \right\|}{\sup_{n\in \NN} \|T_n(\delta)\|}  > 0.
  $$
  Using these facts, for each $n\in \NN$ choose $h _n \in \otimes_{\ell \neq n}^{\boldsymbol{y}^n} H_\ell$ with $\|h_n\|=1$ and
  \[
   \left\| \left( \otimes_{\ell \neq n}^{\boldsymbol{y}^n} T_\ell(\delta) \right) h_n \right\|
   \ge \frac12 \left\| \otimes_{\ell \neq n}^{\boldsymbol{y}^n} T_\ell(\delta) \right\|.
  \]
  Then
  \[
   \inf_{n\in \NN} \left\| \left( \otimes_{\ell \neq n}^{\boldsymbol{y}^n} T_\ell(\delta) \right) h_n \right\|
   \ge \frac12 \inf_{n\in \NN}\left\| \otimes_{\ell \neq n}^{\boldsymbol{y}^n} T_\ell(\delta) \right\|>0.
  \]
  Set
  $$D_n := \mu_\delta \left[ \left\| \left( \otimes_{\ell \neq n}^{\boldsymbol{y}^n} T_\ell(\cdot) \right) h_n\right\|^2 \right], \qquad n \in \NN.
  $$
  Then, for all $n\in \NN$ and all $t \in \delta - A_n$, we have that
  \begin{align*}
   \left\| \left( \otimes_{\ell \neq n}^{\boldsymbol{y}^n} T_\ell(\delta) \right) h_n\right\|
   &\leq \left\| \otimes_{\ell \neq n}^{\boldsymbol{y}^n} T_\ell(\delta - t) \right\|
        \left\| \left( \otimes_{\ell \neq n}^{\boldsymbol{y}^n} T_\ell(t) \right) h_n\right\| 
   \leq M \left\| \left( \otimes_{\ell \neq n}^{\boldsymbol{y}^n} T_\ell(t) \right) h_n\right\|.
  \end{align*}
  where we used that $\delta - t \in A_n$ if $t \in \delta - A_n$. As $\nu_\delta(\delta - A_n) = \nu_\delta (A_n) \geq 1/2$, it follows that
  \begin{equation*}
  \begin{aligned}
   D:= \inf_{n\in \NN} D_n & \geq \inf_{n \in \NN}  \left( \mu_\delta \left[\left\| \left( \otimes_{\ell \neq n}^{\boldsymbol{y}^n} T_\ell(\cdot) \right) h_n\right\|^2  \chi_{\delta - A_n}(\cdot) \right] \right)
   \geq  \frac{\inf_{n\in \NN} \left\| \left( \otimes_{\ell \neq n}^{\boldsymbol{y}^n} T_\ell(\delta) \right) h_n \right\|^2}{2M^2} > 0.
  \end{aligned}
  \end{equation*}


\smallskip
\textbf{Step 3: Step 3. Construction of factorwise contractive norms.}
Now, let $\|\cdot\|_{\rm{eq}}$ be an equivalent Hilbertian norm on $\otimes_{\ell \in \NN}^{\boldsymbol{y}} H_\ell$ making $\otimes_{\ell \in \NN}^{\boldsymbol{y}} \mathcal T_\ell$ a contractive semigroup and satisfying $\| \cdot\| \leq \|\cdot \|_{\rm{eq}} \leq \mathcal C \|\cdot\|$, where $\mathcal C := \mathcal C\left(\otimes_{\ell \in \NN}^{\boldsymbol{y}} \mathcal T_\ell\right)$. 
  
  Recall that by the associative law of incomplete tensor products, one has $ \otimes_{\ell \in \NN}^{\boldsymbol{y}} H_\ell \simeq H_n  \otimes \left( \otimes_{\ell \neq n}^{\boldsymbol{y}^n} H_\ell\right)$ unitarily for every $n\in \NN$, see~\eqref{eq:assLawIncomplete}. Then, for each $n \in \NN$, define
  \begin{align*}
   \|f\|_n^2 := \frac{1}{D_n}\mu_\delta \left[ \left\|f \otimes \left( \otimes_{\ell \neq n}^{\boldsymbol{y}^n} T_\ell(\cdot) \right) h_n \right\|_{\rm{eq}}^2 \right], \qquad f \in H_n.
  \end{align*}
  We obtain that
  \begin{align*}
   \|f\|_n^2 & \leq \frac{\mathcal C^2}{D_n} \mu_\delta \left[ \left\|f \otimes \left( \otimes_{\ell \neq n}^{\boldsymbol{y}^n} T_\ell(\cdot) \right) h_n \right\|^2 \right] \leq \mathcal C^2  \|f\|^2, \qquad f \in H_n, \, n \in \NN,
  \end{align*}
  and similarly we get that
  \begin{align*}
   \|f\|_n^2 & \geq \|f\|^2, \qquad f \in H_n, \, n \in \NN.
  \end{align*}
  Consequently, $\|\cdot\|_n$ is an equivalent Hilbertian norm on $H_n$ for every $n\in \NN$. Moreover, for all $s \in (0,\delta)$, $n \in \NN$ and $f \in H_n$, we have that
  \begin{equation}\label{eq:quasi_1}
   \begin{aligned}
   &\mu_\delta \left[ \left\| T_n(s) f \otimes \left(\otimes_{\ell \neq n}^{\boldsymbol{y}^n} T_\ell(\cdot)\right) h_n \right\|_{\rm{eq}}^2 \chi_{[s,\delta)} (\cdot) \right]
   \\= & \mu_\delta \left[ \left\| T_n(s) f \otimes \left(\otimes_{\ell \neq n}^{\boldsymbol{y}^n} T_\ell(s+\cdot)\right) h_n \right\|_{\rm{eq}}^2 \chi_{[0,\delta-s)} (\cdot) \right]
   \\ \leq & \mu_\delta \left[ \left\| f \otimes \left(\otimes_{\ell \neq n}^{\boldsymbol{y}^n} T_\ell(\cdot)\right) h_n \right\|_{\rm{eq}}^2 \right]
   = D_n \|f\|_n^2,
   \end{aligned}
  \end{equation}
  and
  \begin{equation}\label{eq:quasi_2}
  \begin{aligned}
   &\mu_\delta \left[ \left\| T_n(s) f \otimes \left(\otimes_{\ell \neq n}^{\boldsymbol{y}^n} T_\ell(\cdot)\right) h_n \right\|_{\rm{eq}}^2 \chi_{[0,s)} (\cdot) \right]
   \\\leq & \mu_\delta \left[ \left\| T_n(s - \cdot) f \otimes  h_n \right\|_{\rm{eq}}^2 \chi_{[0,s)} (\cdot) \right]
    \leq  \mathcal C^2 K_{n,s}^2  \frac{s}{\delta} \|f\|^2
    \leq  \mathcal C^2 K_{n,s}^2   \frac{s}{\delta} \|f\|_n^2,
  \end{aligned}
  \end{equation}
  where 
  \begin{equation}\label{eq:Kns_bound}
   K_{n,s}:= \sup_{t \in [0,s]} \|T_n(t)\| = \sup_{t\in [0,s]} f_n(t) < \infty,
  \end{equation}
  see Lemma~\ref{lem:unif_bound}(i). Since 
  $$\|T_n(s) f\|_n^2 := \frac{1}{D_n} \mu_\delta \left[ \left\| T_n(s) f \otimes \left(\otimes_{\ell \neq n}^{\boldsymbol{y}^n} T_\ell(\cdot)\right) h_n \right\|_{\rm{eq}}^2 \right],
  $$
  we obtain from \eqref{eq:quasi_1} and \eqref{eq:quasi_2} that
  $$\|T_n(s)\|_n \leq \sqrt{1 + \frac{\mathcal C^2 K_{n,s}^2}{\delta D_n} s}  \leq \sqrt{1 + \frac{\mathcal C^2 K_{n,s}^2}{\delta D} s}, \qquad s \in [0,\delta], \, n \in \NN.
  $$
  Using Lemma~\ref{lem:unif_bound}(i), we obtain that
  $$\limsup_{s\to 0^+} K_{n,s} \leq K:= \limsup_{s \to 0^+} \left\| \otimes_{\ell \in \NN}^{\boldsymbol{y}} T_\ell(s) \right\| < \infty.
  $$
  Hence
  \[
   \limsup_{s\to 0^+}\frac{\|T_n(s)\|_n-1}{s}
   \le \frac{\mathcal C^2 K^2}{2\delta D},\qquad n\in\NN,
  \]
  and thus, using the submultiplicativity of each of the mappings $s \mapsto \|T_n(s)\|_n$,
  \begin{equation}\label{eq:boundTnC}
   \|T_n(s)\|_n \leq \exp\left(\frac{\mathcal C^2 K^2}{2\delta D} \, s\right), \qquad s \geq 0, \, n \in \NN.
  \end{equation}
  
  Consequently, $\mathcal T_n \in \mathcal{SQC}(H_n)$ and 
  \begin{equation*}
   \mathcal C \left(e_{- \frac{\mathcal C^2 K^2}{2\delta D}}\mathcal T_n \right) \leq \mathcal C, \qquad n \in \NN.
  \end{equation*}

  \smallskip
  \noindent\textbf{Step 4. Application of the low-regularity similarity theorem.}
  We are now exactly in the situation of the one-semigroup similarity theorem from Section~\ref{sec:sim-reg-th}: the norms $\|\cdot\|_n$ provide uniform quasi-contractive control, while the estimate $\|T_n(\delta)\|<1$ for $n>N$ supplies the one-time contractive input.
  
  Since $\|h\|\leq \|h\|_n \leq \mathcal C\|h\|$, it follows from~\eqref{eq:boundTnC} that
  \[
   \|T_n(t)\| \leq \mathcal C\exp\left(\frac{\mathcal C^2 K^2}{2\delta D}t\right), \qquad t\in[0,\maxone{\delta}],\ n>N,
  \]
  so the local bound required in Proposition~\ref{Prop:simConstant} is uniform in $n$. 
  Since \(\|T_n(\delta)\|<1\), the operator \(T_n(\delta)\) is already
  contractive in the original Hilbert norm.  Thus the one-time contractive
  input in Theorem~\ref{Th:eqNorm} may be taken with \(C_2=1\), which yields that
  $\mathcal T_n \in \mathcal{SC}(H_n)$ for all $n > N$ and
  \begin{equation*}\label{eq:Cn_bound}
   \sup_{n>N} \mathcal C(\mathcal T_n)< \infty.
  \end{equation*}
  For the finitely many remaining factors \(n=1,\ldots,N\), the inequality
  \(\omega(\mathcal T_n)<0\), together with the local boundedness near zero obtained
  above, gives a quasi-contractive equivalent norm. Choosing \(\tau>0\) so
  large that \(\|T_n(\tau)\|<1\), the operator \(T_n(\tau)\) is similar to a
  contraction.  Hence Theorem~\ref{Th:eqNorm} applies and yields \(\mathcal T_n\in\mathcal{SC}(H_n)\).

  Returning to the original family, the required balancing sequence is obtained
  by adding to \((a_n)_{n\in\NN}\) the finitely supported balancing used above to replace
  \(\mathcal T_N\) and \(\mathcal T_{N+1}\).  This does not affect absolute summability, and the
  sum of the resulting sequence is still zero.
  
  Hence, we have proven assertion (ii) when $\omega \left(\otimes_{n\in\NN}^{\boldsymbol{y}} \mathcal T_n\right) < 0$. \medskip

  \smallskip
  \noindent\textbf{Step 5. The zero-growth case.} If $\omega \left(\otimes_{n\in\NN}^{\boldsymbol{y}} \mathcal T_n\right) = 0$, the aim in this case is to extract, from a single contractive Hilbertian norm on the whole tensor product, a uniformly controlled family of equivalent Hilbertian norms on the individual factors. First, set 
  $$d_n:=-\omega(\mathcal T_n), \qquad n \in \NN.
  $$
  The preceding lemma gives 
  \[
  \sum_{n\in \NN} d_n = - \omega \left(\otimes_{n\in \NN}^{\boldsymbol{y}} \mathcal T_n\right) = 0
  \]
  in the extended unconditional sense. Hence, the positive and negative parts of the series are both finite, and therefore \((d_n)_{n\in\mathbb N}\in\ell^1\). 
  Then, replacing $\mathcal T_n$ by $e_{d_n} \mathcal T_n$, we may assume that 
  $$\omega(\mathcal T_n)=0, \qquad n\in\NN.
  $$

  
  For each \(n\in\NN\), set \(\boldsymbol y^n:=(y_\ell)_{\ell\ne n}\in\times_{\ell\ne n}H_\ell\), which is a
  \(C_0\)-sequence, and then, arguing as in~\eqref{eq:ellSemigroup}, we obtain that
  $$\motimes_{\ell\neq n}^{\boldsymbol{y}^n} \mathcal T_\ell \subset \linearOp \Big(\motimes_{\ell \neq n}^{\boldsymbol{y}^n} H_\ell\Big)
  $$
  is a well-defined nonzero semigroup. 
  
  Set $K :=  \sup_{t\geq0} \left\| \otimes_{\ell \in \NN}^{\boldsymbol{y}} T_\ell(t)\right\|$ and note that $K< \infty$ since $\otimes_{\ell \in \NN}^{\boldsymbol{y}} \mathcal T_\ell$ is similar to a contractive semigroup. Since $\|T_\ell(t)\|\geq 1$ for all $l\in \NN$ and $t\geq 0$ (recall that $\omega(\mathcal T_\ell)=0$), we have that
  \begin{equation}\label{eq:Kbound}
   1\leq \prod_{\ell \neq n} \left\|  T_\ell(t)\right\| \leq \prod_{\ell \in \NN} \left\|  T_\ell(t)\right\| = \left\| \otimes_{\ell \in \NN}^{\boldsymbol y} T_\ell(t) \right\| \leq K, \qquad t \geq 0, \, n \in \NN.
  \end{equation}
  As $\left\| \otimes_{\ell \neq n}^{\boldsymbol{y}^n} T_\ell(t) \right\| = \prod_{\ell \neq n} \|T_\ell(t)\|$, we deduce that 
  $$1 \leq \left\| \otimes_{\ell \neq n}^{\boldsymbol{y}^n} T_\ell(t) \right\| \leq K, \qquad n\in \NN, \, t \geq 0.
  $$
  Therefore, for each $n,m\in \NN$ we may choose a unit vector $h_{n,m}\in \otimes_{\ell \neq n}^{\boldsymbol{y}^n} H_\ell$ such that
  \[
   \Big\| \Big(\otimes_{\ell \neq n}^{\boldsymbol{y}^n} T_\ell(m) \Big) h_{n,m} \Big\|\geq \frac12.
  \]
  For $t\in[0,m]$ we then have
  \[
   \frac12 \le \Big\| \Big(\otimes_{\ell \neq n}^{\boldsymbol{y}^n} T_\ell(m) \Big) h_{n,m} \Big\| = \Big\|\Big(\otimes_{\ell \neq n}^{\boldsymbol{y}^n} T_\ell(m-t) \Big) \Big(\otimes_{\ell \neq n}^{\boldsymbol{y}^n} T_\ell(t) \Big) h_{n,m} \Big\| \le K\, \Big\| \Big(\otimes_{\ell \neq n}^{\boldsymbol{y}^n} T_\ell(t) \Big) h_{n,m} \Big\|,
  \]
  and hence
  \begin{equation*}
   \left\| \left(\otimes_{\ell \neq n}^{\boldsymbol{y}^n} T_\ell(t) \right) h_{n,m} \right\| 
   \geq \frac{1}{2K}, \qquad t \in [0,m].
  \end{equation*}
  For each $n,m\in \NN$, set also
  \begin{align*}
   D_{n,m} &:= \mu_m \left[ \left\| \left(\otimes_{\ell \neq n}^{\boldsymbol{y}^n} T_\ell(\cdot) \right) h_{n,m}  \right\|^2 \right],
  \end{align*}
  where we consider the canonical tensor product norm in $\otimes_{\ell \neq n}^{\boldsymbol{y}^n} H_\ell$. 
  Then, for all $n,m \in \NN$, we have that $D_{n,m} \geq 1/(4K^2)$ by our choice of $h_{n,m}$, and
  \begin{align*}
   D_{n,m}  \leq \sup_{t \in [0,m]} \left\| \otimes_{\ell \neq n}^{\boldsymbol{y}^n} T_\ell(t) \right\|^2 \leq K^2, 
  \end{align*}
  see \eqref{eq:Kbound}. Now let $\|\cdot\|_{\rm{eq}}$ be a Hilbertian norm on $\otimes_{\ell \in \NN}^{\boldsymbol{y}}H_\ell$ such that $\|\otimes_{\ell \in \NN}^{\boldsymbol{y}} T_\ell(t)\|_{\rm{eq}}\leq 1$ for all $t\geq0$ and $\|\cdot\| \leq \|\cdot\|_{\rm{eq}} \leq \mathcal C \|\cdot\|$, where $\mathcal C := \mathcal C \left( \otimes_{\ell \in \NN}^{\boldsymbol{y}} \mathcal T_\ell\right)$, and define
  \begin{equation*}
   \|f\|_{n,m}^2 := \frac{1}{D_{n,m}}\mu_m \left[ \left\| f \otimes \left(\otimes_{\ell \neq n}^{\boldsymbol{y}^n} T_\ell(\cdot) \right) h_{n,m} \right\|_{\rm{eq}}^2 \right], \qquad f \in H_n,
  \end{equation*}
  where we used the unitary equivalence given by the associative law for incomplete tensor products
  $$\otimes_{\ell \in \NN}^{\boldsymbol{y}} H_\ell \simeq H_n \otimes \left(\otimes_{\ell \neq n}^{\boldsymbol{y}^n} H_\ell\right),
  $$ 
  see~\eqref{eq:assLawIncomplete}.

  Therefore, one has
  \begin{align*}
   \|f\|_{n,m}^2
   &\leq \frac{\mathcal C^2}{D_{n,m}}
      \mu_m \left[ \left\| f \otimes \left(\otimes_{\ell \neq n}^{\boldsymbol{y}^n} T_\ell(\cdot) \right) h_{n,m} \right\|^2 \right] 
   \leq \mathcal C^2 \|f\|^2,
   \qquad n,m\in \NN,\ f \in H_n,
  \end{align*}
  and similarly, $\|f\|_{n,m}  \geq  \|f\|$ for all $n,m \in \NN$ and $f\in H_n$.
  
  Let \(\operatorname{LIM}\) be a shift-invariant mean, also called a Banach limit,
  on \(\ell^\infty(\mathbb \NN)\) and define
  $$\|f\|_n^2 := \operatorname{LIM}_{m\to\infty} \|f\|_{n,m}^2, \qquad f \in H_n, \, n \in \NN.
  $$
  Let $\langle\cdot,\cdot\rangle_{n,m}$ denote the inner product inducing $\|\cdot\|_{n,m}$, and define
  $$\langle f,g\rangle_n := \operatorname{LIM}_{m\to\infty} \langle f,g\rangle_{n,m}, \qquad f,g\in H_n.$$
  Then $\langle\cdot,\cdot\rangle_n$ is a positive semidefinite sesquilinear form on $H_n$ and
  $\|f\|_n^2=\langle f,f\rangle_n$ for all $f\in H_n$.
  Since the two-sided estimates above imply $\|f\| \le \|f\|_n \le \mathcal C\|f\|$, the form is definite,
  and hence, for each $n\in \NN$, $\|\cdot\|_n$ is an equivalent Hilbertian norm on $H_n$ satisfying
  \begin{equation}\label{eq:sim_Const_Hn}
   \| f \| \leq \|f\|_n \leq  \mathcal C \| f \|, \qquad f \in H_n,\, n \in \NN.
  \end{equation}
  Moreover, for every $n,m\in \NN$, $s\in [0,1]$ and $f\in H_n$, one has
  \begin{align*}
   \|T_n(s)f\|_{n,m}^2 &= \frac{1}{D_{n,m}} \mu_m \left[  \left\| T_n(s) f \otimes \left(\otimes_{\ell \neq n}^{\boldsymbol{y}^n} T_\ell(\cdot)\right) h_{n,m} \right\|_{\rm{eq}}^2 \right].
  \end{align*}
  Note that, for all $n,m \in \NN$ and $f \in H_n$,
  \begin{equation}\label{eq:mean_bound_1}
   \begin{aligned}
   &\mu_m \left[ \left\| T_n(s) f \otimes \left(\otimes_{\ell \neq n}^{\boldsymbol{y}^n} T_\ell(\cdot) \right) h_{n,m} \right\|_{\rm{eq}}^2 \chi_{[s,m)}(\cdot) \right]
   \\ \leq & \mu_m \left[ \left\| f \otimes \left(\otimes_{\ell \neq n}^{\boldsymbol{y}^n} T_\ell(\cdot) \right) h_{n,m} \right\|_{\rm{eq}}^2 \chi_{[0,m-s)}(\cdot) \right]
   \leq  D_{n,m} \|f\|_{n,m}^2,
   \end{aligned}
  \end{equation}
  where we used $\|\otimes_{\ell \in \NN}^{\boldsymbol{y}} T(s)\|_{\rm{eq}} \leq 1$ and the identities from Section~\ref{subsec:periodic-means}, and
  \begin{equation}\label{eq:mean_bound_2}
  \begin{aligned}
   &\mu_m \left[ \left\| T_n(s) f \otimes \left(\otimes_{\ell \neq n}^{\boldsymbol{y}^n} T_\ell(\cdot) \right) h_{n,m} \right\|_{\rm{eq}}^2 \chi_{[0,s)}(\cdot) \right] 
   \\ \leq & \mathcal C^2 \mu_m \left[ \left\| T_n(s - \cdot) f \otimes  h_{n,m} \right\|^2 \chi_{[0,s)}(\cdot) \right]
   \\ \leq & \mathcal C^2 \sup_{t\geq 0} \left(\|T_n(t)\|^2 \right) \frac{s}{m} \|f\|^2
   \leq  \mathcal C^2  K^2 \frac{s}{m} \|f\|_{n,m}^2,
  \end{aligned}
  \end{equation}
  where we used that 
  $$\sup_{n\in \NN, \, t\geq 0} \|T_n(t)\| \leq \sup_{t\geq0} \prod_{\ell \in \NN} \left\|T_\ell(t)\right\| = K,
  $$ 
  since $\|T_\ell(t)\| \geq 1$ for all $l\in \NN$ and $t\geq0$ as $\omega(\mathcal T_\ell) = 0$.
  Then \eqref{eq:mean_bound_1} and \eqref{eq:mean_bound_2} imply that $\|T_n(s)\|_{n,m}^2 \leq 1 + \frac{\mathcal C^2 K^2 }{D_{n,m}} \frac{s}{m}$ for $n,m\in \NN$, $s\in [0,1]$. Hence, we conclude
  $$\|T_n(s)\|_n^2 \leq  \operatorname{LIM}_m \left[  1 + \frac{\mathcal C^2 K^2 }{D_{n,m}} \frac{s}{m} \right] = 1, \qquad n \in \NN, \, s\in [0,1],
  $$
  where we used that $D_{n,m}\geq 1/(4K^2)$ for all $n,m\in \NN$. Hence $\|T_n(s)\|_n\le 1$ for all $s\in[0,1]$. If $t\ge0$ and $m\in\NN$ is chosen so that $t/m\in[0,1]$, then $T_n(t)=T_n(t/m)^m$, and therefore $\|T_n(t)\|_n\le 1$. Consequently, $\mathcal T_n \in \mathcal{SC}(H_n)$ and, by \eqref{eq:sim_Const_Hn}, 
  $$\mathcal C(\mathcal T_n) \leq \mathcal C,  \qquad n\in \NN,
  $$
  finishing the proof.
\end{proof}

The quasi-contractive analogue follows by a simple reduction to Theorem~\ref{thm:main_incomplete} via exponential rescaling of the factors. 

\begin{theorem}\label{continuousInfiniteCor11}
 Let $(H_n)_{n\in\NN}$ be a family of Hilbert spaces, let \(\boldsymbol y=(y_n)_{n\in\NN}\in\times_{n\in\NN}H_n\) be a \(C_0\)-sequence, and for each $n\in\NN$ let $\mathcal T_n=(T_n(t))_{t\ge0}\subset\linearOp(H_n)$ be a semigroup locally bounded on $(0,\infty)$. Assume that $\motimes_{n\in\NN}^{\boldsymbol y}\mathcal T_n\subset\linearOp\bigl(\motimes_{n\in\NN}^{\boldsymbol y}H_n\bigr)$ is a nonzero semigroup. Then the following assertions hold.
 \begin{enumerate}
  \item[(i)] If there exists a sequence $(d_n)_{n\in\NN}\subset\RR$ such that
  \[
   \sum_{n\in\NN}|d_n|<\infty,
  \]
  $e_{d_n}\mathcal T_n\in\mathcal{SC}(H_n)$ for all $n\in\NN$, and
  \[
   \prod_{n\in\NN}\mathcal C(e_{d_n}\mathcal T_n)<\infty,
  \]
  then
  \[
   \motimes_{n\in\NN}^{\boldsymbol y}\mathcal T_n
   \in
   \mathcal{SQC}\Bigl(\motimes_{n\in\NN}^{\boldsymbol y}H_n\Bigr).
  \]
  
  \item[(ii)] Conversely, if
  \[
   \motimes_{n\in\NN}^{\boldsymbol y}\mathcal T_n
   \in
   \mathcal{SQC}\Bigl(\motimes_{n\in\NN}^{\boldsymbol y}H_n\Bigr),
  \]
  then there exists a sequence $(d_n)_{n\in\NN}\subset\RR$ such that
  \[
   \sum_{n\in\NN}|d_n|<\infty,
  \]
  \[
   e_{d_n}\mathcal T_n\in\mathcal{SC}(H_n)\qquad\text{for all } n\in\NN,
  \]
  and
  \[
   \sup_{n\in\NN}\mathcal C(e_{d_n}\mathcal T_n)<\infty.
  \]
 \end{enumerate}
\end{theorem}
\begin{proof}
 To prove (i), consider the semigroups
  \[
   \mathcal S_n:=e_{d_n}\mathcal T_n,\qquad n\in\NN.
  \]
Since \(\sum_{n\in\NN}|d_n|<\infty\), the scalar product
\[
\prod_{n\in\NN} e^{d_nt}
=
\exp\left(t\sum_{n\in\NN}d_n\right)
\]
is finite and nonzero for every \(t\ge0\).  Hence the rescaled incomplete
tensor product is again well defined and nonzero, and
\[
\motimes_{n\in\NN}^{\boldsymbol y}S_n(t)
=
\exp\left(t\sum_{n\in\NN}d_n\right)
\motimes_{n\in\NN}^{\boldsymbol y}T_n(t),
\qquad t\ge0 .
\]
Moreover, each \(\mathcal S_n\) belongs to \(\mathcal{SC}(H_n)\), the
sequence \((0)_{n\in\NN}\) satisfies the balancing condition~\eqref{eq:dn},
and
  \[
   \prod_{n\in\NN}\mathcal C(\mathcal S_n)<\infty.
  \]
  Hence Theorem~\ref{thm:main_incomplete}(i) gives
  \(
   \motimes_{n\in\NN}^{\boldsymbol y}\mathcal S_n\in
   \mathcal{SC}\!\left(\motimes_{n\in\NN}^{\boldsymbol y}H_n\right).
  \)
It follows that
\[
\motimes_{n\in\NN}^{\boldsymbol y}\mathcal T_n
\in
\mathcal{SQC}\!\left(\motimes_{n\in\NN}^{\boldsymbol y}H_n\right).
\]

 We now prove (ii). Suppose that $\otimes_{n\in\NN}^{\boldsymbol y}\mathcal T_n\in\mathcal{SQC}\!\left(\otimes_{n\in\NN}^{\boldsymbol y}H_n\right)$. Then there exists $c\in\RR$ such that
  \[
   e_{-c}\!\left(\motimes_{n\in\NN}^{\boldsymbol y}\mathcal T_n\right)
   \in
   \mathcal{SC}\!\left(\motimes_{n\in\NN}^{\boldsymbol y}H_n\right).
  \]
  Define
  \[
   \mathcal S_1:=e_{-c}\mathcal T_1,\qquad \mathcal S_n:=\mathcal T_n\ \ (n\ge2).
  \]
  Then
  \[
   \motimes_{n\in\NN}^{\boldsymbol y}\mathcal S_n
   =
   e_{-c}\!\left(\motimes_{n\in\NN}^{\boldsymbol y}\mathcal T_n\right)
   \in
   \mathcal{SC}\!\left(\motimes_{n\in\NN}^{\boldsymbol y}H_n\right).
  \]
  Applying Theorem~\ref{thm:main_incomplete}(ii) to the family $(\mathcal S_n)_{n\in\NN}$, we obtain a sequence $(a_n)_{n\in\NN}\subset\RR$ such that
  \[
   \sum_{n\in\NN}|a_n|<\infty,
  \]
  $e_{a_n}\mathcal S_n\in\mathcal{SC}(H_n)$ for all $n\in\NN$, and
  \(
   \sup_{n\in\NN}\mathcal C(e_{a_n}\mathcal S_n)<\infty.
  \)
  Setting
  \[
   d_1:=a_1-c,\qquad d_n:=a_n\ \ (n\ge2),
  \]
  we obtain
  \[
   \sum_{n\in\NN}|d_n|<\infty,
  \]
  $e_{d_n}\mathcal T_n\in\mathcal{SC}(H_n)$ for all $n\in\NN$, and
  \(
   \sup_{n\in\NN}\mathcal C(e_{d_n}\mathcal T_n)<\infty.
  \)
\end{proof}

We conclude this section with a  counterexample showing that the local boundedness assumption used earlier
is not merely technical. 
The example is elementary in construction but sharp enough to show that, without local boundedness, even very simple tensor-product phenomena may behave in a qualitatively different way, already in the simplest possible setting: tensor products of two one-dimensional Hilbert spaces.


Before that, we recall the construction of an auxiliary function that will be needed for the construction of our example.

\begin{remark}\label{rem:additive}
	Here we recall the construction given in \cite{hamel1905basis} of an additive function $f: \RR \to \RR$ that is not locally bounded. Let $B$ be a Hamel basis of $\RR$ as a vector space over the field $\mathbb Q$, and fix $b \in B$. Then define $f$ on $B$ as
	$$f(b) = 1, \qquad f(a) = 0, \quad a \in B \setminus\{b\},
	$$
	and extend $f$ by linearity to $\RR$. Then $f$ is additive by construction, and since $f^{-1}(\{q\})$ is dense in $\RR$ for every $q \in \mathbb Q$, we obtain that the graph of $f$ is dense in $\RR^2$. In particular, $f$ is not bounded, neither from above nor from below, in any open set.
\end{remark}

\begin{proposition}\label{prop:locboundednes}
	There exist groups of bounded operators $\mathcal T = (T(t))_{t\in \RR} \subset \linearOp(\CC)$ and $\mathcal S = (S(t))_{t\in \RR} \subset \linearOp(\CC)$ such that neither $\mathcal T$ nor $\mathcal S$ is locally bounded (hence they are not in $\mathcal{SQC}(\CC)$) and such that $\mathcal T \otimes \mathcal S$ is the identity group on $\CC \otimes \CC$.
\end{proposition}
\begin{proof}
	Let $f: \RR \to \RR$ be the additive, non-locally bounded function given in Remark~\ref{rem:additive}. Define
	$$T(t)z := e^{f(t)} z, \qquad t \in \RR, \, z \in \CC.
	$$
	Set $\mathcal T =(T(t))_{t\in \RR}$ and set also $\mathcal S = (S(t))_{t\in \RR}$ where $S(t) := T(-t)$ for each $t\in \RR$. Note that, since $f$ is additive, $\mathcal S$ and $\mathcal T$ are groups of bounded operators on $\CC$. It is clear that
	$$\|T(t)\| = e^{f(t)}, \qquad \|S(t)\| = e^{-f(t)}, \qquad t\in \RR,
	$$
	and therefore neither $\mathcal T$ nor $\mathcal S$ is locally bounded. 
	Moreover,
	\begin{align*}
		T(t) \otimes S(t) &= e^{f(t)} \otimes e^{-f(t)} = I, \qquad t \in \RR.
	\end{align*}
	Thus $\mathcal T$ and $\mathcal S$ satisfy the claimed properties, finishing the proof.
	
\end{proof}

\subsection{Non-sharpness of the similarity theorem}

The previous subsection gives sufficient and necessary conditions for an
incomplete tensor product semigroup to be similar to a contraction semigroup in
terms of the similarity constants of its factors.  There remains, however, a
genuine gap between the two directions.  After a suitable rescaling of the
factors, Theorem~\ref{thm:main_incomplete} gives, schematically,
\[
\prod_{n\in\mathbb N} \mathcal C( \mathcal T_n)<\infty
\quad\Longrightarrow\quad
\mathcal T\in \mathcal{SC}(H)
\quad\Longrightarrow\quad
\sup_{n\in\mathbb N} \, \mathcal C( \mathcal T_n)<\infty,
\]
where \(\mathcal T\) denotes the corresponding incomplete tensor product semigroup on
the incomplete tensor product space \(H\).

The examples below show that neither implication can be reversed.  The first
example proves that similarity of the infinite tensor product need not come from
a convergent product of the factor similarity constants.  The second shows that
uniform boundedness of the factor similarity constants, even together with
the well-definedness of $\mathcal T$, does not imply similarity
of the incomplete tensor product semigroup to a contraction semigroup.

Thus the product condition in the sufficient part of Theorem~\ref{thm:main_incomplete} is not a mere
technicality.  Even for two-dimensional factors, similarity at the level of the
individual factors is insufficient: one must also control the tensorized
Hilbertian geometries.  This is why the quantitative hypotheses in the infinite
splitting theorem are structural rather than cosmetic.

Before the claimed examples, we collect in the following lemma an elementary but nontrivial fact on similarity constants of projections.

\begin{lemma}\label{lem:Psim}
	Let $H$ be a Hilbert space, and let $P \in \linearOp(H)$ be a projection. Then $P$ is similar to a contraction and
	\begin{equation*}
		\inf\{\|R\|\,\|R^{-1}\|:
		RPR^{-1}\text{ is a contraction}\}
		=	
		\|P\|+\sqrt{\|P\|^2-1}.
	\end{equation*}
\end{lemma}

\begin{proof}
	We first prove the following claim.
	
	\emph{Claim}. Fix $a>0$ and, for each $c>0$, set
	\[
	\|x\|_c^2:=\langle M_cx,x\rangle_{\CC^2}=x^*M_cx,
	\qquad x\in\CC^2,
	\]
	where
	$$M_{c}:= \begin{pmatrix}
		1 & a \\
		a & c
	\end{pmatrix}.
	$$ 
	Then $\|\cdot\|_c$ is a norm if and only if
	$M_c$ is positive definite, equivalently $c>a^2$, and the distortion constant $\mathcal C_{\CC^2}(\|\cdot\|_{c})$ attains its absolute minimum at $c_{a} := 1+2a^2$, with value
	\begin{equation}\label{eq:Msim}
		\mathcal C_{\CC^2}(\|\cdot\|_{c_a}) = a + \sqrt{1+a^2}.
	\end{equation}
	
	To prove the claim, note first that $M_c$ is positive definite if and only if $c>a^2$. For these values of $c$, the eigenvalues of $M_c$ are given by
	$$\lambda_\pm(c)
	=
	\frac{1+c\pm\sqrt{(c-1)^2+4a^2}}{2},
	$$
	and~\eqref{eq:Msim} follows by minimizing the right-hand side of the following identity
	$$\mathcal C_{\CC^2}(\|\cdot\|_{c}) = \sqrt{\frac{\lambda_+(c)}{\lambda_-(c)}}, \qquad c>a^2.
	$$
	\bigskip
	
	Now we continue with the proof of the lemma. Put $M=\operatorname{Ran}P$ and write $H=M\oplus M^\perp$.  Then $P$ has the
	block form
	\[
	P=\begin{pmatrix} I&A\\ 0&0\end{pmatrix}
	\]
	for some $A\in\mathcal L(M^\perp,M)$.  Hence
	\[
	\|P\|=(1+\|A\|^2)^{1/2}.
	\]
	Indeed, for $u\in M$ and $v\in M^\perp$ one has $P(u+v)=u+Av$, and optimizing
	over $u$ for fixed $v$ gives the displayed formula.
	
	First we obtain the upper bound for the similarity constant.  For $r>0$ set
	\[
	S_r := \begin{pmatrix}   I& A\\0& r\end{pmatrix}.
	\]
	Then
	\[
	S_rPS_r^{-1}=\begin{pmatrix} I&0\\0&0\end{pmatrix},
	\]
	so $S_rPS_r^{-1}$ is the orthogonal projection onto $M$.  Then $\|S_r\| \|S_r^{-1}\|$ 
	is computed on the two-dimensional reducing subspaces generated by pairs $u\in M$, $v\in M^\perp$ with $Av=au$, $a\geq0$, and is therefore reduced
	to the matrices
	\[
	U_r := \begin{pmatrix} 1& a\\0&r\end{pmatrix},
	\qquad 0\leq a\leq\|A\|.
	\]
	and hence to the distortion coefficient of the positive definite matrices
	$$U_r^\ast U_r = \begin{pmatrix}
		1 & a \\
		a & a^2 + r^2
	\end{pmatrix}, \qquad 0 \leq a \leq \|A\|.
	$$
	For each fixed $r>0$, the corresponding condition number is increasing as a function of $a\in[0,\|A\|]$. Hence the supremum over the
	two-dimensional reducing subspaces is obtained, up to approximation, at
	\(a=\|A\|\).
	Then the formula~\eqref{eq:Msim} yields that the optimal choice is $r=(1+\|A\|^2)^{1/2}$ and gives
	\[
	\|S_r\|\,\|S_r^{-1}\| = \sqrt{1+\|A\|^2} + \|A\| =\|P\|+\sqrt{\|P\|^2-1}.
	\]
	This proves the required upper bound.
	
	Conversely, suppose that $RPR^{-1}$ is a contraction.  A contractive projection
	on a Hilbert space is an orthogonal projection, so the Hilbertian norm
	$\|x\|_R:=\|Rx\|$ makes $\ran P$ orthogonal to $\ker P$.
	Choose unit vectors $v_k\in M^\perp$ such that $\|Av_k\|\to\|A\|$.  If
	$Av_k\neq0$, put $u_k=Av_k/\|Av_k\|$ and consider the two-dimensional invariant
	subspace
	\[
	L_k:=\operatorname{span}\{u_k,v_k\}, \qquad k \in \NN.
	\]
	The restriction of $P$ to $L_k$ has, in the orthonormal basis $(u_k,v_k)$, the
	matrix
	\[
	\begin{pmatrix}1&\|Av_k\|\\0&0\end{pmatrix}, \qquad k \in \NN.
	\]
	Hence, up to normalization, there exists $c_k> \|A v_k\|^2$ such that the restricted norm $\|\cdot\|_R|_{L_k}$ is represented by the positive definite matrix
	$$\begin{pmatrix}1&\|Av_k\|\\  \|A v_k\|& c_k \end{pmatrix}, \qquad k \in \NN.
	$$
	Since the distortion of $\|\cdot\|_R|_{L_k}$ is not larger than
	$\|R\|\,\|R^{-1}\|$,  then the two-dimensional Gram-matrix computation~\eqref{eq:Msim} yields
	\[
	\|R\|\,\|R^{-1}\|
	\geq
	(1+\|Av_k\|^2)^{1/2}+\|Av_k\|, \qquad k \in \NN.
	\]
	Letting $k\to\infty$ gives
	\[
	\|R\|\,\|R^{-1}\|
	\geq
	\sqrt{1+\|A\|^2}+\|A\|
	=
	\|P\|+\sqrt{\|P\|^2-1}.
	\]
	This proves the lower bound and hence the projection estimate.
\end{proof}


\begin{theorem}
	There exist a sequence of Hilbert spaces $(H_n)_{n\in\NN}$, a \(C_0\)-sequence \(\boldsymbol y=(y_n)_{n\in\NN}\in\times_{n\in\NN}H_n\), and for each $n\in\NN$, $C_0$-semigroups $\mathcal T_n = (T_n(t))_{t\geq0}$ and $\mathcal S_n = (S_n(t))_{t\geq0}$ on $H_n$, such that the following hold.
	\begin{enumerate}
		\item [(i)] There is no summable sequence $(d_n)_{n\in\NN} \in\ell^1$ with
		$\sum_{n\in\NN} d_n = 0
		$
		such that
		$$\prod_{n\in\NN} \mathcal C (e_{d_n} \mathcal T_n) < \infty.
		$$
		Yet, $\motimes_{n\in\NN}^{\boldsymbol y} \mathcal T_n$ is a well-defined nonzero semigroup satisfying
		$$\motimes_{n\in\NN}^{\boldsymbol y} \mathcal T_n \in \mathcal{SC} \left(\motimes_{n \in \NN}^{\boldsymbol{y}} H_n\right).
		$$
		\item[(ii)] The semigroup $\motimes_{n \in \NN}^{\boldsymbol{y}} \mathcal S_n$ is well-defined,
		 nonzero, and
		$$\motimes_{n\in\NN}^{\boldsymbol y} \mathcal S_n \notin \mathcal{SC} \left(\motimes_{n \in \NN}^{\boldsymbol{y}} H_n\right).
		$$
		Yet, for every $n\in \NN$, $\mathcal S_n \in \mathcal{SC}(H_n)$ and
		$$\lim_{n\to \infty} \mathcal C(\mathcal S_n) = 1,
		$$
		so, in particular, $\sup_{n\in\NN} \mathcal C(\mathcal S_n) < \infty$.
	\end{enumerate}
\end{theorem}
\begin{proof}
	We first collect some definitions and properties that will be used in the proofs of both (i) and (ii). Put
	\[
	H_n=\CC^2,
	\qquad y_n=e_0 \in \CC^2,\qquad n\in \NN,
	\]
	and define
	\[
	H:=\motimes_{n\in\NN}^{\boldsymbol{y}}H_n.
	\]
	Fix positive numbers $\alpha,\varepsilon>0$, and define
	\[
	Q_\alpha:=
	\begin{pmatrix}
		1&\alpha\\
		0&0
	\end{pmatrix},
	\qquad
	V_{\alpha,\varepsilon}(t) :=e^{-\varepsilon t} I  + (1-e^{-\varepsilon t}) Q_\alpha,
	\qquad t\geq0.
	\]
	Equivalently,
	\[
	V_{\alpha,\varepsilon}(t)=
	\begin{pmatrix}
		1&\alpha (1- e^{-\varepsilon t})\\
		0& e^{-\varepsilon t}
	\end{pmatrix}, \qquad t\geq0.
	\]
	Then \(\mathcal V_{\alpha,\varepsilon} = (V_{\alpha,\varepsilon}(t))_{t\geq0}\) is a uniformly continuous \(C_0\)-semigroup, \(V_{\alpha,\varepsilon}(0)=I\), and
	\(V_{\alpha,\varepsilon}(t)e_0=e_0\) for all \(t\geq0\). Moreover, using the definition of $V_{\alpha,\varepsilon}(t)$ and observing that 
	$$\lim_{t\to \infty} V_{\alpha,\varepsilon}(t) = Q_\alpha,
	$$
	it follows that an equivalent Hilbertian norm $\|\cdot\|_{\rm{eq}}$ in $\CC^2$ makes $\mathcal V_{\alpha,\varepsilon}$ contractive if and only if it makes the projection $Q_\alpha$ contractive. Using Lemma~\ref{lem:Psim} and the identity $\|Q_\alpha\|^2 = 1+\alpha^2$, we obtain that
	\begin{equation}\label{eq:simValpha}
		\mathcal C (\mathcal V_{\alpha,\varepsilon}) = \|Q_\alpha\| + \sqrt{ \|Q_\alpha\|^2 -1} = \alpha + \sqrt{1+\alpha^2}.
	\end{equation} \bigskip

		(i) Let \(\boldsymbol{\alpha} := (\alpha_n)_{n\in\NN}\) be a real non-negative sequence such that
		\[
		\boldsymbol{\alpha} \in\ell^2\setminus\ell^1
		\]
		and, in addition,
		\[
		\|\boldsymbol{\alpha}\|_2
		\prod_{n\in\NN}(1+\alpha_n^2)^{1/2}<1 .
		\]
		Then, set $P_n$ and $\mathcal T_n = (T_n(t))_{t\geq0}$ as
		$$P_n:= Q_{\alpha_n}, \qquad T_n(t) := V_{\alpha_n,1}(t) =
		\begin{pmatrix}
			1&\alpha_n (1- e^{- t})\\
			0& e^{- t}
		\end{pmatrix}, \qquad t\geq0, \, n\in\NN.
		$$
		We first check that their incomplete tensor product is well-defined and nonzero. The inequality
		\[
		\|T_n(t)\|
		\leq
		\e^{-t}\|I\|+(1-\e^{-t})\|P_n\|
		\leq
		\|P_n\|
		=
		(1+\alpha_n^2)^{1/2},
		\qquad t\geq0,
		\]
		yields that
		\[
		\prod_{n\in\NN} \maxone{\|T_n(t)\|}
		\leq
		\prod_{n\in\NN}(1+\alpha_n^2)^{1/2}<\infty,
		\qquad t\geq0,
		\]
		see Lemma~\ref{lem:vn1}(i). Moreover, $\langle T_n(t)e_0,e_0\rangle=1$ for all $n\in \NN$ and $t\geq 0$. Using Proposition~\ref{strongContInfProp}, we obtain that $\mathcal T = (T(t))_{t\geq0}$, where
		\[
		T(t):=\motimes_{n\in\NN}^{\boldsymbol{y}}T_n(t), \qquad t \geq 0,
		\]
		is a well-defined nonzero \(C_0\)-semigroup on $H$. We set $\mathcal T = (T(t))_{t\geq0}$. Similarly, using Theorem~\ref{incompleteOperatorTh} we have that
		$$P:= \motimes_{n\in \NN}^{\boldsymbol{y}} P_n \in \linearOp(H)
		$$
		is a well-defined nonzero projection with
		$$\|P\| = \prod_{n\in \NN} \|P_n\| = \prod_{n\in \NN} \sqrt{1+\alpha_n^2} =: M.
		$$
		Moreover, we have that,
		\begin{align}\label{eq:PTcommute}
			P = \lim_{t\to \infty} T(t),  
		\end{align}
		in the strong operator topology. Indeed, the above formula holds when acting on finitely $\boldsymbol{y}$-elementary tensors and, using the uniform boundedness of $(T(t))_{t\geq0}$, it extends to $H$ by linearity and density.

		Since $\omega(\mathcal T_n) = 0$ for all $n\in \NN$, it follows that $e_d \mathcal T_n \in \mathcal{SC}(\CC^2)$ implies that $d\leq 0$. Hence, the null sequence is the only balancing sequence $(d_n)_{n\in\NN} \in \ell^1$ satisfying
		$$\sum_{n\in\NN} d_n = 0
		$$
		and $e_{d_n} \mathcal T_n \in \mathcal{SC}(\CC^2)$ for all $n\in\NN$. As $\boldsymbol{\alpha} \notin \ell^1$, the identity~\eqref{eq:simValpha} yields that
		$$\prod_{n\in\NN} \mathcal C (\mathcal T_n) = \prod_{n\in\NN} \left(\alpha_n  + \sqrt{1 +\alpha_n^2} \right) \geq \prod_{n\in\NN} (1+\alpha_n) = \infty.
		$$
		
		Thus, we finish the proof of (i) if we show that $\mathcal T \in \mathcal{SC}(H)$.
		 We
		begin with a concrete model of the incomplete tensor product.  Recall that
		\({\finiteN}\) denotes the net of all finite subsets of \(\NN\).  For
		\(F\in{\finiteN}\), set
		\[
		e_F:=\motimes_{n\in\NN}x_n^{(F)},
		\qquad
		x_n^{(F)}=
		\begin{cases}
			e_1, & n\in F,\\
			e_0, & n\notin F .
		\end{cases}
		\]
		As shown in \cite[Theorem V]{von1939infinite}, these vectors form an orthonormal basis of $H$.
		Indeed, $H$ is the closed linear span of finitely $\boldsymbol{y}$-elementary tensors, and these tensors are linear combinations of the vectors
		\(e_F\).  Moreover,
		\[
		\langle e_F,e_G\rangle
		=\prod_{n\in F\cup G}\langle x_n^{(F)},x_n^{(G)}\rangle
		=\delta_{F,G} .
		\]
		Thus
		\[
		H\simeq \ell^2({\finiteN})
		=
		\left\{
		(c_F)_{F\in{\finiteN}}:
		\sum_{F\in{\finiteN}}|c_F|^2<\infty
		\right\}
		\]
		with norm
		\[
		\left\|\sum_{F\in{\finiteN}}c_F e_F\right\|^2
		=
		\sum_{F\in{\finiteN}}|c_F|^2 .
		\]
		It is more convenient for our purposes to use a square-free polynomial
		notation, which is equivalent to the above one.  For each $F \in \finiteN$, we identify \(e_F\) with the monomial
		\[
		z_F:=\prod_{n\in F}z_n,
		\qquad z_\varnothing=1 .
		\]
		Thus a vector of \(H\) is written formally as
		\[
		f=\sum_{F\in{\finiteN}}c_F z_F,
		\qquad
		\|f\|^2=\sum_{F\in{\finiteN}}|c_F|^2 .
		\]
		The finite-coordinate core is
		\[
		{\mathcal P}:=\operatorname{span}\{z_F:F\in{\finiteN}\},
		\]
		where $z_{\varnothing} := \mathbf{1}$, i.e., $z_{\varnothing}$ is the constant function with value $1$. Thus \({\mathcal P}\) is the space of finite linear combinations of square-free
		monomials, or equivalently the space of polynomial functions depending on only
		finitely many of the variables \(z_n\), with each variable occurring with degree
		at most one, and its closure under the above norm is \(H\). 
		
		Put
		\[
		\alpha_F:=\prod_{n\in F}\alpha_n,
		\qquad \alpha_\varnothing:=1 .
		\]
		Since \(\boldsymbol{\alpha}\in\ell^2\), then Lemma~\ref{lem:vn1} implies that
		\[
		\sum_{F\in{\finiteN}}\alpha_F^2
		=
		\prod_{n\in\NN}(1+\alpha_n^2)
		<\infty .
		\]
		Hence evaluation at \( \boldsymbol{\alpha}\), defined on $\mathcal P$ as
		$$f(\boldsymbol{\alpha}) = \sum_{F \in \finiteN} c_F \alpha_F, \qquad f = \sum_{F \in \finiteN} c_F z_F \in \mathcal P, 
		$$
		extends to a bounded functional on \(H\).
		Furthermore, the projection $P$ acts on the orthonormal basis $(z_F)_{F\in \finiteN}$ as
		\[
		P z_F = \left(\motimes_{n\in\NN}^{\boldsymbol{y}}P_n\right)z_F
		=
		\left(\prod_{n\in F}\alpha_n\right) \mathbf{1} = \alpha_F \mathbf{1}, \qquad F\in{\finiteN}.
		\]
		Thus, by linearity and density, \(Pf=f(\boldsymbol{\alpha}) \mathbf 1\) first for
		\(f\in\mathcal P\), and then for all \(f\in H\), where \(f(\boldsymbol{\alpha})\) denotes
		the continuous extension of the evaluation functional.
		Consequently 
		$$\ker P=\{f\in H:f(\boldsymbol{\alpha})=0\}.
		$$
		Similarly, on \({\mathcal P}\), the semigroup $\mathcal T$ has the composition form
		\begin{equation}\label{eq:Tcomp}
		T(t)f =f(\e^{-t}\boldsymbol{z}+(1-\e^{-t})\boldsymbol{\alpha}), \qquad f \in \mathcal P, \, t\geq 0,
		\end{equation}
		where $\boldsymbol{z}$ stands for $(z_n)_{n\in \NN}$ as free variables. Indeed, for a monomial \(z_F\) with $F \in \finiteN$,
		\[
		\begin{aligned}
			T(t)z_F
			&=
			\prod_{n\in F}
			\bigl(\e^{-t}z_n+(1-\e^{-t})\alpha_n\bigr)  
			=
			z_F(\e^{-t} \boldsymbol{z}+(1-\e^{-t})\boldsymbol{\alpha}), \qquad t \geq 0,
		\end{aligned}
		\]
		which proves the formula by linearity on \({\mathcal P}\).

		We next compute the generator of $\mathcal T$ on the finite-coordinate core $\mathcal P$. Note that $\mathcal P$ is contained in the domain of the generator since, using the finite-tensor product embedding~\eqref{eq:finite_tensors}, it follows that $\mathcal T$ acts on each fixed element of $\mathcal P$ as the tensor product of a finite number of $(\mathcal T_n)_{n\in \NN}$. Differentiating~\eqref{eq:Tcomp}, at \(t=0\), we obtain
		\[
		\frac{d}{dt}T(t)f\bigg|_{t=0}
		=
		\sum_{n\in\NN}(\alpha_n-z_n)\partial_n f, \qquad f \in \mathcal P.
		\]
		Thus the generator acts on $\mathcal P$ as \(-(D - B_{\boldsymbol{\alpha}})\), where \(D\) is the degree operator
		\[
		Dz_F =|F|z_F, \qquad F \in \finiteN,
		\]
		and \(B_{\boldsymbol{\alpha}}\) is the coordinate-deletion operator
		\[
		B_{\boldsymbol{\alpha}} z_F
		=
		\sum_{n\in F}\alpha_n z_{F\setminus\{n\}}, \qquad F \in \finiteN.
		\]
		with the convention $B_{\boldsymbol{\alpha}} \mathbf{1} = 0$.
		
		Since $D$ is a positive operator on $\mathcal P$ with $\ker D = \CC \langle \mathbf{1}\rangle$, we can consider its inverse square root $D^{-1/2}$ on $\mathcal P$ with the convention \(D^{-1/2} \mathbf{1} := 0\).
		Then, the elementary estimate
		\begin{equation}\label{eq:BalphaDestimate}
		\|B_{\boldsymbol{\alpha}} D^{-1/2} f\|\leq \|\boldsymbol{\alpha}\|_2 \|f\|, \qquad f \in \mathcal P,
		\end{equation}
		is proved as follows. If
		\[
		f_m=\sum_{|F|=m}c_Fz_F
		\]
		belongs to the homogeneous subspace of $\mathcal P$ of degree \(m\geq1\), then $D^{-1/2} f_m = m^{-1/2} f_m$. Thus,
		\[
		B_{\boldsymbol{\alpha}} D^{-1/2}f_m
		=
		\frac1{\sqrt m}\sum_{|G|=m-1}
		\left(\sum_{n\notin G}\alpha_n c_{G\cup\{n\}}\right)z_G,
		\]
		which is a homogeneous polynomial of degree $m-1$. By Cauchy's inequality,
		\[
		\begin{aligned}
			\|B_{\boldsymbol{\alpha}} D^{-1/2}f_m\|^2
			&\leq
			\frac1m \sum_{|G|=m-1}
			\left(\sum_{n\notin G}\alpha_n^2\right)
			\left(\sum_{n\notin G}|c_{G\cup\{n\}}|^2\right)  
			\leq
			\frac{\|\boldsymbol{\alpha}\|_2^2}{m}
			\sum_{|G|=m-1}\sum_{n\notin G}|c_{G\cup\{n\}}|^2 .
		\end{aligned}
		\]
		Each set \(F\) with \(|F|=m\) occurs exactly \(m\) times in the last double
		sum.  Hence
		\[
		\|B_{\boldsymbol{\alpha}} D^{-1/2}f_m\|^2
		\leq
		\|\boldsymbol{\alpha}\|_2^2\|f_m\|^2.
		\]
		Since the homogeneous subspaces are mutually orthogonal, the same estimate
		holds on \({\mathcal P}\), which proves~\eqref{eq:BalphaDestimate}.
		
		We will now prove that $D - B_{\boldsymbol{\alpha}}$ is an accretive operator on $\mathcal P \cap \ker P$. To see this, fix \(f\in{\mathcal P}\cap\ker P\) and write
		\[
		f=c_\varnothing\mathbf 1+g,
		\qquad g\perp\mathbf 1.
		\]
		Since \(\ran P=\mathbb C \mathbf 1\), the condition \(g\perp1\) means
		that \(g\in(\ran P)^\perp\), where the orthogonal complement is
		taken with respect to the original Hilbert norm.
		Since \(Pf = f(\boldsymbol{\alpha}) \mathbf{1} =0\), one has
		\[
		c_\varnothing=-g(\boldsymbol{\alpha}),
		\]
		or equivalently, $Pg = -c_{\varnothing} \mathbf{1}$. As $P$ is a projection, the norm of its restriction to $(\ran P)^\perp$ is equal to $(M^2-1)^{1/2}$.
		Therefore, $|c_\varnothing|^2 \leq (M^2-1)\|g\|^2$, and consequently
		\[
		\|f\|^2=|c_\varnothing|^2+\|g\|^2
		\leq
		M^2\|g\|^2.
		\]
		Since \(g\) is orthogonal to the constants and \(D\ge I\) on the orthogonal
		complement of the constants, we have
		\[
		\|g\|^2\le \langle Dg,g\rangle=\langle Df,f\rangle .
		\]
		Together with \(\|f\|^2\le M^2\|g\|^2\), this gives
		\[
		\|f\|\le M\langle Df,f\rangle^{1/2}.
		\]
		Combining this with \(\|B_{\boldsymbol{\alpha}} D^{-1/2}\|\le \|{\boldsymbol{\alpha}}\|_2\) and
		\(B_{\boldsymbol{\alpha}} \mathbf 1=0\), we obtain
		\[
		\begin{aligned}
			|\langle B_{\boldsymbol{\alpha}} f,f\rangle|
			&=
			|\langle B_{\boldsymbol{\alpha}} D^{-1/2}D^{1/2}f,f\rangle|  
			\le
			\|{\boldsymbol{\alpha}}\|_2\langle Df,f\rangle^{1/2}\|f\|      
			\le
			\|{\boldsymbol{\alpha}}\|_2 M \langle Df,f\rangle .
		\end{aligned}
		\]
		Consequently,
		\[
		\operatorname{Re}\langle (D-B_{\boldsymbol{\alpha}})f,f\rangle
		\ge
		(1-\|{\boldsymbol{\alpha}}\|_2M)\langle Df,f\rangle
		\ge 0 .
		\]
		Since this holds for every \(f\in\mathcal P\cap\ker P\), the required
		accretivity estimate follows.

		Next, note that \({\mathcal P}\cap\ker P\) is dense in \(\ker P\).  Indeed, if
		\(h\in\ker P\) and \(p_k\in{\mathcal P}\) with \(p_k\to h\), then
		\[
		{\mathcal P}\cap\ker P \ni  p_k-Pp_k\to h-Ph=h .
		\]
		Because of~\eqref{eq:PTcommute}, $P$ and $\mathcal T$ commute, and then space \(\ker P\) is invariant under \(\mathcal T\). 
		For \(f\in{\mathcal P}\cap\ker P\), the invariance of this core under \(T(t)\)
		and the accretivity estimate give
		\[
		\frac{d}{dt}\|T(t)f\|^2
		=
		-2\operatorname{Re}
		\langle (D - B_{\boldsymbol{\alpha}}) T(t)f,T(t)f\rangle
		\leq0 .
		\]
		Thus \(T(t)\) is contractive on \({\mathcal P}\cap\ker P\), and by density and
		boundedness it is contractive on all of \(\ker P\).
		
		Finally define an equivalent Hilbertian norm on \(H\) by
		\[
		\|h\|_*^2
		=
		\|Ph\|^2+\|(I-P)h\|^2, \qquad h \in H .
		\]
		In this norm the decomposition $H=\operatorname{Ran}P\oplus\ker P$
		is orthogonal.  Since \(T(t)\) is the identity on \(\operatorname{Ran}P\) and is
		contractive on \(\ker P\), we have
		\[
		\|T(t)h\|_*\leq \|h\|_*,
		\qquad t\geq0,\, h\in H.
		\]
		Therefore
		\[
		\mathcal T\in \mathcal{SC}(H),
		\]
		completing the proof of (i).
		\bigskip

		(ii) Let $(\varepsilon_n)_{n\in\NN}$ and $(\alpha_n)_{n\in\NN}$ be sequences of positive real numbers such that
		$$(\varepsilon_n)_{n\in\NN} \in \ell^1, \qquad (\alpha_n)_{n\in\NN} \notin \ell^2, \qquad \lim_{n\to \infty} \alpha_n = 0.
		$$
		Then, define $\mathcal S_n =(S_n(t))_{t\geq0}$ by
		$$S_n(t) := V_{\alpha_n, \varepsilon_n}(t) = \begin{pmatrix}
			1& \alpha_n(1- e^{-\varepsilon_n t})\\
			0& e^{-\varepsilon_n t}
		\end{pmatrix}, \qquad n\in \NN, \, t \geq 0,
		$$
		and set $P_n := Q_{\alpha_n}$. Using~\eqref{eq:simValpha}, we obtain that each $\mathcal S_n$ is similar to a contraction on $\CC^2$ and that
		$$\mathcal C (\mathcal S_n) = \alpha_n + \sqrt{1+\alpha_n^2}, \qquad n\in \NN.
		$$
		Thus, $\lim_{n\to \infty} \mathcal C (\mathcal S_n) =1$. Moreover, decomposing $S_n(t)$ into its diagonal and its off-diagonal part, it follows that for every $T>0$ there exists $C_T>0$ such that
		$$\|S_n(t)\| \leq 1 + C_T \varepsilon_n \alpha_n, \qquad n\in \NN, \, t \in [0,T].
		$$
		Since $(\varepsilon_n)_{n\in\NN} \in \ell^1$ and $(\alpha_n)_{n\in\NN}$ is bounded, the above implies
		$$\prod_{n\in\NN} \maxone{\|S_n(t)\|} < \infty, \qquad t\geq 0. 
		$$
		Then, we obtain from Proposition~\ref{strongContInfProp} that $\mathcal S = (S(t))_{t\geq0}$, given by
		$$S(t) := \motimes_{n\in\NN}^{\boldsymbol{y}} S_n(t), \qquad t \geq 0,
		$$
		is a $C_0$-semigroup on $H$.
		
		We claim that $\mathcal S \notin \mathcal{SC}(H)$. Indeed, assume by contradiction that $\mathcal S \in \mathcal{SC}(H)$. For any $N\in \NN$, we regard the finite tensor product $\otimes_{n=1}^N H_n$ as a subspace of $H$ through the unitary isomorphism \eqref{eq:finite_tensors}. 
		Under this identification, the restriction of \(\mathcal S\) to
		\(\motimes_{n=1}^N H_n\) is unitarily equivalent to
		\(
		\motimes_{n=1}^N \mathcal S_n .
		\)
		This yields that $\otimes_{n=1}^N \mathcal S_n \in \mathcal{SC} (\otimes_{n=1}^N H_n)$ and
		\begin{equation}\label{eq:NfiniteSim}
			\mathcal C \left(\motimes_{n=1}^N \mathcal S_n\right) \leq \mathcal C \left( \mathcal S \right), \qquad N \in \NN.
		\end{equation}
		Now recall that, for all $n\in \NN$, we have $\lim_{t\to\infty} S_n(t) = P_n$. Hence, if for each $N\in \NN$ we define $E_N := \otimes_{n=1}^N P_n$, then
		\begin{equation*}
			E_N = \lim_{t\to\infty} \motimes_{n=1}^N S_n(t).
		\end{equation*}
		Consequently $E_N$ is similar to a contraction and its similarity constant $C_N$ satisfies
		\begin{equation}\label{eq:notSim2}
			C_N \leq \mathcal C \left(\motimes_{n=1}^N \mathcal S_n\right), \qquad N \in \NN.
		\end{equation}
		On the other hand, since $E_N$ is itself a projection, Lemma~\ref{lem:Psim} yields that
		\begin{equation*}
			C_N =  \|E_N\| + \sqrt{\|E_N\|^2-1} \geq \|E_N\| = \prod_{n=1}^N \|P_n\| =  \prod_{n=1}^N (1+\alpha_n^2)^{1/2} , \qquad n\in \NN.
		\end{equation*}
		Using Lemma~\ref{lem:vn1}(i), then the assumption $(\alpha_n)_{n\in\NN} \notin \ell^2$ implies that
		$$\lim_{N\to \infty} C_N = \infty.
		$$
		This contradicts the upper bounds given in~\eqref{eq:NfiniteSim} and~\eqref{eq:notSim2}. Thus $\mathcal S\notin \mathcal {SC}(H)$ as claimed, finishing the proof.


\end{proof}

\subsection{Similarity on complete infinite tensor products}\label{CompleteSubSect}

We then turn to the complete infinite tensor product. In contrast with the incomplete case, this
setting is much more rigid from the viewpoint of continuity and invariance, and the purpose of
this subsection is to clarify how that rigidity interacts with similarity.

In fact, as explained in Subsection~\ref{subsec:complete-regularity} (see Examples~\ref{ex:strong-comp-counter-1} and~\ref{ex:strong-comp-counter-2}), the complete tensor product semigroup may lose most of the regularity properties of each of its factors. Despite this fact, we prove below Theorem~\ref{split_infinite}, which is the counterpart of Theorem~\ref{thm:main_incomplete} in the setting of complete infinite tensor products. 
This is possible because the arguments used in the proof of Theorem~\ref{thm:main_incomplete} do not rely on any regularity properties of either the factors or their incomplete tensor product. 
Although the proof adapts to the setting of complete tensor products with only minor modifications, these modifications involve several technical aspects related to the distinction between incomplete and complete tensor products. Since the subtleties of these constructions are somewhat delicate, we have chosen to include the full proof of Theorem~\ref{thm:main_complete} for the convenience of the reader.

\begin{theorem}\label{thm:main_complete}\label{split_infinite}
 Let $(H_n)_{n\in\NN}$ be a family of Hilbert spaces, and for each $n\in\NN$ let $\mathcal T_n=(T_n(t))_{t\ge0}\subset\linearOp(H_n)$ be a semigroup locally bounded on $(0,\infty)$. Assume that $\motimes_{n\in\NN}\mathcal T_n\subset\linearOp\bigl(\motimes_{n\in\NN}H_n\bigr)$ is a nonzero semigroup. Then the following assertions hold.
 \begin{enumerate}
  \item[(i)] If there exists a sequence $(d_n)_{n\in\NN}\subset\RR$ such that
  \begin{equation}\label{eq:dnBis}
   \sum_{n\in\NN}|d_n|<\infty,
   \qquad
   \sum_{n\in\NN} d_n=0,
  \end{equation}
  $e_{d_n}\mathcal T_n\in\mathcal{SC}(H_n)$ for all $n\in\NN$, and
  \[
   \prod_{n\in\NN}\mathcal C(e_{d_n}\mathcal T_n)<\infty,
  \]
  then
  \[
   \motimes_{n\in\NN}\mathcal T_n
   \in
   \mathcal{SC}\Bigl(\motimes_{n\in\NN}H_n\Bigr).
  \]
  
  \item[(ii)] Conversely, if
  \[
   \motimes_{n\in\NN}\mathcal T_n
   \in
   \mathcal{SC}\Bigl(\motimes_{n\in\NN}H_n\Bigr),
  \]
  then there exists a sequence $(d_n)_{n\in\NN}\subset\RR$ satisfying \eqref{eq:dnBis} such that
  \[
   e_{d_n}\mathcal T_n\in\mathcal{SC}(H_n)\qquad\text{for all } n\in\NN,
  \]
  and
  \[
   \sup_{n\in\NN}\mathcal C(e_{d_n}\mathcal T_n)<\infty.
  \]
 \end{enumerate}
\end{theorem}

\begin{proof}
To prove (i), set $C_n:=\mathcal C(e_{d_n}\mathcal T_n)$ 
and for each $n\in\NN$, let $\|\cdot\|_n^{\sim}$ be an equivalent Hilbertian norm on $H_n$ such that
\[
\|h\| \le \|h\|_n^{\sim} \le C_n \|h\|,\qquad h\in H_n,
\]
and
\[
\|e^{d_n t}T_n(t)h\|_n^{\sim} \le \|h\|_n^{\sim},\qquad h\in H_n,\ t\ge0.
\]
Using Proposition~\ref{prop:equivalent-infinite-tensor}(i), we obtain that the corresponding complete tensor product norm on $\otimes_{n\in\NN} H_n$ is then an equivalent Hilbertian norm, with distortion controlled by $\prod_n C_n$, and by construction the semigroup
\(
\motimes_{n\in\NN} e_{d_n}\mathcal T_n
\)
is contractive for that tensorized norm. Since $\sum_{n\in\NN} d_n=0$, we have
\[
\motimes_{n\in\NN} e^{d_n t}T_n(t)
=
\motimes_{n\in\NN} T_n(t),\qquad t\ge0.
\]
Hence $\otimes_{n\in \NN} \mathcal T_n \in \mathcal{SC}\left(\otimes_{n\in \NN} H_n\right)$.


 We now prove (ii). As in the proof of Theorem~\ref{thm:main_incomplete}, we divide the argument into steps 
 
 \smallskip
 \noindent \textbf{Step 1. The negative-growth case.}
 First assume that $\omega \left(\otimes_{n\in\NN} \mathcal T_n\right) < 0$ and choose $d \in  \RR$ such that
  \[
   \omega \left(\otimes_{n\in\NN}  \mathcal T_n \right)  < d < 0.
  \]
  Recall that $\omega \left(\otimes_{n\in\NN} \mathcal T_n\right) = \sum_{n\in \NN} \omega(\mathcal T_n)$ in the extended unconditional sense (as defined in Section~\ref{sec:num-prod}), see Lemma~\ref{lem:omega_tensor}. Hence we may choose $N \in \NN$ such that
  \[
   \sum_{n=1}^N \omega(\mathcal T_n) < d.
  \]
  Since $\otimes_{n\in\NN} \mathcal T_n$ is nonzero, we may choose $\delta>0$ such that
  \[
   \otimes_{n\in\NN} T_n(\delta) \neq 0.
  \]
  By the norm formula for complete tensor products,
  \begin{equation}\label{eq:deltaNonzero}
   \prod_{n\in\NN} \|T_n(\delta)\| = \left\|\otimes_{n\in\NN} T_n(\delta)\right\| \in (0,\infty).
  \end{equation}
  Hence also $\prod_{n\in\NN} \maxone{\|T_n(\delta)\|} <\infty$, and enlarging $N$ if necessary we may moreover assume that
  \[
   \prod_{n>N} \maxone{\|T_n(\delta)\|} < e^{-d\delta}.
  \]
  Then, up to replacing $\mathcal T_N$ and $\mathcal T_{N+1}$ by $e_{-d} \mathcal T_N$ and $e_d \mathcal T_{N+1}$ respectively, we may assume $\sum_{n=1}^N \omega\left( \mathcal T_n\right) < 0$ and $\prod_{n>N} \|T_n(\delta)\| < 1$. 
  Thus, if we set
  \[
  b_n=
  \begin{cases}
  	\omega(T_n),& n=1,\ldots,N,\\[2mm]
  	\delta^{-1}\log\|T_n(\delta)\|,& n>N,
  \end{cases}
  \]
  then
  \(
  \sum_{n\in\mathbb N}b_n<0
  \)
  in the extended unconditional sense.  By Lemma~\ref{strictlynegativeLemma}, there exists
  \((a_n)_{n\in\mathbb N}\in\ell^1\) with
  \[
  \sum_{n\in\mathbb N}a_n=0,
  \qquad
  a_n+b_n<0\quad(n\in\mathbb N).
  \]
  Replacing \(\mathcal T_n\) by \(e_{a_n} \mathcal T_n\), we may therefore suppose that
  \[
  \omega(\mathcal T_n)<0\quad(n=1,\ldots,N),
  \qquad
  \|T_n(\delta)\|<1\quad(n>N).
  \]

%

 \smallskip
\noindent \textbf{Step 2. Complementary complete products.}
  Next, we pass from non-vanishing at the single time \(\delta\) to a
  small-time statement.  Since
  \(
  \motimes_{\ell\in\NN}T_\ell(\delta)\ne0,
  \)
  the semigroup law for the complete product implies
  \[
  \motimes_{\ell\in\NN}T_\ell(t)\ne0,
  \qquad 0<t\le\delta .
  \]
  Indeed, if the product vanished at some \(0<t\le\delta\), then writing
  \(\delta=mt+r\), with \(m\in\NN\) and \(0\le r<t\), would force
  \(\motimes_{\ell\in\NN}T_\ell(\delta)\) to vanish.  Hence, by
  Proposition~\ref{prop:nakagami},
  \[
  \prod_{\ell\in\NN}\|T_\ell(t)\|
  =
  \left\|\motimes_{\ell\in\NN}T_\ell(t)\right\|
  \in(0,\infty),
  \qquad 0<t\le\delta .
  \]
  In particular,
  \[
  \prod_{\ell\ne n}\|T_\ell(t)\|\in(0,\infty),
  \qquad n\in\NN,\quad 0<t\le\delta .
  \]
  Hence, by Proposition~\ref{prop:nakagami},
  \[
  \motimes_{\ell\ne n}T_\ell(t),
  \qquad n\in\NN,\quad 0<t\le\delta,
  \]
  are well-defined and nonzero.
  
  Set
  \begin{equation}\label{eq:average1}
  	f_\ell(t):=\|T_\ell(t)\|,
  	\qquad t>0,\ \ell\in\NN .
  \end{equation}
  For arbitrary \(t>0\), choose \(k\in\NN\) such that \(s:=t/k\le\delta\).  By
  submultiplicativity,
  \[
  \prod_{\ell\ne n}\|T_\ell(t)\|
  \le
  \left(\prod_{\ell\ne n}\|T_\ell(s)\|\right)^k
  <\infty .
  \]
  Thus Proposition~\ref{prop:nakagami} applies to the family
  \((T_\ell(t))_{\ell\ne n}\) for every \(t>0\), and shows that the complete
  tensor product operator
  \(
  \motimes_{\ell\ne n}T_\ell(t)
  \)
  is well-defined and bounded.  Moreover,
  \[
  \left\|\motimes_{\ell\ne n}T_\ell(t)\right\|
  =
  \prod_{\ell\ne n} f_\ell(t),
  \qquad n\in\NN,\quad t>0 .
  \]
  The family \(\{\motimes_{\ell\ne n}T_\ell(t):t\ge0\}\), with the identity
  operator at \(t=0\), is a semigroup, because for \(s,t\ge0\) the bounded
  operators
  \[
  \left(\motimes_{\ell\ne n}T_\ell(s)\right)
  \left(\motimes_{\ell\ne n}T_\ell(t)\right)
  \quad\text{and}\quad
  \motimes_{\ell\ne n}T_\ell(s+t)
  \]
  agree on the dense linear span of elementary tensor vectors.  
  Moreover, since $\otimes_{\ell \in \NN} \mathcal T_\ell \in \mathcal{SC}(\otimes_{\ell \in \NN} H_\ell)$, it follows that
  $$\limsup_{t\to 0^+} \prod_{\ell \in \NN} f_\ell (t) = \limsup_{t\to 0^+} \| \otimes_{\ell \in \NN} T_\ell(t)\| < \infty.
  $$
  Therefore, Lemma~\ref{lem:unif_bound}(ii) implies that there exists $M > 0$ such that, if 
  \begin{equation*}
   A_n :=  \left\{ t \in (0,\delta) \, : \, \left\|\otimes_{\ell \neq n}  T_\ell(t) \right\| \leq M\right\}, \qquad n \in \NN,
  \end{equation*}
  then $\nu_\delta(A_n) \geq 1/2$ for all $n\in \NN$. Note also that
  $$\inf_{n\in \NN} \left\| \otimes_{\ell \neq n} T_\ell(\delta) \right\| =\frac{\left\| \otimes_{\ell \in \NN} T_\ell(\delta) \right\|}{\sup_{n\in \NN} \|T_n(\delta)\|}  > 0.
  $$
  Using these facts, for each $n\in \NN$ choose $h _n \in \otimes_{\ell \neq n} H_\ell$ with $\|h_n\|=1$ and
  \[
   \left\|\left(\otimes_{\ell \neq n} T_\ell(\delta) \right) h_n \right\|
   \ge \frac12 \left\|\otimes_{\ell \neq n} T_\ell(\delta) \right\|.
  \]
  Then
  \[
   \inf_{n\in \NN} \left\|\left(\otimes_{\ell \neq n} T_\ell(\delta) \right) h_n \right\|
   \ge \frac12 \inf_{n\in \NN}\left\|\otimes_{\ell \neq n} T_\ell(\delta) \right\|>0.
  \]
  Set
  $$D_n := \mu_\delta \left[ \left\| \left( \otimes_{\ell \neq n} T_\ell(\cdot) \right) h_n\right\|^2 \right], \qquad n \in \NN.
  $$
  Then, for all $n\in \NN$ and all $t \in \delta - A_n$, we have that
  \begin{align*}
   \left\| \left( \otimes_{\ell \neq n} T_\ell(\delta) \right) h_n\right\|
   &\leq \left\| \otimes_{\ell \neq n} T_\ell(\delta - t) \right\|
        \left\| \left(\otimes_{\ell \neq n} T_\ell(t) \right) h_n\right\| 
   \leq M \left\| \left( \otimes_{\ell \neq n} T_\ell(t) \right) h_n\right\|.
  \end{align*}
  where we used that $\delta - t \in A_n$ if $t \in \delta - A_n$. As $\nu_\delta(\delta - A_n) = \nu_\delta (A_n) \geq 1/2$, it follows that
  \begin{equation*}
   \begin{aligned}
    D:= \inf_{n\in \NN} D_n & \geq \inf_{n \in \NN}  \left( \mu_\delta \left[\left\| \left( \otimes_{\ell \neq n} T_\ell(\cdot) \right) h_n\right\|^2  \chi_{\delta - A_n}(\cdot) \right] \right)
    \geq  \frac{\inf_{n\in \NN} \left\| \left( \otimes_{\ell \neq n} T_\ell(\delta) \right) h_n \right\|^2}{2M^2} > 0.
   \end{aligned}
  \end{equation*}

\smallskip
\noindent\textbf{Step 3. Construction of the factor norms.}
  Now, let $\|\cdot\|_{\rm{eq}}$ be an equivalent Hilbertian norm on $\otimes_{\ell \in \NN} H_\ell$ making $\otimes_{\ell \in \NN} \mathcal T_\ell$ a contractive semigroup and satisfying $\| \cdot\| \leq \|\cdot \|_{\rm{eq}} \leq \mathcal C \|\cdot\|$, where $\mathcal C := \mathcal C\left(\otimes_{\ell \in \NN} \mathcal T_\ell\right)$. 
  
  Recall that by the associative law of complete tensor products, one has $ \otimes_{\ell \in \NN} H_\ell \simeq H_n  \otimes \left( \otimes_{\ell \neq n} H_\ell\right)$ unitarily for every $n\in \NN$, see~\eqref{eq:assoc_law1}. Then, for each $n \in \NN$, define
  \begin{align}\label{eq:average2}
   \|f\|_n^2 := \frac{1}{D_n}\mu_\delta \left[ \left\|f \otimes \left( \otimes_{\ell \neq n} T_\ell(\cdot) \right) h_n \right\|_{\rm{eq}}^2 \right], \qquad f \in H_n.
  \end{align}
  We obtain that
  \begin{align*}
   \|f\|_n^2 & \leq \frac{\mathcal C^2}{D_n} \mu_\delta \left[ \left\|f \otimes \left( \otimes_{\ell \neq n} T_\ell(\cdot) \right) h_n \right\|^2 \right] \leq \mathcal C^2  \|f\|^2, \qquad f \in H_n, \, n \in \NN,
  \end{align*}
  and similarly we get that
  \begin{align*}
   \|f\|_n^2 & \geq \|f\|^2, \qquad f \in H_n, \, n \in \NN.
  \end{align*}
  Consequently, $\|\cdot\|_n$ is an equivalent Hilbertian norm on $H_n$ for every $n\in \NN$. Moreover, for all $s \in (0,\delta)$, $n \in \NN$ and $f \in H_n$, we have that

  \begin{equation}\label{eq:quasi_1Bis}
   \begin{aligned}
    &\mu_\delta \left[ \left\| T_n(s) f \otimes \left(\otimes_{\ell \neq n} T_\ell(\cdot)\right) h_n \right\|_{\rm{eq}}^2 \chi_{[s,\delta)} (\cdot) \right]
    \\=& \mu_\delta \left[ \left\| T_n(s) f \otimes \left(\otimes_{\ell \neq n} T_\ell(s+\cdot)\right) h_n \right\|_{\rm{eq}}^2 \chi_{[0,\delta-s)} (\cdot) \right]
    \\ \leq & \mu_\delta \left[ \left\| f \otimes \left(\otimes_{\ell \neq n} T_\ell(\cdot)\right) h_n \right\|_{\rm{eq}}^2 \right]
    = D_n \|f\|_n^2,
   \end{aligned}
  \end{equation}
  and
  \begin{equation}\label{eq:quasi_2Bis}
   \begin{aligned}
    &\mu_\delta \left[ \left\| T_n(s) f \otimes \left(\otimes_{\ell \neq n} T_\ell(\cdot)\right) h_n \right\|_{\rm{eq}}^2 \chi_{[0,s)} (\cdot) \right]
    \\\leq & \mu_\delta \left[ \left\| T_n(s - \cdot) f \otimes  h_n \right\|_{\rm{eq}}^2 \chi_{[0,s)} (\cdot) \right]
    \\ \leq &  K_{n,s}^2  \mathcal C^2  \frac{s}{\delta} \|f\|^2
    \leq  K_{n,s}^2 \mathcal C^2  \frac{s}{\delta} \|f\|_n^2,
   \end{aligned}
  \end{equation}
  where 
  \begin{equation}\label{eq:Kns_boundBis}
   K_{n,s}:= \sup_{t \in [0,s]} \|T_n(t)\| = \sup_{t\in [0,s]} f_n(t) < \infty,
  \end{equation}
  see Lemma~\ref{lem:unif_bound}(i). Since 
  $$\|T_n(s) f\|_n^2 := \frac{1}{D_n} \mu_\delta \left[ \left\| T_n(s) f \otimes \left(\otimes_{\ell \neq n} T_\ell(\cdot)\right) h_n \right\|_{\rm{eq}}^2 \right],
  $$
  we obtain from \eqref{eq:quasi_1Bis} and \eqref{eq:quasi_2Bis} that
  $$\|T_n(s)\|_n \leq \sqrt{1 + \frac{\mathcal C^2 K_{n,s}^2}{\delta D_n} s}  \leq \sqrt{1 + \frac{\mathcal C^2 K_{n,s}^2}{\delta D} s}, \qquad s \in [0,\delta], \, n \in \NN.
  $$
  Using Lemma~\ref{lem:unif_bound}(i), we obtain that
  $$\limsup_{s\to 0^+} K_{n,s} \leq K:= \limsup_{s \to 0^+} \left\| \otimes_{\ell \in \NN} T_\ell(s) \right\| < \infty.
  $$  
  Hence
  \[
   \limsup_{s\to 0^+}\frac{\|T_n(s)\|_n-1}{s}
   \le \frac{\mathcal C^2 K^2}{2\delta D},\qquad n\in\NN,
  \]
  and thus, using the submultiplicativity of the mappings $s\mapsto \|T_n(s)\|_n$,
  \begin{equation}\label{eq:quasiComplete}
   \|T_n(s)\|_n \leq \exp\left(\frac{\mathcal C^2 K^2}{2\delta D} \, s\right), \qquad s \geq 0, \, n \in \NN.
  \end{equation}
  
  Consequently, $\mathcal T_n \in \mathcal{SQC}(H_n)$ and 
  \begin{equation*}
   \mathcal C \left(e_{- \frac{\mathcal C^2 K^2}{2\delta D}}\mathcal T_n \right) \leq \mathcal C, \qquad n \in \NN.
  \end{equation*}

  \smallskip
  \noindent\textbf{Step 4. Application of the low-regularity theorem.}
  We are again in the setting of the one-semigroup similarity theorem from Section~\ref{sec:sim-reg-th}: the auxiliary norms supply the quasi-contractive control, while $\|T_n(\delta)\|<1$ for $n>N$ provides the required one-time contractive estimate.
  
  Since $\|h\|\leq \|h\|_n \leq \mathcal C\|h\|$, it follows from~\eqref{eq:quasiComplete} that
  \[
   \|T_n(t)\| \leq \mathcal C\exp\left(\frac{\mathcal C^2 K^2}{2\delta D}t\right), \qquad t\in[0,\maxone{\delta}],\ n>N,
  \]
  so the local bound required in Theorem~\ref{Th:eqNorm} is uniform in $n$. 
  Since \(\|T_n(\delta)\|<1\), the operator \(T_n(\delta)\) is already
  contractive in the original Hilbert norm.  Thus the one-time contractive
  input in Theorem~\ref{Th:eqNorm} may be taken with \(C_2=1\), which yields that
  $\mathcal T_n \in \mathcal{SC}(H_n)$ for all $n > N$ and
  \begin{equation*}\label{eq:Cn_boundBis}
   \sup_{n>N} \mathcal C(\mathcal T_n)< \infty.
  \end{equation*}
  For the finitely many remaining factors \(n=1,\ldots,N\), the inequality
  \(\omega(\mathcal T_n)<0\), together with the local boundedness near zero obtained
  above, gives a quasi-contractive equivalent norm. Choosing \(\tau>0\) so
  large that \(\|T_n(\tau)\|<1\), the operator \(T_n(\tau)\) is similar to a
  contraction.  Hence Theorem~\ref{Th:eqNorm} applies and yields \(\mathcal T_n\in\mathcal{SC}(H_n)\).

	Returning to the original family, the required balancing sequence is obtained
  by adding to \((a_n)_{n\in\NN}\) the finitely supported balancing used above to replace
  \(\mathcal T_N\) and \(\mathcal T_{N+1}\).  This does not affect absolute summability, and the
  sum of the resulting sequence is still zero.
  
  Hence, we have proven assertion (ii) when $\omega \left(\otimes_{n\in\NN} \mathcal T_n\right) < 0$. \medskip

  \smallskip
  \noindent\textbf{Step 5. The zero-growth case.} If $\omega \left(\otimes_{n\in\NN} \mathcal T_n\right) = 0$, set 
  $$d_n:=-\omega(\mathcal T_n), \qquad n \in \NN.
  $$
  Then Lemma~\ref{lem:omega_tensor} gives
  \[
  \sum_{n\in \NN} d_n = - \omega \left(\otimes_{n\in \NN} \mathcal T_n\right) = 0
  \]
  in the extended unconditional sense. Hence, the positive and negative parts of the series are both finite, and therefore \((d_n)_{n\in\mathbb N}\in\ell^1\). 
  Then, replacing $\mathcal T_n$ by $e_{d_n} \mathcal T_n$, we may assume that 
  $$\omega(\mathcal T_n)=0, \qquad n\in\NN.
  $$

%
  
  The goal is now the same as in the incomplete case: starting from one contractive Hilbertian norm on the whole tensor product, we construct a uniformly controlled family of equivalent Hilbertian norms on the factors and then recover factorwise contractivity.
  Set $K :=  \sup_{t\geq0} \left\| \otimes_{\ell \in \NN} T_\ell(t)\right\|$, and note that $K< \infty$ since $\otimes_{\ell \in \NN} \mathcal T_\ell$ is similar to a contractive semigroup. Since $\|T_\ell(t)\|\geq 1$ for all $l\in \NN$ and $t\geq 0$ (recall that $\omega(\mathcal T_\ell)=0$), we have that
  \begin{equation}\label{eq:KboundBis}
   1\leq \prod_{\ell \neq n} \left\|  T_\ell(t)\right\| \leq \prod_{\ell \in \NN} \left\|  T_\ell(t)\right\| = \left\| \otimes_{\ell \in \NN} T_\ell(t) \right\| \leq K, \qquad t \geq 0, \, n \in \NN.
  \end{equation}
  Arguing as in~\eqref{eq:ellSemigroup}, we deduce that for every $n\in \NN$,
  $$\motimes_{\ell \neq n} \mathcal T_\ell \subset \linearOp (\motimes_{\ell \neq n} H_\ell),
  $$
  is a well-defined, nonzero semigroup. Since $\left\| \otimes_{\ell \neq n} T_\ell(t) \right\| = \prod_{\ell \neq n} \|T_\ell(t)\|$, the inequalities~\eqref{eq:KboundBis} imply that 
  $$1 \leq \|\otimes_{\ell \neq n} T_\ell (t)\| \leq K, \qquad n \in \NN, \, t \geq 0.
  $$
  Therefore, for each $n,m\in \NN$ we may choose a unit vector $h_{n,m}\in \otimes_{\ell \neq n} H_\ell$ such that
  \[
   \|\otimes_{\ell \neq n} T_\ell (m) h_{n,m}\|\geq \frac12.
  \]
  For $t\in[0,m]$ we then have
  \[
   \frac12 \le \| (\otimes_{\ell \neq n} T_\ell (m)) h_{n,m}\| = \| (\otimes_{\ell \neq n} T_\ell (m-t)) (\otimes_{\ell \neq n} T_\ell (t)) h_{n,m}\| \le K\,\| (\otimes_{\ell \neq n} T_\ell (t)) h_{n,m}\|,
  \]
  and hence
  \begin{equation*}
   \left\| \left(\otimes_{\ell \neq n} T_\ell(t) \right) h_{n,m} \right\|
	\geq \frac{1}{2K}, \qquad t \in [0,m].
  \end{equation*}
  For each $n,m\in \NN$, set also
  \begin{align*}
   D_{n,m} &:= \mu_m \left[ \left\| \left(\otimes_{\ell \neq n} T_\ell(\cdot) \right) h_{n,m}  \right\|^2 \right],
  \end{align*}
  where we consider the canonical tensor product norm in $\otimes_{\ell \neq n} H_\ell$. Then, for all $n,m \in \NN$, we have that $D_{n,m} \geq 1/(4K^2)$ by our choice of $h_{n,m}$, and
  \begin{align*}
   D_{n,m}  \leq \sup_{t \in [0,m]} \left\| \otimes_{\ell \neq n} T_\ell(t) \right\|^2 \leq K^2, 
  \end{align*}
  see \eqref{eq:KboundBis}. Now let $\|\cdot\|_{\rm{eq}}$ be a Hilbertian norm on $\otimes_{\ell \in \NN} H_\ell$ such that $\|\otimes_{\ell \in \NN} \mathcal T_\ell(t)\|_{\rm{eq}}\leq 1$ for all $t\geq0$ and $\|\cdot\| \leq \|\cdot\|_{\rm{eq}} \leq \mathcal C \|\cdot\|$, where $\mathcal C := \mathcal C \left( \otimes_{\ell \in \NN} \mathcal T_\ell\right)$, and define
  \begin{equation*}
   \|f\|_{n,m}^2 := \frac{1}{D_{n,m}}\mu_m \left[ \left\| f \otimes \left(\otimes_{\ell \neq n} T_\ell(\cdot) \right) h_{n,m} \right\|_{\rm{eq}}^2 \right], \qquad f \in H_n,
  \end{equation*}
  where we used the unitary equivalence given by the associative law for complete tensor products
  $$\otimes_{\ell \in \NN} H_\ell \simeq H_n \otimes \left(\otimes_{\ell \neq n} H_\ell\right),
  $$ 
  see~\eqref{eq:assoc_law1}.

  Therefore, one has
  \begin{align*}
   \|f\|_{n,m}^2
   &\leq \frac{\mathcal C^2}{D_{n,m}}
      \mu_m \left[ \left\| f \otimes \left(\otimes_{\ell \neq n} T_\ell(\cdot) \right) h_{n,m} \right\|^2 \right] 
   \leq \mathcal C^2 \|f\|^2,
   \qquad n,m\in \NN,\ f \in H_n,
  \end{align*}
  and similarly, $\|f\|_{n,m}  \geq  \|f\|$ for all $n,m \in \NN$ and $f\in H_n$.
  Fix a shift-invariant mean $\operatorname{LIM}$ on $\ell^\infty(\NN)$ and define
  $$\|f\|_n^2 := \operatorname{LIM}_{m\to\infty} \|f\|_{n,m}^2, \qquad f \in H_n, \, n \in \NN.
  $$
  Let $\langle\cdot,\cdot\rangle_{n,m}$ denote the inner product inducing $\|\cdot\|_{n,m}$, and define
  $$\langle f,g\rangle_n := \operatorname{LIM}_{m\to\infty} \langle f,g\rangle_{n,m}, \qquad f,g\in H_n.$$
  Then $\langle\cdot,\cdot\rangle_n$ is a positive semidefinite sesquilinear form on $H_n$ and
  $\|f\|_n^2=\langle f,f\rangle_n$ for all $f\in H_n$.
  Since the two-sided estimates above imply $\|f\| \le \|f\|_n \le \mathcal C\|f\|$, the form is definite,
  and hence, for each $n\in \NN$, $\|\cdot\|_n$ is an equivalent Hilbertian norm on $H_n$ satisfying
  \begin{equation}\label{eq:sim_Const_HnBis}
   \| f \| \leq \|f\|_n \leq  \mathcal C \| f \|, \qquad f \in H_n,\, n \in \NN.
  \end{equation}
  Moreover, for every $n,m\in \NN$, $s\in [0,1]$ and $f\in H_n$, one has
  \begin{align*}
   \|T_n(s)f\|_{n,m}^2 &= \frac{1}{D_{n,m}} \mu_m \left[  \left\| T_n(s) f \otimes \left(\otimes_{\ell \neq n} T_\ell(\cdot)\right) h_{n,m} \right\|_{\rm{eq}}^2 \right]
  \end{align*}
  Note that, for all $n,m \in \NN$ and $f \in H_n$,
  \begin{equation}\label{eq:mean_bound_1Bis}
   \begin{aligned}
    &\mu_m \left[ \left\| T_n(s) f \otimes \left(\otimes_{\ell \neq n} T_\ell(\cdot) \right) h_{n,m} \right\|_{\rm{eq}}^2 \chi_{[s,m)}(\cdot) \right]
    \\ \leq & \mu_m \left[ \left\| f \otimes \left(\otimes_{\ell \neq n} T_\ell(\cdot) \right) h_{n,m} \right\|_{\rm{eq}}^2 \chi_{[0,m-s)}(\cdot) \right]
    \leq  D_{n,m} \|f\|_{n,m}^2,
   \end{aligned}
  \end{equation}
  where we used $\|\otimes_{\ell \in \NN} T(s)\|_{\rm{eq}} \leq 1$ and the identities from Section~\ref{subsec:periodic-means}, and
  \begin{equation}\label{eq:mean_bound_2Bis}
   \begin{aligned}
    &\mu_m \left[ \left\| T_n(s) f \otimes \left(\otimes_{\ell \neq n} T_\ell(\cdot) \right) h_{n,m} \right\|_{\rm{eq}}^2 \chi_{[0,s)}(\cdot) \right] 
    \\ \leq & \mathcal C^2 \mu_m \left[ \left\| T_n(s - \cdot) f \otimes  h_{n,m} \right\|^2 \chi_{[0,s)}(\cdot) \right]
    \\ \leq & \mathcal C^2 \sup_{t\geq 0} \left(\|T_n(t)\|^2 \right) \frac{s}{m} \|f\|^2
    \leq  \mathcal C^2  K^2 \frac{s}{m} \|f\|_{n,m}^2,
   \end{aligned}
  \end{equation}
  where we used that 
  $$\sup_{n\in \NN, \, t\geq 0} \|T_n(t)\| \leq \sup_{t\geq0} \prod_{\ell \in \NN} \left\|T_\ell(t)\right\| = K,
  $$ 
  since $\|T_\ell(t)\| \geq 1$ for all $l\in \NN$ and $t\geq0$ as $\omega(\mathcal T_\ell) = 0$.
  Then \eqref{eq:mean_bound_1Bis} and \eqref{eq:mean_bound_2Bis} imply that $\|T_n(s)\|_{n,m}^2 \leq 1 + \frac{\mathcal C^2 K^2 }{D_{n,m}} \frac{s}{m}$ for $n,m\in \NN$, $s\in [0,1]$. Hence, we conclude
  $$\|T_n(s)\|_n^2 \leq  \operatorname{LIM}_m \left[  1 + \frac{\mathcal C^2 K^2 }{D_{n,m}} \frac{s}{m} \right] = 1, \qquad n \in \NN, \, s\in [0,1],
  $$
  where we used that $D_{n,m}\geq 1/(4K^2)$ for all $n,m\in \NN$. Hence $\|T_n(s)\|_n\le 1$ for all $s\in[0,1]$. If $t\ge0$ and $m\in\NN$ is chosen so that $t/m\in[0,1]$, then $T_n(t)=T_n(t/m)^m$, and therefore $\|T_n(t)\|_n\le 1$. Consequently, $\mathcal T_n \in \mathcal{SC}(H_n)$ and, by \eqref{eq:sim_Const_HnBis}, 
  $$\mathcal C(\mathcal T_n) \leq \mathcal C, \qquad n\in \NN,
  $$
  finishing the proof.
\end{proof}


The analogue of Theorem~\ref{continuousInfiniteCor11} for quasi-contraction semigroups also holds
in the setting of complete tensor products.  Its statement and proof are
completely analogous, apart from the evident replacement of incomplete
tensor products by complete tensor products, so we omit the details.

\subsection{Similarity of infinite tensor products of operators}

In this section, we present a counterpart of our similarity results adapted to infinite tensor products of single operators. These may be of interest to a broader audience, and naturally extend the ones given in \cite[Section 7]{OlivaMazaTomilovTensorI} for finite tensor products to the infinite setting.

This setting is simpler than that of semigroups for several reasons. First, the analogues of the averaging functions considered in the proofs of Theorems~\ref{thm:main_incomplete} and~\ref{thm:main_complete} (see, for instance, \cite[Section 7]{OlivaMazaTomilovSimilarity}) are now discrete functions from $\ZZ_+$ to $[0,\infty)$. Hence, these functions are automatically measurable, and then the shift invariant means used in the proofs can be replaced by finite sums. Moreover, all arguments needed to prove local bounds near zero that also are uniform across all factors become trivial. 

Second, and most importantly, proving similarity to a contraction for a single operator does not require the splitting similarity result for small times and a single large time provided in Section~\ref{sec:sim-reg-th}. As a consequence, all the results from Section~\ref{sec:sim-reg-th} become unnecessary in this context. Even more, the proof in the case where the spectral radius of the operator tensor product is strictly less than one (the analogue of Theorem~\ref{thm:main_incomplete}(ii) when $\omega (\otimes_{n\in \NN}^{\boldsymbol y} \mathcal T_n )<0$) becomes almost immediate after a suitable rescaling of the operators using Lemmas~\ref{infiniteSpectralRadiusLemma} and~\ref{strictlynegativeLemma}.

In view of the observations above, we restrict ourselves to presenting a sample result for incomplete tensor products of operators, without including the proofs. 
We also omit the corresponding result for complete tensor products of operators, since it is analogous to the incomplete case presented below.

\begin{theorem}\label{thm:main_incomplete_op}
	Let $(H_n)_{n\in\NN}$ be a family of Hilbert spaces, let \(\boldsymbol y=(y_n)_{n\in\NN}\in\times_{n\in\NN}H_n\) be a
	\(C_0\)-sequence and, for each $n\in\NN$, let $A_n \in \linearOp(H_n)$. Assume that $\motimes_{n\in\NN}^{\boldsymbol y} A_n \in \linearOp\bigl(\motimes_{n\in\NN}^{\boldsymbol y}H_n\bigr) \setminus\{0\}$. Then the following assertions hold.
	\begin{enumerate}
		\item[(i)] If there exists a sequence $(\alpha_n)_{n\in\NN}\subset (0,\infty)$ such that
		\begin{equation}\label{eq:dn2}
			\prod_{n\in\NN}\alpha_n
			\quad\hbox{converges unconditionally to }1,
		\end{equation}
		$ \alpha_n A_n$ is similar to a contraction for all $n\in\NN$, and
		\[
		\prod_{n\in\NN}\mathcal C(\alpha_n A_n)<\infty,
		\]
		then $\otimes_{n\in\NN}^{\boldsymbol y} A_n$
		is similar to a contraction on $\motimes_{n\in\NN}^{\boldsymbol y}H_n$.
		
		\item[(ii)] Conversely, if $\otimes_{n\in\NN}^{\boldsymbol y} A_n$ is similar to a contraction on $\otimes_{n\in\NN}^{\boldsymbol y}H_n$, then there exists a sequence $(\alpha_n)_{n\in\NN}\subset (0,\infty)$ satisfying \eqref{eq:dn2} such that $\alpha_n A_n$ is similar to a contraction for all $n\in \NN$, and
		\[
		\sup_{n\in\NN} \,\mathcal C(\alpha_n A_n)<\infty.
		\]
	\end{enumerate}
\end{theorem}


\end{document}